\documentclass{amsart}
\usepackage{amsmath,amsthm,amsfonts,amssymb, color, hyperref, verbatim, appendix}

\usepackage[normalem]{ulem}
\usepackage{cancel}

\numberwithin{equation}{section}
\theoremstyle{plain}

\newcommand{\one}{{{\rm 1\mkern-1.5mu}\!{\rm I}}}

\newtheorem{theorem}{\sc Theorem}[section]
\newtheorem{definition}[theorem]{\sc Definition}
\newtheorem{lemma}[theorem]{\sc Lemma}
\newtheorem{proposition}[theorem]{\sc Proposition}
\newtheorem{corollary}[theorem]{\sc Corollary}

\theoremstyle{remark}

\newtheorem{remark}[theorem]{\sc Remark}
\newtheorem{example}[theorem]{\sc Example}

\newcommand{\be}{\begin{equation}}
\newcommand{\ee}{\end{equation}}
\newcommand{\nn}{\nonumber}
\providecommand{\abs}[1]{\vert#1\vert}
\providecommand{\norm}[1]{\Vert#1\Vert}

\def\empp{R}  
\def\E{\bE}
\def\P{\bP} 
\def\wlev{\omvec} 
\def\Ospace{\mathbf{\Omega}} 
\def\slabv{\mathbf{s}}  
\def\slabvs{\mathbf{\bar s}}  
\def\Smap{S}  
\def\shift{\theta}  
\def\shifta{\tau}  
\def\zvecd{\bar\zvec}  
\def\zproj{\Phi} 
\def\gmap{\gamma} 

\def\cA{\mathcal{A}}
\def\cC{\mathcal{C}}
\def\cD{\mathcal{D}}
\def\cF{\mathcal{F}}

\def\cL{\mathcal{L}}
\def\cM{\mathcal{M}}

\def\cP{\mathcal{P}}

\def\cR{\mathcal{R}}

\def\bE{\mathbb{E}}
\def\bN{\mathbb{N}}
\def\bP{\mathbb{P}}
\def\bQ{\mathbb{Q}}
\def\bR{\mathbb{R}}
\def\bZ{\mathbb{Z}}

\def\kS{\mathfrak{S}}

\def\zvec{\mathbf{z}}
\def\Zvec{\mathbf{Z}}

\def\w{\omega}
\def\omvec{\bar\omega}
\def\wvec{\bar\omega} 

\def\OSP{(\Omega, \kS, \P)} 
\def\llnxi{\xi^*}  

\def\esssup{\mathop{\mathrm{ess\,sup}}}

\def\e{\varepsilon}

\begin{document}

\title[Large deviations and entropy for RWDRE]{Averaged vs. quenched large deviations and entropy for\\ random walk in a dynamic random environment}

\author[F.~Rassoul-Agha]{Firas Rassoul-Agha}
\address{Firas Rassoul-Agha\\ Department of Mathematics\\ University of Utah\\ 155 South 1400 East\\ Salt Lake City, UT 84109\\ USA.}
\email{firas@math.utah.edu}
\urladdr{http://www.math.utah.edu/~firas}
\thanks{F.\ Rassoul-Agha was partially supported by National Science Foundation grant DMS-1407574 and by Simons Foundation grant 306576.}

\author[T.~Sepp\"al\"ainen]{Timo Sepp\"al\"ainen}
\address{Timo Sepp\"al\"ainen\\ Department of Mathematics\\ University of Wisconsin-Madison\\ 480 Lincoln Dr.\\   Madison, WI 53706\\ USA.}
\email{seppalai@math.wisc.edu}
\urladdr{http://www.math.wisc.edu/~seppalai}
\thanks{T.\ Sepp\"al\"ainen was partially supported by  National Science Foundation grants DMS-1306777 and DMS-1602486,  by Simons Foundation grant 338287, and  by the Wisconsin Alumni Research Foundation.} 

\author[A.~Yilmaz]{Atilla Yilmaz}
\address{Atilla Yilmaz\\ Department of Mathematics\\ Ko\c{c} University\\Rumelifeneri Yolu, Sar\i yer, Istanbul 34450, Turkey.}
\email{atillayilmaz@ku.edu.tr}
\urladdr{http://home.ku.edu.tr/~atillayilmaz}
\thanks{A.\ Yilmaz was partially supported by European Union FP7 Marie Curie Career Integration Grant no.\ 322078.}

\date{October 20, 2015.}
\date{\today}

\subjclass[2000]{60K37, 60F10, 82C41, 82C44.} 
\keywords{Random walk, dynamic random environment, large deviations, averaged, quenched, empirical process, Donsker-Varadhan relative entropy, specific relative entropy, Doob $h$-transform, nonstationary process.}

\maketitle

\begin{abstract}

We consider random walk with bounded jumps on a hypercubic lattice of arbitrary dimension in a dynamic  random environment. The environment is temporally independent and spatially translation invariant.   We study the rate functions of the level-3 averaged and quenched large deviation principles from the point of view of the particle.  In the averaged case the rate function is a specific relative entropy,  while in the quenched case it is a Donsker-Varadhan type relative entropy for  Markov processes.   We relate these entropies to each other and seek to identify the minimizers of the level-3 to level-1 contractions in both settings.   Motivation for this  work comes from variational descriptions of the quenched free energy of directed polymer models where the same Markov process  entropy appears.



\end{abstract}

\section{Introduction}

After surveying  the background of the present work, this introductory section describes the   random walk in a dynamic random environment (RWDRE) model and then some general notions such as large deviation principles and the point of view of the particle.  The section concludes with an overview of the rest of the paper.  

\subsection{Background}

This paper   studies an entropy function  for Markov processes that appears in random medium models.  We give here some background motivation.  
A much-studied model   is the random path in a random potential model, also called the 
polymer model.    The random environment  $\w$ comes from a probability space $\OSP$ with an ergodic  group action $\{T_x\}_{x\in\bZ^d}$.    The random path is a random walk $X_k$ on $\bZ^d$.   The   potential $V(\w, z)$ is   a function of  $\w$ and a step $z$ of the random walk.   A key quantity is the limiting {\it quenched  free energy}   \be\label{gV}
g(V)= \lim_{n\to\infty}   n^{-1}\log E_0\bigl[ e^{\sum_{k=0}^{n-1} V(T_{X_k}\w, \,X_{k+1}-X_k)}\bigr] 
\ee 
where $E_0$ is the expectation of the random walk and $\w$ is fixed.  
The limit exists for $\P$-almost every $\w$ under hypotheses on the moments of $V$ and the degree of mixing of $\P$. 

  The limit $g(V)$ can be calculated only in a handful of exactly solvable models that exist only in $1+1$ dimension.  More generally, properties of $g(V)$ have remained an insurmountable problem.  This question is the positive temperature version of the question of understanding limit shapes of stochastic growth models such as first- and last-passage percolation.  The latter question  has also remained insurmountable  since the origins of the subject over  50 years ago, except for a few exactly solvable models in $1+1$ dimension.    For surveys of models of   type \eqref{gV}, see \cite{ComShiYos04, Hol07}.  
 

Our article \cite{RasSepYil13} introduced   two 
  variational formulas for $g(V)$.   Let $p(z)$ be the jump kernel of the underlying random walk.     The first  formula 
 \be\label{gV:cc}
 g(V)=\inf_F\, \P\text{-}\esssup_\w \,\log\sum_z p(z) e^{V(\w,z)+F(\w,z)}  
 \ee 
expresses $g(V)$ as  an infimum over the $L^1(\bP)$ closure of gradients $F(\w,z) = f(T_z\w) - f(\w)$, which we called the space of   {\it cocycles}.   Since this formula is not the topic of the present paper, we  refer to  \cite{geor-rass-sepp-var, RasSepYil_preprint, RasSepYil13} for precise definitions. 
  
 The second formula gives $g(V)$ as the  dual of an entropy adapted to   the point of view of the particle: 
 \be\label{gV:H} 
g(V)=  \sup\big\{ E^\mu[V]- H(\mu): \,\mu_\Omega\ll\P, \,  E^\mu[ V^- ] < \infty\big\}.
 \ee
 The supremum is over probability measures $\mu$ on $ \Omega\times\{\text{steps}\}$ with a natural  invariance property and with a $\P$-absolutely continuous $\Omega$-marginal $\mu_\Omega$.   The entropy is given by 
 \be\label{gV:H2} 
H(\mu) =   \int_{\Omega} \sum_z \,\mu(d\w, z) \log\frac{\mu(z\,\vert\,\w)}{p(z)}  . 
 \ee
 Formula \eqref{gV:H} was proved in \cite{RasSepYil13}, and this formulation is Theorem 7.5 in \cite{geor-rass-sepp-var}. 
 
  Article \cite{geor-rass-sepp-var} extended these formulas from  positive to zero temperature, that is, to last-passage percolation models. 
 The  goal is to shed light  on $g(V)$ and limit shapes through the  variational formulas.    The relationship between formulas \eqref{gV:cc} and \eqref{gV:H} is well understood presently only for directed polymers in weak disorder (Examples 3.7 and 7.7 in \cite{geor-rass-sepp-var}) and  in periodic environments (Section 8 in \cite{geor-rass-sepp-var}).  
  
 
 Here is a brief overview of the current state of the study of these formulas.  
The cocycle variational formula \eqref{gV:cc}  has been studied in several subsequent papers while the entropy formula \eqref{gV:H} has received no serious attention before the present paper.  \cite{RasSepYil_preprint} shows that \eqref{gV:cc}  always has a minimizer and uses the  minimizer(s) to characterize weak and strong disorder  of directed polymers.   \cite{GeoRasSepYil15} proves the existence of Busemann functions for the exactly solvable 1+1 dimensional log-gamma polymer and shows that these provide minimizing cocycles for \eqref{gV:cc}  and also a limiting polymer measure for infinite paths.   \cite{geor-rass-sepp-lppgeo, geor-rass-sepp-lppbuse}  construct the minimizing  cocycles for the 2-dimensional corner growth model with general i.i.d.\ weights  and use these to investigate Busemann functions, geodesics and the competition interface.  These notions have  become  central in the field of random medium models over the last twenty years, beginning with the work of Newman in the early 1990s on the geodesics of first-passage percolation \cite{New95}.  
 
In the current paper we begin the study of the entropy \eqref{gV:H2}.  This  entropy is the level-2 projection of an entropy that appears in the rate function of a level-3 quenched large deviation principle (LDP) for  RWDRE.  
(See \eqref{isimdustu} and Theorem \ref{3qLDP} in Section \ref{sec:prevresults}.)     
We study  the entropy in this large deviations context. In particular, we consider its relation to the entropy that serves as the rate function for a level-3 averaged LDP.  

The {\it point-to-point} version of the quenched  free energy \eqref{gV}  is 
\be\label{gV:xi}
g(V,\xi)= \lim_{n\to\infty}   n^{-1}\log E_0\bigl[ e^{\sum_{k=0}^{n-1} V(T_{X_k}\w, \,X_{k+1}-X_k)}, \, X_n=[n\xi] \bigr] 
\ee 
defined for  $\xi$  in the convex hull of the support of the kernel $p(z)$, and where  $[n\xi]$ is a lattice point that approximates  $n\xi$ and is  reachable  from the origin in $n$ steps.   The entropy variational formula  now takes the form 
 \be\label{gV:Hxi} 
g(V,\xi)=  \sup\big\{ E^\mu[V]- H(\mu): \,\mu_\Omega\ll\P, \,  E^\mu[ V^- ] < \infty, \, E^\mu[Z_1]=\xi\big\}
 \ee
where $Z_1$ is the step variable under distribution $\mu$. 
 Formula \eqref{gV:Hxi} was proved in \cite{rass-sepp-p2p} for a directed walk in an i.i.d.\ environment and a local potential $V\in L^{d+\e}(\P)$ for $\e>0$. This formulation is Theorem 7.6 in \cite{geor-rass-sepp-var}.

 Minimizing entropy under a mean step condition $E^\mu[Z_1]=\xi$ as   in \eqref{gV:Hxi}  is also  done in the level-3 to level-1 contraction in large deviation theory.    For this reason the  main focus of the present paper is to study  these  contractions, both averaged and quenched.   The averaged contraction can be understood completely.  Then we seek to characterize when the averaged and quenched contractions lead to the same level-1 rate function and have  the same minimizers.  
 
The quenched rate function is hard to study.  It begins with an entropy of a familiar type.  But this entropy is corrected in a  singular manner  to account for the  environment distribution $\P$, and then regularized again to be lower semicontinuous.   The opaqueness of the  l.s.c.\ regularization makes it difficult to analyze examples.    
By simplifying the situation so that the environment varies only temporally we can  describe  fully also the quenched contraction.  We discover that the connection between the averaged and quenched rate functions can break down rather spectacularly.     This part of the paper illuminates earlier large deviation work by Comets \cite{come-89} and one of the authors  \cite{baxt-jain-sepp,  sepp-ptrf-93I} that appears in the equilibrium statistical mechanics of disordered Gibbs measures.  

The present  paper studies only random walk in a dynamic random environment while connections to   polymer models  are left for future work.  
Our results in Section \ref{sec:results} begin with  the level-3 averaged LDP from the point of view of the particle and the existence of the relevant limiting specific relative entropy.   After understanding the contraction from the  level-3 to level-1 averaged LDP we turn to study the quenched rate functions.

\subsection{The model}\label{sec:modmod}

Consider the $d$-dimensional hypercubic lattice $\bZ^d$ with an arbitrary $d\in\bN = \{1,2,3,\ldots\}$.
Fix a finite $\cR\subset\bZ^d$ with at least two elements and let
\be\label{pipi}
\cP = \{q:\cR\to [0,1]:\ \sum_{z\in\cR} q(z)=1\}
\ee denote the set of probability measures on $\cR$. Elements of $\Omega=\cP^{\bZ\times\bZ^d}$ are called space-time environments and they are of the form $\w=(\w_{i,x})_{(i,x)\in\bZ\times\bZ^d}$. 
Each $\omega\in\Omega$ defines a time-inhomogeneous discrete-time Markov chain $(X_i)_{i\ge0}$ on $\bZ^d$ for which $X_0 = 0$ and the transition probability from state $x$ to $y$ at time $i$ is
$$\pi_{i,i+1}(x,y\,|\,\omega) = \begin{cases}\omega_{i,x}(y-x)&\text{if $y-x\in\cR$,}\\0&\text{otherwise.}\end{cases}$$
If $\omega$ is randomly sampled from a probability distribution $\bP$ on $(\Omega,\kS)$ rather than being deterministic, then $(X_i)_{i\ge 0}$ is a random walk (RW) in a dynamic (or space-time) random environment, which we abbreviate as RWDRE. Here, $\kS$ is the Borel $\sigma$-algebra with respect to (w.r.t.) the product topology on $\Omega$.

RWDRE (started at the origin) induces a probability measure $P_0(d\w,d\zvec) = \bP(d\w)P_0^\w(d\zvec)$ on the space $\Ospace_\bN = \Omega\times\cR^\bN$ of environments and walks. Here, $\zvec = (z_i)_{i\ge1}\in\cR^\bN$ is a sequence of steps, and $P_0^\w$ is the quenched path measure defined by
$$P_0^\w(z_1,\ldots,z_n) = \prod_{i=0}^{n-1}\pi_{i,i+1}(x_i,x_{i+1}\,|\,\omega),\quad\text{$n\ge1$, $x_0 = 0$ and $x_{i+1} = x_i + z_{i+1}$}.$$
The marginal of $P_0$ on $\cR^\bN$ is called the averaged path measure and also denoted by $P_0$ whenever no confusion occurs. $\bE, E_0$ and $E_0^\w$ stand for expectation under $\bP, P_0$ and $P_0^\w$, respectively. In general, we will write $E^\mu[f]$ or $\langle f,\mu\rangle$ for the integral of a function $f$ against a probability measure $\mu$.

Denote the entire spatial environment at a given time $i\in\bZ$ by $\wlev_i=(\w_{i,x}:x\in\bZ^d)$. Let $(T^s_y)_{y\in\bZ^d}$ be the group of spatial translations, defined by $(T^s_y \wlev_i)_x = \w_{i,x+y}$ for $x,y\in\bZ^d$. Throughout the article, we will make the following underlying assumptions.
\begin{itemize}
\item {\it Temporal independence:} $(\wlev_i)_{i\in\bZ}$ are independent and identically distributed (i.i.d.) under $\bP$ with a common distribution $\bP_s$ on $\cP^{\bZ^d}$, i.e., $\bP = (\bP_s)^{\otimes\bZ}$. (The subscript of $\bP_s$ stands for ``spatial".)
\item {\it Spatial translation invariance:} $\bP_s$ is invariant under $(T^s_y)_{y\in\bZ^d}$.
\end{itemize}
These two conditions are of course satisfied when $(\w_{i,x})_{(i,x)\in\bZ\times\bZ^d}$ are i.i.d. However, restricting to that special case would not change the statements or the proofs in this paper. Moreover, it should be relatively straightforward to adapt our results to various discrete-time continuous-space models (such as RWDRE on $\bR^d$ considered in \cite{BolMinPel09, JosRas11}) where spatial independence is not applicable. 
Note in particular that we do not assume ergodicity under spatial translations.

The only condition we impose on the one-step range $\cR$ of the walk is $2\le |\cR|<\infty$. ($|\cR|$ is the number of elements in the set $\cR$. The case $|\cR| = 1$ is trivial.) We will assume without loss of generality that $\bP(\w_{0,0}(z)>0) > 0$ for every $z\in\cR$. (Otherwise, we can replace $\cR$ by $\{z\in\cR:\, \bP(\w_{0,0}(z)>0) > 0\}$.) Our quenched results will require various ellipticity conditions 
which we will indicate as needed in their statements. See also Remark \ref{remell}.

As the name suggests, RWDRE is a variant of the much-studied random walk in a random environment (RWRE) model (see \cite{Zei04} for a survey). In fact, $(i,X_i)_{i\ge0}$ can be viewed as a directed RWRE on $\bZ^{d+1}$ because its component in the direction of $(1,0,\ldots,0)$ is strictly increasing. This directedness simplifies certain aspects of the analysis of the model. Most notably, RWDRE under the averaged measure $P_0$ is a classical RW on $\bZ^d$ with transition probabilities $\hat q(z) = \bE[\w_{0,0}(z)]>0$. In particular, the strong law of large numbers (LLN) and Donsker's invariance principle (IP) hold for the averaged walk. Since  any $P_0$-almost sure statement holds $P_0^\w$-almost surely for $\bP$-a.e.\ $\w$, there is no need for  a separate strong LLN for the quenched walk. On the other hand, an averaged IP does not a priori imply a quenched one. Nevertheless, for the i.i.d.\ case, there is an IP under $P_0^\w$ for $\bP$-a.e.\ $\w$  \cite{RasSep05}. In stark contrast to these limit theorems, for (undirected) RWRE the validity of even the LLN is an open problem. See \cite{BerDreRam14} for the best sufficient condition in the literature.

\subsection{Large deviation principles, the point of view of the particle, and empirical measures}\label{IntroLDP}

Recall that a sequence $\left(Q_n\right)_{n\geq1}$ of Borel probability measures on a topological space $\mathbb{X}$ is said to satisfy a large deviation principle (LDP) with (exponential scale $n$ and) rate function $I:\mathbb{X}\to[0,\infty]$ if $I$ is lower semicontinuous, and for any measurable set $G$, $$-\inf_{x\in G^o}I(x)\leq\liminf_{n\to\infty}\frac{1}{n}\log Q_n(G)\leq\limsup_{n\to\infty}\frac{1}{n}\log Q_n(G)\leq-\inf_{x\in\overline{G}}I(x).$$   $G^o$ is the topological interior of $G$ and $\overline{G}$ its topological closure. See \cite{DemZei10, Hol00, RasSep15} for general background regarding large deviations.

In the context of RWDRE, the LDP for $(P_0(X_n/n\in\cdot\,))_{n\ge1}$ is nothing but Cram\'er's theorem for classical multidimensional RW  (see, e.g., \cite[Chapter 4]{RasSep15}), with rate function $I_{1,a}:\bR^d\to[0,\infty]$ given by
\be\label{level1a}
I_{1,a}(\xi) = \sup_{\rho\in\bR^d}\left\{\langle\rho,\xi\rangle - \log\phi_a(\rho)\right\} = (\log\phi_a)^*(\xi),
\ee
the convex conjugate of the logarithm of the moment generating function
\be
\phi_a(\rho) = \sum_{z\in\cR}\hat q(z)e^{\langle\rho,z\rangle},\label{MGF}
\ee
where $\langle\cdot,\cdot\rangle$ denotes inner product. This is an averaged LDP, hence the subscript $a$. (The other subscript of $I_{1,a}$ stands for level-1 which is explained two paragraphs below.) Establishing the analogous quenched LDP  for $(P_0^\w(X_n/n\in\cdot\,))_{n\ge1}$ and identifying the   rate function is more arduous. It involves considering certain empirical measures from the point of view (POV) of the particle which we introduce next.

Define space-time translations $(T_{j,y})_{(j,y)\in\bZ\times\bZ^d}$ on $\Omega$ by $(T_{j,y}\w)_{i,x} = \w_{i+j,x+y}$. Then, $(T_{i,X_i}\omega)_{i\ge0}$ is a discrete-time Markov chain taking values in $\Omega$, and its transition probability from state $\w$ to state $\w'$ is   given by
$$\bar\pi(\w'|\,\w) =\!\!\! \sum_{z\in\cR\boldsymbol: \, T_{1,z}\w = \w'}\!\!\!\pi_{0,1}(0,z\,|\,\omega).$$
Every limit theorem about this so-called {\it environment Markov chain} implies a corresponding limit theorem for the walk. This general and robust approach was first introduced in the context of interacting particle systems \cite{KipVar86} and was later  successfully adapted to RWRE (see for example \cite{Ras03, RasSep09, Yil09b}).

In light of the previous paragraph, the large deviation behavior of RWDRE can be analyzed via various statistics of either the walk itself or the environment Markov chain. Among these statistics, the empirical velocity $X_n/n$ is the coarsest one, and hence its large deviation analysis is referred to as level-1. Finer statistics are provided by the occupation measure
$$L_n = \frac1{n}\sum_{i=0}^{n-1}\delta_{T_{i,X_i}\w}$$ which records the environments seen from the POV of the particle. The pair-empirical measure
$$L_n^2 = \frac1{n}\sum_{i=0}^{n-1}\delta_{T_{i,X_i}\w,Z_{i+1}}$$ goes one step further by essentially keeping track of the pairs of consecutive environments that the particle sees. (In the Markov chain literature, the pair-empirical measure typically refers to $\frac1{n}\sum_{i=0}^{n-1}\delta_{T_{i,X_i}\w,T_{i+1,X_{i+1}}\w}$ which is measurable w.r.t.\ our choice of $L_n^2$.) Pairs can be replaced with $\ell$-tuples for any $\ell\ge2$ to define more detailed empirical measures. Large deviations of each of these empirical measures are called level-2. Finally, level-3 involves the so-called empirical process
\be\label{eq:empi}
L_n^\infty = \frac1{n}\sum_{i=0}^{n-1}\delta_{T_{i,X_i}\w,\,\theta^i\Zvec}.
\ee
Here and throughout, $\Zvec=(Z_i)_{i\ge1}$ denotes the sequence of steps $Z_i = X_i - X_{i-1}$ of the random path $(X_i)_{i\ge0}$, and $\theta$ is the forward shift on sequences, i.e., $(\theta\Zvec)_j = Z_{j+1}$ for every $j\in\bN$. Under the topology of weak convergence of measures, the empirical process contains precisely the same information as all of the empirical measures for $\ell$-tuples combined. Level-1,2,3 large deviations for Markov processes were established (under certain conditions) in a series of papers by Donsker and Varadhan \cite{DonVar75, DonVar76, DonVar83}. The level terminology was introduced later in \cite{Ell85}.

\subsection{Further notation for steps, environments and $\sigma$-algebras}\label{ss:furnot}

Throughout the paper, for any bi-infinite sequence $\zvecd = (\ldots,z_{-2}, z_{-1}, z_0, z_1, z_2, \ldots)\in\cR^\bZ$ of steps and any pair of indices $-\infty< i\le j<\infty$, we write
$$z_{i,j} = (z_i, z_{i+1},\ldots,z_{j}),\quad z_{i,\infty} = (z_i,z_{i+1},z_{i+2},\ldots)\quad\mbox{and}\quad z_{-\infty,j} = (\ldots, z_{j-2},z_{j-1},z_{j}).$$
We also use  $\zvec = z_{1,\infty}$ and  $\zvecd = z_{-\infty,\infty}$. Similarly, for any environment $\w = (\wlev_i)_{i\in\bZ}$ and any pair of indices $-\infty< k\le\ell<\infty$,
$$\wlev_{k,\ell} = (\wlev_k,\wlev_{k+1},\ldots,\wlev_{\ell}),\quad\wlev_{k,\infty} = (\wlev_k,\wlev_{k+1},\wlev_{k+2},\ldots)\quad\text{and}\quad\wlev_{-\infty,\ell} = (\ldots,\wlev_{\ell-2},\wlev_{\ell - 1},\wlev_{\ell}).$$
We use this notation to introduce the $\sigma$-algebras 
$$\cA_{k,\ell}^{i,j} = \sigma\{\wvec_{k,\ell-1},Z_{i+1,j}\}\qquad\text{and}\qquad\kS_{k,\ell} = \sigma\{\wlev_{k,\ell-1}\}$$
on appropriate spaces, for $-\infty\le i< j\le\infty$ and $-\infty\le k<\ell\le\infty$.   The reason for the indexing convention  is that the distribution of step $Z_{n+1}$ is part of environment $\wvec_n$.   Note also that  $\cA_{-\infty,\infty}^{0,\infty}$ and $\cA_{-\infty,\infty}^{-\infty,\infty}$ are the Borel $\sigma$-algebras (w.r.t.\ the product topology) on $\Ospace_\bN = \Omega\times\cR^\bN$ and $\Ospace_{\bZ} = \Omega\times\cR^\bZ$, respectively. For any $\sigma$-algebra $\cF$, the space of bounded and $\cF$-measurable functions is denoted by $b\cF$.

\subsection{Content and organization of the article}

Section \ref{sec:prevresults} reviews previous results on large deviations for   RWDRE. The new results are in Section \ref{sec:results}. The paper is organized so that the results of Section $3.n$ are proved in Section $3+n$. Section \ref{sec:results} concludes with remarks and open problems.   The following list summarizes the  results   (with  proofs  in the indicated sections):
\begin{itemize}
\item [(i)]   level-3 averaged LDP for the joint environment-path  Markov chain (Section \ref{sec:3aLDP});
\item [(ii)] analysis of  the averaged contraction from level-3 to level-1 (Section \ref{sec:mini});
\item [(iii)] alternative formula for the level-3 quenched rate function (Section \ref{sec:for_mod});
\item [(iv)]  
relationship of level-3 averaged and quenched rate functions
(Section \ref{sec:connect});
\item [(v)]  characterizations of  the equality of   level-1 averaged and quenched rate functions (Section \ref{sec:charac});
\item [(vi)]  minimizers of quenched contractions from level-3 to level-1 (Section \ref{sec:miniq});
\item [(vii)] spatially constant environments  (Section \ref{sec:constant}).
\end{itemize}

\section{Summary of previous results on large deviations}\label{sec:prevresults}

Recall from the Introduction that the level-1 averaged LDP, i.e., the LDP for $\left(P_0\left(\frac{X_n}{n}\in\cdot\,\right)\right)_{n\geq1}$, is simply the multidimensional Cram\'er theorem with the rate function $I_{1,a}$ given in (\ref{level1a}), whereas the statement and the proof of its quenched counterpart is relatively technical. In fact, it is more convenient to first present the level-3 quenched LDP for the environment Markov chain, and we will proceed in this order.

Let $S$ denote the temporal shift operator from the POV of the particle.  It acts on $\Ospace_\bN = \Omega\times\cR^\bN$  via $S(\w,\zvec)=(T_{1,z_1}\w, \shift\zvec)$, and on $\Ospace_\bZ = \Omega\times\cR^\bZ$ via  $S(\w,\zvecd)=(T_{1,z_1}\w, \shift\zvecd)$.
On  $\Ospace_\bZ$  $S$  is invertible.  We can write $S^k(\w,\zvecd)=(T_{k,x_k}\w, \shift^k\zvecd)$ for all $k\in\bZ$,  
with this convention:  bi-infinite paths $x_\centerdot$  through the origin and sequences    $\zvecd\in\cR^\bZ$ are bijectively associated to each other  by 
\be\label{xz} x_0=0, \quad x_k=-\sum_{i=k+1}^0 z_i\quad\text{and}\quad x_\ell=\sum_{i=1}^\ell z_i\quad\text{for}\quad k<0<\ell. \ee    
 
\begin{remark}\label{aboutSinv}
The empirical process $L_n^\infty$ defined in \eqref{eq:empi} satisfies $$\int f\,dL_n^\infty = \frac1{n}\sum_{i=0}^{n-1}f\circ S^i$$ for every $f\in b\cA_{-\infty,\infty}^{0,\infty}$. In particular, $$\left|\int (f\circ S)\,dL_n^\infty - \int f \,dL_n^\infty\right| \le\frac{2\|f\|_{\infty}}{n}.$$
Thus, $L_n^\infty$ is an asymptotically $S$-invariant element of $\cM_1(\Ospace_\bN)$ for $\bP$-a.e.\ $\w$ and every realization of $\Zvec\in\cR^\bN$.
\end{remark}

For any $S$-invariant $\mu\in\cM_1(\Ospace_\bN)$, let
\begin{itemize}
\item [(i)] $\bar\mu$ be the unique $S$-invariant extension of $\mu$ to $\Ospace_\bZ$,
\item [(ii)] $\bar\mu_-$ the restriction of $\bar\mu$ to $\cA_{-\infty,\infty}^{-\infty,0}$, and
\item [(iii)] $\pi_{0,1}^{\bar\mu}(0,z\,|\,\omega, z_{-\infty,0}) = \bar\mu(Z_1 = z\,|\,\cA_{-\infty,\infty}^{-\infty,0})(\w, z_{-\infty,0})$ for every $z\in\cR$.
\end{itemize}
Define $\bar\mu_-\times\pi$ and $\bar\mu_-\times\pi^{\bar\mu}$ on $\cA_{-\infty,\infty}^{-\infty,1}$ by
\begin{align*}
(\bar\mu_-\times\pi)(d\w, \, dz_{-\infty, 1}) &= \bar\mu_-(d\w, \, dz_{-\infty, 0})\pi_{0,1}(0,z_1\,|\,\w)c_\cR(z_1)\quad\text{and}\\
(\bar\mu_-\times\pi^{\bar\mu})(d\w, \, dz_{-\infty, 1}) &= \bar\mu_-(d\w, \, dz_{-\infty, 0})\pi_{0,1}^{\bar\mu}(0,z_1\,|\,\w,z_{-\infty,0})c_\cR(z_1),
\end{align*}
respectively. Here, $c_\cR = \sum_{z\in\cR}\delta_z$ is the counting measure on $\cR$. Note that $\bar\mu_-\times\pi^{\bar\mu}$ is simply the restriction of $\bar\mu$ to $\cA_{-\infty,\infty}^{-\infty,1}$. Let $H_q(\mu)$ denote the entropy of $\bar\mu_-\times\pi^{\bar\mu}$ relative to $\bar\mu_-\times\pi$ on $\cA_{-\infty,\infty}^{-\infty,1}$, i.e.,
\begin{align}
H_q(\mu)  &=  H_{\cA_{-\infty,\infty}^{-\infty, 1}}(\bar\mu_-\times\pi^{\bar\mu}\,\vert\,\bar\mu_-\times\pi)\nonumber\\
&= \int \bar\mu_-(d\w, \, dz_{-\infty, 0})\sum_{z\in\cR}\pi_{0,1}^{\bar\mu}(0,z\,|\,\omega,z_{-\infty,0})\log\left(\frac{\pi_{0,1}^{\bar\mu}(0,z\,|\,\omega,z_{-\infty,0})}{\pi_{0,1}(0,z\,|\,\omega)}\right).\label{isimdustu}
\end{align}
Projecting this entropy to $\cA_{-\infty,\infty}^{0,1}$ and replacing $\pi_{0,1}(0,z\,|\,\omega)$ with a constant jump kernel $p(z)$ gives the entropy   \eqref{gV:H2} discussed in the Introduction.  

The rate function of the level-3 quenched LDP is obtained via the following modification of $H_q$. For any $\mu\in\cM_1(\Ospace_\bN)$, denote its $\Omega$-marginal by $\mu_\Omega$, and set
\be\label{tanimuymaz}
H_{q,\bP}^{S}(\mu)=\begin{cases}  H_q(\mu) &\text{if $\mu$ is $S$-invariant and $\mu_\Omega\ll\P$,}\\ \infty  &\text{otherwise}.  \end{cases}
\ee
$H_{q,\bP}^{S}$ is convex but not lower semicontinuous, and the double convex conjugate $(H_{q,\bP}^{S})^{**}$ of $H_{q,\bP}^{S}$ gives its lower semicontinuous regularization (see \cite[Theorem 4.17]{RasSep15}).

\begin{theorem}[Level-3 quenched LDP]\label{3qLDP}
Assume 
\be\label{polyell}
\text{$\exists\,p>d+1$ such that $\bE[|\log\w_{0,0}(z)|^p] < \infty$ for every $z\in\cR$.}
\ee
Then, for $\mathbb{P}$-a.e.\ $\omega$, $\left(P_0^\w(L_n^\infty\in\cdot\,)\right)_{n\geq1}$ satisfies an LDP with rate function $I_{3,q}:\cM_1(\Ospace_\bN)\to[0,\infty]$ given by
$$I_{3,q}(\mu) = (H_{q,\bP}^{S})^{**}(\mu).$$
\end{theorem}

This result is a special case of the level-3 quenched LDP we  established in \cite{RasSepYil13} for a class of models including both directed and undirected RWRE with a rather general but technical condition on the environment measure. We show in Proposition \ref{prop:suffcon3q} in Appendix \ref{app:sufficient} that this technical condition holds in our current setting under the ellipticity assumption \eqref{polyell}.

Since the empirical velocity $$\frac{X_n}{n} = E^{L_n^\infty}[Z_1] = \int z_1L_n^\infty(d\w,d\zvec)$$ is a bounded and continuous function of the empirical process, the level-1 quenched LDP follows immediately from Theorem \ref{3qLDP} via the contraction principle (see, e.g., \cite[Chapter 3]{RasSep15}).

\begin{corollary}[Level-1 quenched LDP]\label{1qLDP}
Assume 
\eqref{polyell}. Then, for $\mathbb{P}$-a.e.\ $\omega$, $\left(P_0^\omega\left(\frac{X_n}{n}\in\cdot\,\right)\right)_{n\geq1}$ satisfies an LDP with rate function $I_{1,q}:\bR^d\to[0,\infty]$ given by
\begin{align}
I_{1,q}(\xi) &= \inf\{I_{3,q}(\mu):\,\mu\in\cM_1(\Ospace_\bN), E^\mu[Z_1] = \xi\}.\label{nasilda}
\end{align} 
\end{corollary}

After the appearance of \cite{RasSepYil13}, the level-1 quenched LDP was established in \cite{CDRRS13} using an alternative method involving the subadditive ergodic theorem, under the stronger assumption of
\be\label{unifell}
\text{uniform ellipticity: $\exists\,c>0$ such that $\bP(\w_{0,0}(z)\ge c) = 1$ for every $z\in\cR$.}
\ee
Originally developed in \cite{Var03} for undirected RWRE, this second method is less technical and it avoids empirical measures, but it does not provide   any formula for the rate function $I_{1,q}$.

Let $\cD := \text{conv}(\cR)$ denote the convex hull of $\cR$, and $\llnxi := \sum_{z\in\cR}\hat q(z)z$ stand for the LLN velocity of the walk. The following proposition lists some elementary facts regarding the level-1 averaged and quenched rate functions. We provide its proof in Appendix \ref{app:elementary} for the sake of completeness.

\begin{proposition}\label{prop:elem}
Assume 
\eqref{polyell}. Then, the following hold.
\begin{itemize}
\item [(a)] $I_{1,a}$ and $I_{1,q}$ are convex and continuous on $\cD$.
\item [(b)] $I_{1,a}(\xi) \le I_{1,q}(\xi) \le \max\{\bE[|\log\w_{0,0}(z)|]:\,z\in\cR\}<\infty$ for every $\xi\in\cD$.
\item [(c)] $I_{1,a}(\xi) = 0 $ iff $I_{1,q}(\xi) = 0$ iff $\xi = \llnxi$.
\item [(d)] $I_{1,a}(z) < I_{1,q}(z)$ for every $z\in\cR$ that is an extreme point of $\cD$ $\mathrm{(}$unless $\w_{0,0}(z)$ is deterministic$\mathrm{)}$.  
\end{itemize}
\end{proposition}

Under additional assumptions, the following further results have been obtained regarding the comparison of the level-1 averaged and quenched rate functions in relation with the spatial dimension.
\begin{theorem}\label{thm:preveze}
Assume that
\begin{align}
&\text{$(\w_{i,x})_{(i,x)\in\bZ\times\bZ^d}$ are i.i.d.,}\nonumber\\
&\text{the environment is uniformly elliptic $\mathrm{(}$see \eqref{unifell}$\mathrm{)}$, and}\label{specass}\\
&\text{the walk is nearest-neighbor, i.e., $\cR = U := \{\pm e_1,\ldots,\pm e_d\}$.}\nonumber
\end{align}
Then, the following hold at the indicated spatial dimensions.
\begin{itemize}
\item [(a)] $(d=1)$ $I_{1,a}(\xi) < I_{1,q}(\xi)$ for every $\xi\in\cD\setminus\{\llnxi\}$, see \cite[Theorem 1.5]{YilZei10}.
\item [(b)] $(d=2)$ $I_{1,a}(\xi) < I_{1,q}(\xi)$  for every $\xi\in\cD$ in a punctured neighborhood of $\llnxi$, see \cite[Theorem 1.6]{YilZei10}.
\item [(c)] $(d\ge3)$ $I_{1,a}(\xi) = I_{1,q}(\xi)$ for every $\xi\in\cD$ in a neighborhood of $\llnxi$, see \cite[Theorem 2]{Yil09a}.
\end{itemize}
\end{theorem}
Examining the proofs given in the references reveals that the last two conditions in \eqref{specass} can be replaced with somewhat weaker versions. However, the spatial independence of the environment is crucial to the proofs and  cannot be relaxed much.

There are other previous results on large deviations for RWDRE such as the ones in \cite{Yil09a} regarding the analysis of the averaged and quenched contractions from level-3 to level-1, but we prefer to mention them in later parts of this paper because they will be either covered by our new results or used in the proofs.

Our temporal independence assumption excludes various concrete models such as RW on particle systems. Level-1,2,3 quenched LDPs for such models (which satisfy uniform ellipticity \eqref{unifell}) are covered in \cite{RasSepYil13}, but averaged LDPs are open in general. See \cite{AveHolRed10} for level-1 averaged and quenched LDPs for RW on one-dimensional shift-invariant attractive spin-flip systems. Finally, for previous results on large deviations for RWRE and closely related models, see \cite[Section 2]{Yil11a}, \cite{Ras04} and \cite[Section 1.3]{RasSepYil13}, and the references therein.

\section{Results}\label{sec:results}

\subsection{Level-3 averaged LDP}

For any $\Smap$-invariant $\mu\in\cM_1(\Ospace_\bN)$, the specific relative entropy 
\be\label{def:hmuPzero}
h(\mu\,\vert\, P_0)
=\lim_{\ell\to\infty}
\frac1{\ell}H_{0,\ell}(\mu\,\vert\, P_0)
=\sup_{0<\ell<\infty} \frac1{\ell}H_{0,\ell}(\mu\,\vert\, 
P_0)
\ee
exists, where
\be\label{relent}
H_{k,\ell}(\mu\,\vert\, P_0) := H_{\cA_{k,\ell}^{k,\ell}}(\mu\,\vert\, P_0) = \sup_{f\in b\cA_{k,\ell}^{k,\ell}} \{E^\mu[f] - \log E_0[e^{f}]\}
\ee
is the entropy of $\mu$ relative to $P_0$ on $\cA_{k,\ell}^{k,\ell}$. The existence of the limit and the identity in \eqref{def:hmuPzero} follow from superadditivity and the independence built into $P_0$, and will be justified in Section \ref{sec:3aLDP}.

Our first result in this paper is the averaged counterpart of Theorem \ref{3qLDP}. Note that it requires only the temporal independence and spatial translation invariance conditions which we assume throughout the paper (see Section \ref{sec:modmod}).
\begin{theorem}[Level-3 averaged LDP]\label{3aLDP}
$(P_0(L_n^\infty\in\cdot\,))_{n\ge1}$ satisfies an LDP with rate function $$I_{3,a}: \cM_1(\Ospace_\bN)\to[0,\infty]$$ given by
\be\label{level3a}
I_{3,a}(\mu)=\begin{cases} h(\mu\,\vert\, P_0) 
&\text{if $\mu$ is $\Smap$-invariant,}\\
 \infty &\text{otherwise.}
\end{cases}
\ee
\end{theorem}

\begin{remark}\label{canmur}
The appearance of $S$-invariance in \eqref{tanimuymaz} and \eqref{level3a} is natural, as observed in Remark \ref{aboutSinv}. 

Every $S$-invariant $\mu\in\cM_1(\Ospace_\bN)$ arises in the following way. Consider $\Ospace_\bZ$ as the product space $(\cP^{\bZ^d}\!\times\cR)^\bZ$ with generic variable $(\w,\zvecd)=(\bar\w_i, z_{i+1})_{i\in\bZ}$ and temporal shift mapping  $(\tau(\w,\zvecd))_i= (\bar\w_{i+1}, z_{i+2})$.   
 Let $\nu$ be a $\tau$-invariant probability measure on  $\Ospace_\bZ$.  Recalling \eqref{xz},  let $\bar\mu\in\cM_1(\Ospace_\bZ)$ be the distribution of the sequence $(T^s_{-x_i}\bar\w_i, z_{i+1})_{i\in\bZ}$ under $\nu$, and finally  let $\mu$ be the marginal of $\bar\mu$ on $\Ospace_\bN$ obtained by dropping the nonpositive steps $z_{-\infty,0}$.  
 
$P_0$ is not $S$-invariant (on $\cA_{-\infty,\infty}^{0,\infty}$), but there is a unique $S$-invariant probability measure $P^\infty_0$ on $\Ospace_\bZ$ that agrees with $P_0$ on $\cA_{0,\infty}^{0,\infty}$ (see Lemmas \ref{Pinfty-lm2} and \ref{Pinfty-lm3}).
The LDP of Theorem \ref{3aLDP} is valid also for the distributions  $(P^\infty_0(L_n^\infty\in\cdot\,))_{n\ge1}$ and will in fact be proved first for these.  
\end{remark}

Similar to Corollary \ref{1qLDP}, the contraction principle gives the following (infinite-dimensional) variational formula for the level-1 averaged rate function:
\be\label{a_contraction}
I_{1,a}(\xi) = \inf\{I_{3,a}(\mu):\,\mu\in\cM_1(\Ospace_\bN), E^\mu[Z_1] = \xi\}.
\ee
Since (\ref{level1a}) is a much simpler formula than (\ref{a_contraction}), the significance of the latter lies not in providing a numerical value for $I_{1,a}(\xi)$, but in the questions it raises regarding the minimizer(s) of this variational formula, which we pursue next.

\subsection{Minimizer of the averaged contraction}\label{subsecmini}

Recall from (\ref{level1a}) that the level-1 averaged rate function $I_{1,a}$ is the convex conjugate of the logarithm of the moment generating function $\phi_a$ defined in (\ref{MGF}).   We have not assumed that $\cD$ has nonempty interior. Consequently  $I_{1,a}$ is not necessarily differentiable, and instead of its gradient we have to work with the set-valued subdifferential $\partial I_{1,a}(\xi)$.     Facts from convex analysis and some proofs of the claims below are  collected in   Appendix \ref{app:subdifferential}.  

Let  $\xi\in\mathrm{ri}(\cD)$, the relative interior of $\cD$.  By basic convex analysis,     every $\rho\in\partial I_{1,a}(\xi)$  maximizes in (\ref{level1a}), that is, 
$$I_{1,a}(\xi) = \langle\rho,\xi\rangle - \log\phi_a(\rho).$$          $I_{1,a}$ is differentiable at $\xi$ if and only if $\partial I_{1,a}(\xi)$ is a singleton if and only if $\mathrm{dim}(\cD) = d$. 
 In general    $\partial I_{1,a}(\xi)$   is a nonempty affine subset of $\bR^d$ parallel to the orthogonal complement of the affine hull of $\cR$.   
From this last point it follows that  any $\rho\in\partial I_{1,a}(\xi)$ can be used below  to define a measure $\mu^\xi\in\cM_1(\Ospace_\bN)$:   for $-\infty<k\le 0 <\ell<\infty$ and a test function $f\in b\cA_{k,\ell}^{0,\ell}$,
\be\label{sutesti}
\int f(\w,\zvec)\mu^\xi(d\w,d\zvec) := E_0\bigl[e^{\langle\rho, X_{\ell-k}\rangle - (\ell-k)\log\phi_a(\rho)}f\circ S^{-k}(\w,\Zvec)\bigr].
\ee
Proposition \ref{prop:elem2} in Section \ref{sec:mini}    provides basic properties of  $\mu^\xi$, beginning with its well-definedness. 

The  second  result in this paper identifies $\mu^\xi$ as the unique minimizer of the averaged contraction from level-3 to level-1.
\begin{theorem}\label{thm:uniquemin}
For every $\xi\in\mathrm{ri}(\cD)$, $\mu^\xi$ is the unique minimizer of the variational formula (\ref{a_contraction}) of the averaged contraction from level-3 to level-1.
\end{theorem}

Measure $\mu^\xi$ was introduced in \cite[Definition 1]{Yil09a}  with different notation and under the stronger assumptions in \eqref{specass}. Theorem \ref{thm:uniquemin} follows from an adaptation of \cite[Theorem 1]{Yil09a} which roughly says that, conditioned on $\{X_n/n \approx \xi\}$, the empirical process $L_n^\infty$ converges to $\mu^\xi$ under $P_0$. See Proposition \ref{averagedconditioning} for the precise statement.

Next we start analyzing the structure of the averaged contraction minimizer $\mu^\xi\in\cM_1(\Ospace_\bN)$. First of all, $\mu^\xi$ is $S$-invariant (see Proposition \ref{prop:elem2}(a)). Using the notation introduced in Section \ref{sec:prevresults}, let $\bar\mu^\xi$ be the unique $S$-invariant extension of $\mu^\xi$ to $\Ospace_{\bZ}$, and 
$$\pi_{0,1}^{\bar\mu^\xi}(0,z\,|\,\omega, z_{-\infty,0}) = \bar\mu^\xi(Z_1 = z\,|\,\cA_{-\infty,\infty}^{-\infty,0})(\w,z_{-\infty,0})$$
for $z\in\cR$.
\begin{proposition}\label{markova}  
For every $\xi\in\mathrm{ri}(\cD)$, $j\ge 0$, 
and $z\in\cR$, 
\be\label{yenker-1}\bar\mu^\xi( Z_{j+1}=z\,\vert\,\cA_{-\infty,\infty}^{-\infty,j})(\w,z_{-\infty,j}) = \mu^\xi(Z_1 = z\,|\,\kS_{0,\infty})(T_{j,x_j}\w).\ee
Hence, the quenched walk under $\bar\mu^\xi$ is Markovian, and its transition kernel
\be\label{yenker}
\pi_{0,1}^\xi(0,z\,|\,\omega) := \mu^\xi(Z_1 = z\,|\,\kS_{0,\infty})(\w) = \pi_{0,1}^{\bar\mu^\xi}(0,z\,|\,\omega, z_{-\infty,0})
\ee
is $\kS_{0,\infty}$-measurable.
\end{proposition}

We denote the $\Omega$-marginal of $\mu^\xi$ by $\mu_\Omega^\xi$.  The proof of Proposition \ref{markova}  in Section \ref{sec:mini}   shows   that,  by martingale convergence, the transition kernel in \eqref{yenker} is given by 
\be\label{yenker-3}
\pi_{0,1}^\xi(0,z\,|\,\omega) =\lim_{n\to\infty} \frac{ E_0^{\w}[e^{\langle\rho,X_{n}\rangle },Z_{1} = z] }{ E_0^{\w}[e^{\langle\rho,X_{n}\rangle }]},\qquad \text{$\mu^\xi_\Omega$-a.s.,}
\ee
for any  $\rho\in\partial I_{1,a}(\xi)$. The following result provides a characterization of the absolute continuity of $\mu_\Omega^\xi$ in terms of a structural representation of $\pi_{0,1}^\xi$ involving a Doob $h$-transform.

\begin{theorem}\label{thm:doobchar}
For every $\xi\in\mathrm{ri}(\cD)$, consider the following statements. 
\begin{itemize}
\item [(i)] There exists a function $u\in L^1(\Omega,\kS_{0,\infty},\bP)$ such that $\bP(u>0) = 1$ and
\be\label{doob}
\pi_{0,1}^\xi(0,z\,|\,\w) = \pi_{0,1}(0,z\,|\,\w)\frac{e^{\langle \rho,z\rangle}}{\phi_a(\rho)}\frac{u(T_{1,z}\w)}{u(\w)}
\ee
for every $\rho\in\partial I_{1,a}(\xi)$. 
\item [(ii)] $\mu_\Omega^\xi\ll\bP$ on $\kS_{0,\infty}$.
\end{itemize}
Then, $(i)\implies(ii)$. Conversely, if
\be\label{newell} 
\text{$\exists\, z'\in\cR$ such that }\bP(\w_{0,0}(z')>0) = 1,
\ee
then $(ii)\implies(i)$. Furthermore, whenever $(i)$ holds, $u$ is equal (up to a multiplicative constant) to $\left.\frac{d\mu_\Omega^\xi}{d\bP}\right|_{\kS_{0,\infty}}$.
\end{theorem}

The proof that (i) implies (ii) in Theorem \ref{thm:doobchar} is adapted from that of \cite[Lemma 4.1]{RasSepYil_preprint} which is concerned with disorder regimes of directed random walks in random potentials. The other implication follows from \eqref{yenker-3} under the mild ellipticity condition \eqref{newell} which ensures that $\mu_\Omega^\xi$ and $\bP$ are in fact mutually absolutely continuous on $\kS_{0,\infty}$. For closely related results on directed polymers and ballistic (undirected) RWRE, see \cite[Proposition 3.1]{ComYos06} and \cite[Theorem 3.3]{Yil11a}, respectively.

\begin{remark}\label{lazmols}
When we choose $\xi$ to be the LLN velocity $\llnxi= \sum_{z\in\cR}\hat q(z)z$,  we can take $\rho=0$ because  $\nabla\log\phi_a(0)=\llnxi$ which is equivalent to $0\in\partial I_{1,a}(\llnxi)$ (see \eqref{entemizig} in Appendix \ref{app:subdifferential}). 
Then  \eqref{yenker-3}  shows  that $\pi_{0,1}^{\llnxi}(0,z\,\vert\,\w)=\pi_{0,1}(0,z\,\vert\,\w)$, the original kernel, and in \eqref{doob} we can take $u\equiv 1$.
Thus  $\mu^{\llnxi}_\Omega=\P$ on $\kS_{0,\infty}$, which is also evident directly from the definition of $\mu^{\xi^*}$ in \eqref{sutesti}.

When $d\ge3$ and the conditions in \eqref{specass} hold, it was shown by one of the authors \cite[Theorem 4]{Yil09a} that statements (i)  and (ii) in Theorem \ref{thm:doobchar} are true not only at $\xi = \llnxi$ but also for $\xi$ sufficiently close to $\llnxi$, and in this case $u\in L^2(\Omega,\kS_{0,\infty},\bP)$.
\end{remark}

\subsection{Modified variational formulas for the quenched rate functions}

Recall from \eqref{tanimuymaz} that the formula given in Theorem \ref{3qLDP} for the level-3 quenched rate function $I_{3,q}$ involves absolute continuity w.r.t.\ $\bP$ (on $\kS$). This formula is valid for a general class of RWRE models. However, in the case of   RWDRE, as we have seen in Proposition \ref{markova} and Theorem \ref{thm:doobchar}, the relevant $\sigma$-algebra is $\kS_{0,\infty}$. Therefore, we next provide appropriately modified formulas for $I_{3,q}$ and $I_{1,q}$ which will be central to some of our subsequent results. Define  
\be\label{tanimmod}
H_{q,\bP}^{S,+}(\mu) = \begin{cases}  H_q(\mu) &\text{if $\mu$ is $S$-invariant and $\mu_\Omega\ll\P$ on $\kS_{0,\infty}$,}\\ \infty  &\text{otherwise}.  \end{cases}
\ee

\begin{theorem}\label{3qLDPmod}
Assume 
\eqref{polyell}. Then, for every $\mu\in\cM_1(\Ospace_\bN)$,
\be
I_{3,q}(\mu) = (H_{q,\bP}^{S,+})^{**}(\mu). \label{level3qmod}
\ee

\end{theorem}

\begin{corollary}\label{1qLDPmod}
Assume 
\eqref{polyell}. Then, for every $\xi\in\mathrm{ri}(\cD)$, 
\begin{align}
I_{1,q}(\xi) &= \inf\{(H_{q,\bP}^{S,+})^{**}(\mu):\,\mu\in\cM_1(\Ospace_\bN), E^\mu[Z_1] = \xi\}\label{qvarmod1}\\
&= \inf\{H_q(\mu):\,\mu\in\cM_1(\Ospace_\bN), E^\mu[Z_1] = \xi,\,\text{$\mu$ is $S$-invariant, $\mu_\Omega\ll\bP$ on $\kS_{0,\infty}$}\}.\label{qvarmod2}
\end{align}
\end{corollary}

\begin{example}
The need for Theorem \ref{3qLDPmod} is justified by the fact that $H_{q,\bP}^{S}(\mu) = H_{q,\bP}^{S,+}(\mu)$ does not hold in general.
The following counterexample is adapted from \cite{BolSzn02}. Assume \eqref{specass} and
the following extra condition on the law of the environment: $$\bP\left(\w_{0,0}(z) > \w_{0,0}(z')\ \text{for every $z'\in U\setminus\{z\}$}\right) = \frac1{2d}$$
for every $z\in U$. 
Consider a new transition kernel $\pi'$ defined by
$$\pi'_{0,1}(0,z\,|\,\w) = \begin{cases}1&\text{if}\ \w_{0,0}(z) > \w_{0,0}(z')\ \text{for every}\ z'\in U\setminus\{z\},\\0 & \text{otherwise}.
\end{cases}$$
For $\bP$-a.e.\ $\w$, the quenched walk under this new kernel is deterministic, the law of the environment Markov chain $(T_{i,X_i}\w)_{i\ge0}$ converges weakly to a $\pi'$-invariant probability measure $\bQ$ on $\Omega$ (see \cite[Proposition 1.4]{BolSzn02}), $\bQ = \bP$ on $\kS_{0,\infty}$, but $\bQ\perp\bP$ on $\kS$ (see \cite[Proposition 1.5]{BolSzn02}). Define an $S$-invariant $\mu\in\cM_1(\Ospace_\bN)$ by setting $\pi_{0,1}^{\bar{\mu}}(0,z\,|\,\w,z_{-\infty,0}) = \pi'_{0,1}(0,z\,|\,\w)$ and $\mu_\Omega = \bQ$. Then, $H_{q,\bP}^{S}(\mu) = \infty$, but \eqref{polyell} ensures that
$$H_{q,\bP}^{S,+}(\mu) = H_q(\mu) = \bE\left[\sum_{z\in U}\pi'_{0,1}(0,z\,|\,\w)\log\left(\frac{\pi'_{0,1}(0,z\,|\,\w)}{\pi_{0,1}(0,z\,|\,\w)}\right)\right] < \infty.$$
\end{example}

\subsection{Decomposing the level-3 averaged rate function}

The level-3 averaged  and quenched LDPs hold with rate functions $I_{3,a}$ and $I_{3,q}$ given in \eqref{level3a} and \eqref{level3qmod}, respectively. Note that $I_{3,a}(\mu) \le I_{3,q}(\mu)$ for every $\mu\in\cM_1(\Ospace_\bN)$. This follows from Jensen's inequality applied to the convex conjugates of the rate functions, and is shown in Corollary \ref{cor:zincir} for the sake of completeness.  How are these two rate functions  related beyond this basic inequality?   The following   theorem provides a partial answer.  Additional remarks follow in Section \ref{ss:addrem}.  

\begin{theorem}\label{entropy_connection}
For every $S$-invariant $\mu\in\cM_1(\Ospace_\bN)$,
\be \label{a=kS+q}  h(\mu\,|\,P_0) = h_{\kS_{0,\infty}}(\mu_\Omega\,|\,\mathbb{P}) + H_q(\mu), \ee
where
$$h_{\kS_{0,\infty}}(\mu_\Omega\,|\,\mathbb{P}) = \lim_{n\to\infty}\frac1{n}H_{\kS_{0,n}}(\mu_\Omega\,|\,\bP).$$
\end{theorem}

Theorem \ref{entropy_connection} is an application of the chain rule for relative entropy (see \cite[Lemma 4.4.7]{DeuStr89}). It does not require any ellipticity condition.
$H_{\kS_{0,n}}(\mu_\Omega\,|\,\bP)$ is the entropy of $\mu_\Omega$ relative to $\bP$ on $\kS_{0,n}$, and $h_{\kS_{0,\infty}}(\mu_\Omega\,|\,\mathbb{P})$ is the specific relative entropy 
whose existence is shown in the proof of Theorem \ref{entropy_connection}. 

\begin{corollary}\label{cor:zincir}
Assume 
\eqref{polyell}. Then, for every $S$-invariant $\mu\in\cM_1(\Ospace_\bN)$,
\begin{align*}
H_q(\mu) &\le I_{3,a}(\mu) = h(\mu\,|\,P_0) = h_{\kS_{0,\infty}}(\mu_\Omega\,|\,\mathbb{P}) + H_q(\mu)\\
& \le I_{3,q}(\mu) = (H_{q,\bP}^{S,+})^{**}(\mu) \le H_{q,\bP}^{S,+}(\mu).
\end{align*}
\end{corollary}

\subsection{Equality of the averaged and quenched rate functions}

Proposition \ref{prop:elem} and Theorem \ref{thm:preveze} summarized what is known about the equality of $I_{1,a}(\xi)$ and $I_{1,q}(\xi)$. The following result complements this picture by providing three characterizations of $I_{1,a}(\xi) = I_{1,q}(\xi)$, each of which involve $\mu^\xi\in\cM_1(\Ospace_\bN)$ (defined in \eqref{sutesti} for $\xi\in\mathrm{ri}(\cD)$) or its $\Omega$-marginal $\mu_\Omega^\xi$.

\begin{theorem}\label{thm:charac}
Assume 
\eqref{polyell}. For every $\xi\in\mathrm{ri}(\cD)$, consider the following statements.
\begin{itemize}
\item [(i)] $I_{1,a}(\xi) = I_{1,q}(\xi)$.
\item [(ii)] $I_{1,q}(\xi) = H_q(\mu^\xi)$.
\item [(iii)] $(H_{q,\bP}^{S,+})^{**}(\mu^\xi) = H_q(\mu^\xi)$.
\item [(iv)] $h_{\kS_{0,\infty}}(\mu_\Omega^\xi\,|\,\mathbb{P}) = 0$.
\end{itemize}
Then, $(i)\iff(ii)\iff(iii)\implies(iv)$. Moreover, if
\be\label{expell}
\text{$\exists\,\delta>0$ such that $\bE[\w_{0,0}(z)^{-\delta}] < \infty$ for every $z\in\cR$,}
\ee
then $(iv)\implies(i)$ and hence all four statements are equivalent.
\end{theorem}

\begin{remark}\label{remell}
The ellipticity conditions that appear in the statements of our results are related as follows:
$$\eqref{expell}\implies\eqref{polyell}\implies\eqref{newell}.$$
They are all strictly weaker than uniform ellipticity \eqref{unifell}.
\end{remark}

Regarding the equality of the level-3 averaged and quenched rate functions, the following result provides a sufficient condition. 
It is also noteworthy that under the stronger condition of uniform ellipticity,  the entropy $H_{\kS_{0,n}}(\mu_\Omega\,|\,\bP)$ can grow at most sublinearly for the absolutely continuous marginals of  $S$-invariant measures.  

\begin{corollary}\label{cor:ugurmugur}
Assume \eqref{polyell}. Then, for every $S$-invariant $\mu\in\cM_1(\Ospace_\bN)$ such that $\mu_\Omega\ll\bP$ on $\kS_{0,\infty}$,
$$I_{3,a}(\mu) = I_{3,q}(\mu) = H_q(\mu).$$
Furthermore, if we strengthen \eqref{polyell} to uniform ellipticity  \eqref{unifell}, then every $S$-invariant $\mu\in\cM_1(\Ospace_\bN)$ such that $\mu_\Omega\ll\bP$ on $\kS_{0,\infty}$ satisfies $h_{\kS_{0,\infty}}(\mu_\Omega\,|\,\mathbb{P}) = 0.$
\end{corollary}

\subsection{Minimizers of the quenched contractions}\label{subsecminijok}


Recall from Theorem \ref{thm:uniquemin} that, for every $\xi\in\mathrm{ri}(\cD)$, $\mu^\xi$ is the unique minimizer of the averaged contraction \eqref{a_contraction} from level-3 to level-1. Finding the minimizers of the quenched contractions \eqref{qvarmod1} and \eqref{qvarmod2} is more difficult in general. The following result treats the case where the level-1 rate functions are equal.

\begin{theorem}\label{thm:miniq}
Assume \eqref{polyell}. For every $\xi\in\mathrm{ri}(\cD)$:
\begin{itemize}
\item [(a)] if $I_{1,a}(\xi) = I_{1,q}(\xi)$, then
\be\label{kipcak}
I_{1,a}(\xi) = I_{1,q}(\xi) = (H_{q,\bP}^{S,+})^{**}(\mu^\xi) = H_q(\mu^\xi),
\ee
and $\mu^\xi$ is the unique minimizer of the quenched contraction \eqref{qvarmod1};
\item [(b)] if $\mu_\Omega^\xi\ll\bP$ on $\kS_{0,\infty}$, then \eqref{kipcak} holds, 
and $\mu^\xi$ is the unique minimizer of the quenched contractions \eqref{qvarmod1} and \eqref{qvarmod2}.
\end{itemize}
\end{theorem}

\begin{remark}\label{minvarmi}
	
Theorem \ref{thm:miniq}(b) is not vacuous or trivial (see Remark \ref{lazmols}). A similar result (regarding level-2 to level-1 contractions for $\xi$ sufficiently close to $\llnxi$) was previously obtained for certain ballistic (undirected) RWREs on $\bZ^d$ with $d\ge4$ (see \cite[Theorem 3.9]{Yil11a}).

The lower semicontinuity of $(H_{q,\bP}^{S,+})^{**}$ and the compactness of $\{\mu\in\cM_1(\Ospace_\bN):\,E^\mu[Z_1] = \xi\}$ ensure that the quenched contraction \eqref{qvarmod1} always has a minimizer. On the other hand, there is currently no general existence result for minimizers of the quenched contraction \eqref{qvarmod2}. See Section \ref{ss:addrem} for further remarks.
\end{remark}

\subsection{Spatially constant environments}
We illustrate our results in a simplified setting where the spatial variation of the environment is removed.      The quenched process $\Zvec$ is now a process of independent but not identically distributed variables.  LDPs for such  processes were originally established in    \cite{baxt-jain-sepp, come-89, sepp-ptrf-93I}, motivated in part by their application to the equilibrium  statistical mechanics of disordered lattice systems such as the Ising or Curie-Weiss models with random fields or coupling constants.  (Some of these large deviation results   have been reproduced  in Chapter 15 of the textbook \cite{RasSep15}.)   The novelty we provide here is the identification of the averaged and quenched contraction minimizers.     We find that many properties such as equality of averaged and quenched rate functions and minimizers fail.  

Take a Borel probability measure $\lambda$ on $\cP$ (defined in \eqref{pipi}). 
Let $(\bar q_i)_{i\in\bZ}$ be sampled from $\cP^\bZ$ according to $\lambda^{\otimes\bZ}$. 
Define $\w\in\Omega = \cP^{\bZ\times\bZ^d}$ by setting
\be\label{def:constenv}
\text{$\w_{i,x} = \bar q_i$ for every $i\in\bZ$ and $x\in\bZ^d$.}
\ee
This induces a probability measure $\bP$ on $(\Omega,\kS)$. Environments under $\bP$ are temporally i.i.d.\ and spatially constant. Hence, $\bP$ is invariant but not ergodic under the spatial translations $(T^s_y)_{y\in\bZ^d}$. 

For $\rho\in\bR^d$ and $\w\in\Omega$, define $$W(\rho,\w) = E_0^\w[e^{\langle\rho,Z_1\rangle}].$$
Observe that $\bE[W(\rho,\w)] 
= \phi_a(\rho)$. 
For the sake of eliminating trivial cases where the environment is effectively deterministic, we assume that
\be\label{yoksasic}
\text{$\bP(W(\rho,\w) = \phi_a(\rho)) < 1$ unless $\phi_a(\rho) = e^{\langle\rho,\llnxi\rangle}$.} 
\ee
The condition $\phi_a(\rho) = e^{\langle\rho,\llnxi\rangle}$ is the same as $\rho\in\partial I_{1,a}(\llnxi)$ (Proposition \ref{abimdeca} in Appendix \ref{app:subdifferential}). 

We start our study by giving a simple formula for the level-1 quenched rate function and showing that it is not equal to the averaged one at any atypical velocity.

\begin{proposition}\label{prop:scneq}
Assume \eqref{polyell} and \eqref{def:constenv}. Then, for every $\xi\in\cD$, 
\be\label{metgel}
I_{1,q}(\xi) = \sup_{\rho\in\bR^d}\{\langle\rho,\xi\rangle - \bE[\log W(\rho,\w)]\} \ge \sup_{\rho\in\bR^d}\{\langle\rho,\xi\rangle - \log\bE[W(\rho,\w)]\} = I_{1,a}(\xi).
\ee
If $\xi\in\mathrm{ri}(\cD)\setminus\{\llnxi\}$ and \eqref{yoksasic} holds, then the inequality in \eqref{metgel} is strict.
\end{proposition}

\begin{remark}
In Proposition \ref{prop:scneq}, we assume \eqref{polyell} in order to apply Corollary \ref{1qLDP}. In fact, when the environment is spatially constant, a weaker ellipticity condition 
is sufficient for the level-1 quenched LDP, but we do not pursue such technical improvements here. 
\end{remark}

Next we present the structure of the unique minimizer $\mu^\xi$  (defined in \eqref{sutesti}) of the averaged contraction (\ref{a_contraction})  (see Theorem \ref{thm:uniquemin}). For   $\rho\in\bR^d$ and $\w\in\Omega$ let $$u_1(\rho,\w) = \frac{W(\rho,\w)}{\phi_a(\rho)}.$$

\begin{proposition}\label{prop:scmin}
Assume \eqref{def:constenv}. Then, for every $\xi\in\mathrm{ri}(\cD)$, the pairs $(\wvec_i,Z_{i+1})_{i\ge0}$ are i.i.d.\ under $\mu^\xi$. The $\Omega$-marginal $\mu_\Omega^\xi$ and the Markov transition kernel $\pi_{0,1}^\xi$ of $\mu^\xi$ are given by
$$\left.\frac{d\mu_\Omega^\xi}{d\bP}\right|_{\kS_{0,n}}\!\!\!\!\!\!\!\!(\w) = \prod_{i=0}^{n-1}u_1(\rho,T_{i,0}\w)\qquad\text{and}\qquad\pi_{0,1}^\xi(0,z\,|\,\omega) = \pi_{0,1}(0,z\,|\,\omega)\frac{e^{\langle\rho,z\rangle}}{W(\rho,\w)},$$ 
for any  $\rho\in\partial I_{1,a}(\xi)$. If $\xi\neq\llnxi$ and \eqref{yoksasic} holds, then $\mu_\Omega^\xi\not\ll\bP$ on $\kS_{0,\infty}$, and 
$\pi_{0,1}^\xi$ is not obtained from $\pi_{0,1}$ via a Doob $h$-transform as in \eqref{doob}. 
\end{proposition}

The simultaneous lack of absolute continuity and Doob $h$-transform are consistent with Theorem \ref{thm:doobchar}.
Now that we have  simple formulas for $\mu_\Omega^\xi$ and $\pi_{0,1}^\xi$, we can  compute $h_{\kS_{0,\infty}}(\mu_\Omega^\xi\,|\,\bP)$ and $H_q(\mu^\xi)$. 

\begin{proposition}\label{prop:scent}
Assume \eqref{def:constenv}. Then, for every $\xi\in\mathrm{ri}(\cD)$ and $\rho\in\partial I_{1,a}(\xi)$,
\begin{align*}
h_{\kS_{0,\infty}}(\mu_\Omega^\xi\,|\,\bP) &= \bE[u_1(\rho,\w)\log u_1(\rho,\w)]\quad \text{and}\\
H_q(\mu^\xi) &= \langle\rho,\xi\rangle - \log\phi_a(\rho) - \bE[u_1(\rho,\w)\log u_1(\rho,\w)].
\end{align*}
If $\xi\neq\llnxi$ and \eqref{yoksasic} holds, then $h_{\kS_{0,\infty}}(\mu_\Omega^\xi\,|\,\bP) > 0$.
\end{proposition}

Proposition \ref{prop:scent} implies that $$h_{\kS_{0,\infty}}(\mu_\Omega^\xi\,|\,\bP) + H_q(\mu^\xi) = \langle\rho,\xi\rangle - \log\phi_a(\rho) = I_{1,a}(\xi) = h(\mu^\xi\,|\,P_0)$$ for every $\xi\in\mathrm{ri}(\cD)$ and $\rho\in\partial I_{1,a}(\xi)$, which is consistent with Theorem \ref{entropy_connection}.

Since the environments are spatially constant, under the quenched conditioning on $\w$   the $\Omega$-marginal of the empirical process $L^\infty_n$ of \eqref{eq:empi}  is a deterministic measure that converges to $\P$.  Consequently the quenched rate must blow up at measures with the ``wrong'' $\Omega$-marginal.  This was observed in 
\cite[Theorem III.1]{come-89}  and \cite[Theorem 3.4]{sepp-ptrf-93I}. 

\begin{proposition}\label{prop:scq}
Assume \eqref{polyell} and \eqref{def:constenv}. For every $\mu\in\cM_1(\Ospace_\bN)$, if $\mu_\Omega \ne \bP$, then $I_{3,q}(\mu) = \infty$. Consequently, if \eqref{yoksasic} holds, then
$$(H_{q,\bP}^{S,+})^{**}(\mu^\xi) = H_{q,\bP}^{S,+}(\mu^\xi) = \infty$$
for every $\xi\in\mathrm{ri}(\cD)\setminus\{\llnxi\}$.
\end{proposition}

If \eqref{polyell}, \eqref{def:constenv} and \eqref{yoksasic} hold, then it follows from Propositions \ref{prop:scneq}, \ref{prop:scent} and \ref{prop:scq} that all four statements in Theorem \ref{thm:charac} are false for every $\xi\in\mathrm{ri}(\cD)\setminus\{\llnxi\}$, which is consistent with their equivalence.

Proposition \ref{prop:scq} shows in a striking way how the alteration  of the entropy $H_q$ can  completely remove the averaged minimizers  $\mu^\xi$ from the effective domain of the quenched rate function.    In particular, 
   $\mu^\xi$   cannot be a minimizer of the quenched contractions \eqref{qvarmod1} or \eqref{qvarmod2}.   Our final result identifies the minimizer(s) of these quenched contractions. 

For   $\xi\in\mathrm{ri}(\cD)$ define $\nu^\xi\in\cM_1(\Ospace_\bN)$ by setting
\be\label{sutestiqqq}
\nu^\xi(d\w,dz_{1,n}) = \P(d\w)\prod_{i=0}^{n-1} \Bigl\{ \pi_{0,1}(0,z_{i+1}\,|\,T_{i,0}\omega)\frac{e^{\langle\rho,z_{i+1}\rangle}}{W(\rho,T_{i,0}\w)}\Bigr\}
\ee
for every $n\in\bN$, where $\rho\in\partial I_{1,q}(\xi)$. 

\begin{proposition}\label{prop:scminq}
Assume \eqref{polyell} and \eqref{def:constenv}. Then, for every $\xi\in\mathrm{ri}(\cD)$:
\begin{itemize}
\item [(a)] $\nu^\xi$ is well-defined and $S$-invariant;
\item [(b)] $E^{\nu^\xi}[Z_1] = \xi$;
\item [(c)] $(H_{q,\bP}^{S,+})^{**}(\nu^\xi) = H_q(\nu^\xi) = I_{1,q}(\xi)$; and
\item [(d)] $\nu^\xi$ is the unique minimizer of the variational formulas \eqref{qvarmod1} and \eqref{qvarmod2} of the quenched contractions from level-3 to level-1.
\end{itemize}
\end{proposition}

The $\Omega$-marginal of $\nu^\xi$ is $\nu_\Omega^\xi = \bP$, which is consistent with Proposition \ref{prop:scq}. The transition kernels of $\mu^\xi$ (see Proposition \ref{prop:scmin}) and $\nu^\xi$ are both of the form
$\pi_{0,1}(0,z\,|\,\omega)\frac{e^{\langle\rho,z\rangle}}{W(\rho,\w)}$, but defined using $\rho\in\partial I_{1,a}(\xi)$ and $\rho\in\partial I_{1,q}(\xi)$, respectively.

\subsection{Additional remarks and open problems}\label{ss:addrem}

\subsubsection{Minimizers of the contractions}

For every $\xi\in\mathrm{ri}(\cD)$, Theorem \ref{thm:uniquemin} identifies $\mu^\xi$ as the unique minimizer of the averaged contraction \eqref{a_contraction} from level-3 to level-1. When $I_{1,a}(\xi) = I_{1,q}(\xi)$, Theorem \ref{thm:miniq} says that $\mu^\xi$ is also the unique minimizer of the quenched contraction \eqref{qvarmod1}. Moreover, if $\mu_\Omega^\xi\ll\bP$ on $\kS_{0,\infty}$ (see Remark \ref{lazmols} for examples), then $\mu^\xi$ is the unique minimizer of the quenched contraction \eqref{qvarmod2}, too. In the latter case, Theorem \ref{thm:doobchar} gives a representation for the Markov transition kernel $\pi_{0,1}^\xi$ of the quenched walk under $\mu^\xi$ via a Doob $h$-transform.

When $I_{1,a}(\xi) < I_{1,q}(\xi)$, identifying the minimizers (if any) of the quenched contractions \eqref{qvarmod1} and \eqref{qvarmod2} or saying anything about their structure is an open problem in general. Note that \eqref{qvarmod1} always has a minimizer (see Remark \ref{minvarmi}). In contrast, we expect that \eqref{qvarmod2} has no minimizers when the environment $(\w_{i,x})_{(i,x)\in\bZ\times\bZ^d}$ is i.i.d., but this is yet to be shown. On the other hand, in the case of spatially constant environments, Proposition \ref{prop:scminq} provides the unique minimizer of both of these quenched contractions.

In a recent article \cite{RasSepYil_preprint}, we obtained results on the existence and identification of minimizers of   variational formula \eqref{gV:cc}  and its counterpart for the annealed  free energy.  
This covers the logarithmic moment generating functions $$\log\phi_a(\rho) = \log E_0[e^{\langle\rho,Z_1\rangle}]\qquad\text{and}\qquad\Lambda_{1,q}(\rho) = \lim_{n\to\infty}\frac1{n}\log E_0^\omega[e^{\langle\rho,X_n\rangle}]$$ for RWDRE. These functions are the convex conjugates of $I_{1,a}$ and $I_{1,q}$, respectively, by Varadhan's lemma. In future work, we hope to combine these previous results with the current ones and thereby deepen our understanding of the large deviation behavior of RWDRE.

\subsubsection{Connecting the rate functions}

How the averaged and quenched rate functions are related to each other is an important question in the study of processes in random environments. For example, at level-1, obtaining an expression for $I_{1,a}$ in terms of $I_{1,q}$ would provide us with valuable information regarding how the path and the environment conspire towards the realization of atypical velocities. This question is answered  with variational formulas in \cite{ComGanZei00} for one-dimensional nearest-neighbor classical RWRE under the i.i.d.\ environment assumption and in \cite{emra-janj-15}  for the exactly solvable corner growth model with random parameters.   It is an open problem for example for RWRE  in higher dimensions or under more general conditions.

In the context of RWDRE, Theorem \ref{entropy_connection} provides a partial answer to the aforementioned question at level-3 since it connects $I_{3,a}(\mu) = h(\mu\,|\,P_0)$ with $I_{3,q}(\mu) = (H_{q,\bP}^{S,+})^{**}(\mu)$ only indirectly via $H_q(\mu)$. This reduces the original question to understanding the variational expression $(H_{q,\bP}^{S,+})^{**}(\mu) - H_q(\mu)$, which is one of our goals for future work. So far, we know that this difference is nonnegative (see Corollary \ref{cor:zincir}), and equal to zero at $\mu^\xi$ if and only if $I_{1,a}(\xi) = I_{1,q}(\xi)$ (see Theorem \ref{thm:charac}).

\subsubsection{Equality of the rate functions}

When $I_{1,a}(\xi) = I_{1,q}(\xi)$ at an atypical velocity $\xi\in\mathrm{ri}(\cD)\setminus\{\llnxi\}$, the walk is solely responsible (in the exponential scale) for the occurrence of the rare event $\{X_n/n\approx\xi\}$ under the joint measure $P_0$. Theorem \ref{thm:charac} makes this precise by the statement $h_{\kS_{0,\infty}}(\mu_\Omega^\xi\,|\,\mathbb{P}) = 0$. 

Theorem \ref{thm:preveze} lists the previous results regarding the equality of the level-1 rate functions. The decisive statement for $d=1$ is believed to be true also for $d=2$. In contrast, recalling Proposition \ref{prop:elem} (a,d), both $\cC = \{\xi\in\cD:\,I_{1,a}(\xi) = I_{1,q}(\xi)\}$ and $\cD\setminus\cC$ have nonempty interiors when $d\ge3$. Hence, there is a phase transition at the boundary of $\cC$, and we would like to analyze the structure of $\mu^\xi$ when $\xi\in\partial\cC$. The characterizations in Theorem \ref{thm:charac} can potentially shed light on this problem.

Theorem \ref{thm:miniq}(b) provides a sufficient condition for $I_{1,a}(\xi) = I_{1,q}(\xi)$, namely $\mu_\Omega^\xi\ll\bP$ on $\kS_{0,\infty}$. Whether this condition is also necessary for $I_{1,a}(\xi) = I_{1,q}(\xi)$ is an important open problem which is related to the existence of the critical (i.e., strong but not very strong) disorder regime for directed polymers. See \cite[Section 1.3]{RasSepYil_preprint} for details.   In the RWDRE setting and notation, the environment is said to manifest
\begin{itemize}
\item [(i)] weak disorder if $\mu_\Omega^\xi\ll\bP$ on $\kS_{0,\infty}$ (see Remark \ref{lazmols} for examples),
\item [(ii)] strong disorder if $\mu_\Omega^\xi\not\ll\bP$ on $\kS_{0,\infty}$, and
\item [(iii)] very strong disorder if $\left.\frac{d\mu_\Omega^\xi}{d\bP}\right|_{\kS_{0,n}}\!\!\!$ decays exponentially to zero as $n\to\infty$ (equivalently $I_{1,a}(\xi) < I_{1,q}(\xi)$).
\end{itemize}
Here, $\xi$ 
is a multidimensional analog of inverse temperature, with the LLN velocity $\llnxi$ corresponding to infinite temperature. It is tempting to connect this problem of critical disorder with the previous one regarding the structure of $\mu^\xi$ at the boundary of $\cC$, but we refrain from proposing any conjectures.

%
%

\section{Level-3 averaged LDP from the point of view of the particle}\label{sec:3aLDP}

We start this section with an important point regarding the relative entropies $H_{k,\ell}(\mu\,\vert\, P_0)$ defined in (\ref{relent}). 
We will refer to this point below in the proof of Theorem \ref{3aLDP}.

\begin{remark}\label{shiftprob}
It is not necessarily the case that $H_{k,\ell}(\mu\,\vert\, P_0) = H_{0,\ell - k}(\mu\,\vert\, P_0)$ for $S$-invariant $\mu\in\cM_1(\Ospace_\bN)$ and $0<k < \ell$. This is because the distribution of $(\wvec_i,Z_{i+1})$ under $P_0$ changes with $i$. Here is an example: The simplest  $\Smap$-invariant probability measure on $\Ospace_\bN$ is of the product type
$$\mu(d\w,d\zvec)=\bigotimes_{i\in\bZ} \nu(d\wvec_i)\otimes\bigotimes_{j\in\bN}\alpha(z_j).$$
Take $\nu(d\wvec_i)=\P_s(d\wvec_i)$ and $\alpha=\delta_z$ for some fixed $z\in\cR$. Then,
$$H_{i,i+1}(\mu\,\vert\,P_0) = -\E[ \log\sum_x P_0(X_i = x) \pi_{i,i+1}(x,x+z\,|\,\w)].$$
\end{remark}

\begin{proof}[Proof of Theorem \ref{3aLDP}]
We will transform the problem into a level-3 LDP for an i.i.d.\ sequence on the space $\Ospace_{\bZ} = \Omega\times\cR^{\bZ}$.  
First, we define the measure on $\Ospace_\bZ$ that will give the desired i.i.d.\ sequence.

Recall the definition  \eqref{xz}  $x_i=-\sum_{j=i+1}^0 z_j$ of the backward path $(x_i)_{i\le0}$.  
For $n\in\bN$, let
\begin{align*}
\varphi^\w_{-n}(z_{-n+1,0}) &= \prod_{i=-n}^{-1}\pi_{i,i+1}(x_i,x_{i+1}\,|\,\w)
= \prod_{i=-n}^{-1}\pi_{0,1}(0,z_{i+1}\,|\,T_{i,x_i}\w). 
\end{align*}
  Note that $\varphi^\w_{-n}(z_{-n+1,0})$ is not a probability distribution on vectors $z_{-n+1,0}$ because it does not sum up to one. Set 
\be\label{f_n}  f_n(\w) = \sum_{z_{-n+1,0}\in\cR^n}\varphi^\w_{-n}(z_{-n+1,0}) = \sum_{x\in\bZ^d} P_0^{T_{-n,-x}\w}(X_n=x). \ee
On the $\sigma$-algebra $\cA_{-n,\infty}^{-n,\infty}$, we define a measure $\tilde P^{(-n)}$ by setting
$$\tilde P^{(-n)}(d\wlev_{-n,\infty}, dz_{-n+1,m})  = \P(d\wlev_{-n,\infty})\, \varphi^\w_{-n}(z_{-n+1,0})P^\w_0(z_{1,m})\!\!\!\prod_{i = -n+1}^m\!\!\!c_\cR(z_i)$$
for every $m\in\bN$. Here, $c_\cR$ denotes the counting measure on $\cR$.
\begin{lemma}\label{Pinfty-lm2}
$(\tilde P^{(-n)})_{n\in\bN}$ are consistent probability measures and hence they induce a probability measure $P_0^\infty$ on $\Ospace_\bZ$.
\end{lemma}
\begin{proof}
For every $m,n\in\bN$ and every test function $f\in b\cA_{-n,\infty}^{-n,m}$,
\begin{align}
&\int \tilde P^{(-n-1)}(d\wlev_{-n-1,\infty}, dz_{-n,m})f(\wlev_{-n,\infty}, z_{-n+1,m})\nonumber\\
&\qquad = \int \P(d\wlev_{-n-1,\infty})\sum_{z_{-n,m}}\varphi^\w_{-n-1}(z_{-n,0})P^\w_0(z_{1,m})f(\wlev_{-n,\infty}, z_{-n+1,m})\nonumber\\
&\qquad = \int \P(d\wlev_{-n,\infty})\sum_{z_{-n+1,m}}\varphi^\w_{-n}(z_{-n+1,0})P^\w_0(z_{1,m})f(\wlev_{-n,\infty}, z_{-n+1,m})\nonumber\\
&\hspace{44mm}\times\int \P_s(d\wlev_{-n-1})\sum_{z_{-n}}\pi_{0,1}(0,z_{-n}\,|\,T_{-n-1, x_{-n}-z_{-n}}\w)\label{indepooo}\\
&\qquad = \int \P(d\wlev_{-n,\infty})\sum_{z_{-n+1,m}}\varphi^\w_{-n}(z_{-n+1,0})P^\w_0(z_{1,m})f(\wlev_{-n,\infty}, z_{-n+1,m})\nonumber\\
&\hspace{44mm}\times\int \P_s(d\wlev_{-n-1})\sum_{z_{-n}}\pi_{0,1}(0,z_{-n}\,|\,T_{-n-1,x_{-n}}\w)\label{invarooo}\\
&\qquad = \int \tilde P^{(-n)}(d\wlev_{-n,\infty}, dz_{-n+1,m})f(\wlev_{-n,\infty}, z_{-n+1,m}).\nonumber
\end{align}
Here, \eqref{indepooo} uses the temporal independence of the environment, whereas \eqref{invarooo} follows from exchanging the order of the last integral and sum, 
recalling the spatial translation invariance assumption, and restoring the order of the last integral and sum.
\end{proof}

Recall that $(T^s_y)_{y\in\bZ^d}$ are spatial translations defined by $(T^s_y \wlev_j)_x = \w_{j,x+y}$ for $j\in\bZ$ and $x,y\in\bZ^d$. We use these translations to introduce the so-called slab variables 
\be\label{slabvaroglu}
\slabv_j=(T^s_{x_j}\wlev_j,z_{j+1})\,,\quad j\in\bZ.
\ee
This choice of terminology comes from viewing RWDRE in $\bZ^d$ as a directed RWRE in $\bZ^{d+1}$. Note that $\slabv_j$ is centered at the point $x_j$ on the path. In this sense, the slab variables are adapted to the POV of the particle. Equivalently, they satisfy $\slabv_j=\slabv_0\circ \Smap^{j}$ for $j\in\bZ$. For any pair of indices $-\infty<k\le\ell<\infty$, we write $\slabv_{k,\ell} = (\slabv_k,\slabv_{k+1},\ldots,\slabv_\ell)$.
\begin{lemma}\label{Pinfty-lm3}
$P_0^\infty$ is $\Smap$-invariant and the slab variables $(\slabv_j)_{j\in\bZ}$ are i.i.d.\ under $P_0^\infty$.
\end{lemma}

\begin{proof}
This is an immediate consequence of the following induction steps.  Let $E_0^\infty$ stand for expectation under $P_0^\infty$. For $-n<0<m$,
\begin{align*}
E_0^\infty [  f(\slabv_{-n,m-1}) g(\slabv_{m})] &= \sum_{z_{-n+1,m+1}} \E \Bigl[   \varphi^\w_{-n}(z_{-n+1,0}) P_0^\w(z_{1,m}) 
f((T^s_{x_j}\wlev_j,z_{j+1})_{-n\le j\le m-1})  \\
&\,\qquad\qquad\qquad\qquad\qquad\times 
\pi_{0,1}(0,z_{m+1}\,|\,T_{m, x_{m}}\w) 
g(T^s_{x_{m}}\wlev_{m},z_{m+1})  \Bigr] \\
&= \sum_{z_{-n+1,m+1}} \E \Bigl[   \varphi^\w_{-n}(z_{-n+1,0}) P_0^\w(z_{1,m}) 
f((T^s_{x_j}\wlev_j,z_{j+1})_{-n\le j\le m-1}) \Bigr]  \\
&\,\quad\qquad\qquad\times 
\E \Bigl[ \pi_{0,1}(0,z_{m+1}\,|\,T_{m, x_{m}}\w) 
g(T^s_{x_{m}}\wlev_{m},z_{m+1})  \Bigr] \\
&=    E_0^\infty [f(\slabv_{-n,m-1})]  E_0[g(\wlev_{0},Z_1)]
\end{align*}
by temporal independence and spatial translation invariance, where $f$ and $g$ are test functions on appropriate spaces. Similarly,
\begin{align*}
E_0^\infty [ g(\slabv_{-n}) f(\slabv_{-n+1,m})] &= \sum_{z_{-n+1,m+1}} \E \Bigl[  \pi_{0,1}(0,z_{-n+1}\,|\,T_{-n, x_{-n}}\w) 
g(T^s_{x_{-n}}\wlev_{-n},z_{-n+1})  \\
&\ \qquad\qquad\qquad\times \varphi^\w_{-n+1}(z_{-n+2,0}) P_0^\w(z_{1,m+1}) 
f((T^s_{x_j}\wlev_j,z_{j+1})_{-n+1\le j\le m}) \Bigr] \\
&= \sum_{z_{-n+1,m+1}} \E \Bigl[  \pi_{0,1}(0,z_{-n+1}\,|\,T_{-n, x_{-n}}\w) 
g(T^s_{x_{-n}}\wlev_{-n},z_{-n+1}) \Bigr]  \\
&\,\quad\qquad\qquad\times 
 \E \Bigl[  \varphi^\w_{-n+1}(z_{-n+2,0}) P_0^\w(z_{1,m+1}) 
f((T^s_{x_j}\wlev_j,z_{j+1})_{-n+1\le j\le m}) \Bigr] \\
&=  \sum_{z} \E \bigl[  \pi_{0,1}(0,z\,|\, \w) 
g(\wlev_{0},z) \bigr]  \\
&\ \quad\times\sum_{z_{-n+2,m+1}} \E \Bigl[  \varphi^\w_{-n+1}(z_{-n+2,0}) P_0^\w(z_{1,m+1}) 
f((T^s_{x_j}\wlev_j,z_{j+1})_{-n+1\le j\le m}) \Bigr]  \\
&=    E_0[g(\wlev_{0},Z_1) ]  E_0^\infty [f(\slabv_{-n+1,m})].
\qedhere\end{align*}
\end{proof}


Denote the full sequence of slab variables by $\slabvs=(\slabv_j)_{j\in\bZ}$ and let $\tau$ be the temporal shift on these sequences, defined by $(\shifta\slabvs)_j=\slabv_{j+1}$. With this notation, the empirical process induced by the slab variables is
$$L^{slab}_n=\frac1n\sum_{i=0}^{n-1}\delta_{\shifta^i\slabvs}.$$
The LDP for $L^{slab}_n$ under $P_0^\infty$ is an instance of the well-known Donsker-Varadhan level-3 LDP for sequences of i.i.d.\ random variables taking values in Polish spaces, see, e.g., \cite[Chapter 6]{RasSep15}. We state this result as Proposition \ref{pr:slabldp} below, after some preparation for providing a formula for the corresponding rate function.

We can glue together the environment components of the slab variables to form an $\w'\in\Omega$ with $\wvec'_j = T^s_{x_j}\wvec_j$ for $j\in\bZ$, and thereby identify the space of slab sequences with $\Ospace_\bZ$. (For the sake of convenience, we will write $\w$ instead of $\w'$.) This identification already factors in the POV of the particle, and the shift $\shifta$ acts on (environment, path) pairs simply by  
$$\bigl(\shifta(\wvec_{\boldsymbol\cdot},z_{\boldsymbol\cdot+1})\bigr)_j = (\wvec_{j+1},z_{j+2}).$$
In other words, the sequence $\slabvs$ can be thought of as a bijective map on $\Ospace_\bZ$. It induces a $\tau$-invariant distribution $P^\infty_0\circ\slabvs^{-1}$ on this space.
Since the $\sigma$-algebras $\cA_{k,\ell}^{k,\ell}$ are now regarded as being generated by the i.i.d.\ slab variables $\slabv_j$,  the problem with shifting relative entropy (cf.\ Remark \ref{shiftprob}) disappears.  
For $\shifta$-invariant probability measures $Q$ on $\Ospace_\bZ$, the specific relative entropy
\be\label{def:hlevel}
h(Q\,\vert\, P_0^\infty\circ\slabvs^{-1}) = \lim_{\ell\to\infty}\frac{H_{0,\ell}(Q\,\vert\, P^\infty_0\circ\slabvs^{-1})}{\ell} = \sup_{-\infty<k<\ell<\infty} \frac{H_{k,\ell}(Q\,\vert\,P^\infty_0\circ\slabvs^{-1})}{\ell-k} 
\ee
exists by a standard superadditivity argument (see \cite[Theorem 6.7]{RasSep15}).

\begin{proposition}\label{pr:slabldp}
$(P_0^\infty(L_n^{slab}\in\cdot\,))_{n\ge1}$ satisfies an LDP with rate function $I^{slab}: \cM_1(\Ospace_\bZ)\to[0,\infty]$ given by
$$I^{slab}(Q)=\begin{cases} h(Q\,\vert\, P^\infty_0\circ\slabvs^{-1}) 
&\text{if $Q$ is $\shifta$-invariant,}\\
 \infty &\text{otherwise.}
\end{cases}$$
\end{proposition}

Next, we transform this LDP into one for the empirical measures
$$L_n^\bZ=\frac1n\sum_{i=0}^{n-1}\delta_{T_{i,X_i}\w\,,\, \shift^i\bar\Zvec} = \frac1n\sum_{i=0}^{n-1}\delta_{\Smap^i(\w\,,\, \bar\Zvec)}.$$
The inverse of the map $(\w,\zvecd)\mapsto \slabvs$ is $\gmap:\Ospace_\bZ\to\Ospace_\bZ$ that acts on slabs via
$$(\gmap(\wlev_{\boldsymbol\cdot},z_{\boldsymbol\cdot+1}))_j =  (T^s_{-x_j}\wlev_j, z_{j+1}),\quad j\in\bZ.$$
Note that
$$\Smap\circ\gmap = \gmap\circ\shifta$$
and $L_n^\bZ=L_n^{slab}\circ\gmap^{-1}$.  
A probability measure $Q$ on $\Ospace_\bZ$ is $\Smap$-invariant iff $Q' = Q\circ\gmap$ is $\shifta$-invariant. Apply the contraction principle to the LDP in Proposition \ref{pr:slabldp} with the map $Q' = Q\circ\gmap\mapsto Q$ on $\cM_1(\Ospace_\bZ)$. This gives an LDP for $(P_0^\infty(L_n^\bZ\in\cdot\,))_{n\ge1}$ with rate function 
$$I^\bZ(Q)=\begin{cases} h(Q\circ\gmap \,\vert\, P^\infty_0\circ\gmap) &\text{if $Q$ is $\Smap$-invariant,}\\
\infty &\text{otherwise.}
\end{cases}$$
Since $\gmap$ acts bijectively on any collection of adjacent slabs that includes the zeroth slab and hence preserves $\cA_{k,\ell}^{k,\ell}$-measurability for $k\le 0<\ell$, 
$$H_{k,\ell}(Q\circ\gmap\,\vert\, P^\infty_0\circ\gmap) = H_{k,\ell}(Q\,\vert\, P^\infty_0).$$
Thus, using \eqref{def:hlevel}, we can define a specific relative entropy for $\Smap$-invariant $Q$ by restricting the intervals $[k,\ell)$ to include $0$: 
\be\label{hakkidogar}
h(Q\,\vert\, P^\infty_0) = \lim_{\ell\to\infty}\frac{H_{0,\ell}(Q\,\vert\, P^\infty_0)}{\ell} = \sup_{-\infty<k\le 0<\ell<\infty} \frac{H_{k,\ell}(Q\,\vert\,P^\infty_0)}{\ell-k} = h(Q\circ\gmap \,\vert\, P^\infty_0\circ\gmap).
\ee
The statement of the LDP we have established is simplified as follows.

\begin{proposition}  
$(P_0^\infty(L_n^\bZ\in\cdot\,))_{n\ge1}$ satisfies an LDP with rate function $I^\bZ: \cM_1(\Ospace_\bZ)\to[0,\infty]$ given by
$$I^\bZ(Q)=\begin{cases} h(Q\,\vert\, P^\infty_0)&\text{if $Q$ is $\Smap$-invariant,}\\
 \infty &\text{otherwise.}
\end{cases}$$
\label{thm:stZsystem}
\end{proposition}

As the last step, we transform this LDP into the one we want. Denote the natural $\Ospace_\bZ\to\Ospace_\bN$  projection by $\zproj(\w,\zvecd)=(\w,\zvec)$. We can think of the empirical process $L_n^\infty$ (introduced in
\eqref{eq:empi}) as a function from $\Ospace_\bZ$ into $\cM_1(\Ospace_\bN)$ by replacing $L_n^\infty$ with $L_n^\infty\circ\zproj$.  We drop the projection from the notation since the coordinates of $\shift^i\zvec=z_{i+1,\infty}$ are defined on $\Ospace_\bZ$ as well as on $\Ospace_\bN$. The contraction principle gives an LDP for $(P_0^\infty(L_n^\infty\in\cdot\,))_{n\ge1}$ with rate function $I_{3,a}^\infty:\cM_1(\Ospace_\bN)\to[0,\infty]$ defined by
$$I_{3,a}^\infty(\mu)=\begin{cases} h(\bar\mu\,\vert\, P^\infty_0) &\text{if $\mu$ is $\Smap$-invariant,}\\ \infty &\text{otherwise.} \end{cases}$$
Recall that, for any $S$-invariant $\mu$ on $\Ospace_\bN$, $\bar\mu$ denotes the unique $S$-invariant extension to $\Ospace_{\bZ}$. 
By the $\Smap$-invariance of both $\bar\mu$ and $P^\infty_0$, the entropies can be shifted to nonnegative levels so that 
$$H_{k,\ell}(\bar\mu\,\vert\,P^\infty_0) = H_{0,\ell-k}(\bar\mu\,\vert\,P^\infty_0) = H_{0,\ell-k}(\mu\,\vert\,P_0)\quad\text{for $k\le 0<\ell$.}$$
The first equality above follows from the observation that $f\mapsto f\circ S^{-k}$ and $g\mapsto g\circ S^{k}$ are bijections between $b\cA_{k,\ell}^{k,\ell}$ and $b\cA_{0,\ell-k}^{0,\ell-k}$ for $k\le 0<\ell$. (However, this is not the case when $0<k<\ell$, cf.\ Remark \ref{shiftprob}.) The second equality is valid because $\bar\mu$ and $\mu$ (resp.\ $P^\infty_0$ and $P_0$) agree at nonnegative times. Comparing \eqref{def:hmuPzero} and \eqref{hakkidogar}, we conclude that $I_{3,a}^\infty$ is equal to the level-3 averaged rate function $I_{3,a}$ defined in \eqref{level3a}. 

It remains to transfer the LDP from $(P_0^\infty(L_n^\infty\in\cdot\,))_{n\ge1}$ to $(P_0(L_n^\infty\in\cdot\,))_{n\ge1}$ separately for lower and upper bounds. This works easily because (i) weak topology on $\cM_1(\Ospace_\bN)$ is determined by finite-dimensional distributions, (ii) the dependence of $L_n^\infty$ on environments at negative times vanishes as $n\to\infty$, and (iii) the measures $P_0$ and $P_0^\infty$ agree on environments and steps at nonnegative times. We leave the routine details to the reader. This completes the proof of Theorem \ref{3aLDP}. 
\end{proof}

\section{Minimizer of the averaged contraction}\label{sec:mini}

We start by listing some properties of the measure $\mu^\xi$ which was introduced in \eqref{sutesti} for $\xi\in\mathrm{ri}(\cD)$. 

\begin{proposition}\label{prop:elem2}
For every $\xi\in\mathrm{ri}(\cD)$:
\begin{itemize}
\item [(a)] $\mu^\xi$ is well-defined and $S$-invariant;
\item [(b)] the slab variables $(\slabv_\ell)_{\ell\ge0}$ are i.i.d.\ under $\mu^\xi$;
\item [(c)] $E^{\mu^\xi}[Z_1] = \xi$;
\item [(d)] $H_{0,\ell}(\mu^\xi\,|\,P_0) = \ell I_{1,a}(\xi)$ for every $\ell\in\bN$; and
\item [(e)] $I_{3,a}(\mu^\xi) = h(\mu^\xi\,|\,P_0) = I_{1,a}(\xi)$.
\end{itemize}
\end{proposition}

\begin{remark}
At the LLN velocity $\llnxi$, the RHS of \eqref{sutesti} gets simplified since $0\in\partial I_{1,a}(\llnxi)$, and we deduce that $\mu^{\llnxi} = P_0$ on $\cA_{0,\infty}^{0,\infty}$. (However, they are not equal on $\cA_{-\infty,\infty}^{0,\infty}$, cf.\ Remark \ref{canmur}.) Moreover, the $S$-invariant extension $\bar\mu^{\llnxi}\in\cM_1(\Ospace_\bZ)$ of $\mu^{\llnxi}$ is equal to the measure $P_0^\infty$ which was defined in Lemma \ref{Pinfty-lm2}. Therefore, Proposition \ref{prop:elem2}(a,b) generalize Lemma \ref{Pinfty-lm3}.
\end{remark}

\begin{proof}[Proof of Proposition \ref{prop:elem2}]
Fix an arbitrary $\xi\in\mathrm{ri}(\cD)$.
\begin{itemize}

\item [(a)] We prove in Theorem \ref{epenlaz}(b) from Appendix \ref{app:subdifferential} that $\langle \rho,z\rangle - \log\phi_a(\rho) = \langle \rho',z\rangle - \log\phi_a(\rho')$ for every $\rho,\rho'\in\partial I_{1,a}(\xi)$ and $z\in\cR$. Therefore, the RHS of \eqref{sutesti} is well-defined. In order to conclude that $\mu^\xi$ is well-defined, it remains to show (for $-\infty<k\le 0 <\ell<\infty$ and $f\in b\cA_{k,\ell}^{0,\ell}$) that the RHS of \eqref{sutesti} does not change if we replace (i) $k$ by $k-1$ or (ii) $\ell$ by $\ell +1$.
\begin{itemize}
\item [(i)] Since $f\circ S^{-k+1}\in b\cA_{1,\ell - k+1}^{-k+1,\ell - k+1}$, 
\begin{align*}
&E_0[e^{\langle\rho, X_{\ell-k+1}\rangle - (\ell-k + 1)\log\phi_a(\rho)}f\circ S^{-k+1}]\\
&\quad = \sum_{z\in\cR}\bE\left[\pi_{0,1}(0,z\,|\,\w)e^{\langle\rho,z\rangle - \log\phi_a(\rho)}\,E_0^{T_{1,z}\w}[e^{\langle\rho,X_{\ell-k}\rangle - (\ell-k)\log\phi_a(\rho)} f\circ S^{-k}]\right]\\
&\quad = \sum_{z\in\cR}\hat q(z)e^{\langle\rho,z\rangle - \log\phi_a(\rho)}E_0[e^{\langle\rho,X_{\ell-k}\rangle - (\ell-k)\log\phi_a(\rho)} f\circ S^{-k}]\\
&\quad = E_0[e^{\langle\rho,X_{\ell-k}\rangle - (\ell-k)\log\phi_a(\rho)} f\circ S^{-k}]
\end{align*}
by temporal independence and spatial translation invariance.
\item [(ii)] Similarly, since $f\circ S^{-k}\in b\cA_{0,\ell - k}^{-k,\ell - k}$, 
\begin{align*}
&E_0[e^{\langle\rho, X_{\ell-k+1}\rangle - (\ell-k + 1)\log\phi_a(\rho)}f\circ S^{-k}]\\
&\quad = \sum_{x}\bE\left[E_0^\w[e^{\langle\rho, X_{\ell-k}\rangle - (\ell-k)\log\phi_a(\rho)}f\circ S^{-k},\,X_{\ell-k} = x]\sum_{z\in\cR}\pi_{0,1}(0,z\,|\,T_{\ell-k,x}\w)e^{\langle\rho,z\rangle - \log\phi_a(\rho)}\right]\\
&\quad = \sum_{x}E_0[e^{\langle\rho, X_{\ell-k}\rangle - (\ell-k)\log\phi_a(\rho)}f\circ S^{-k},\,X_{\ell-k} = x]\sum_{z\in\cR}\hat q(z)e^{\langle\rho,z\rangle - \log\phi_a(\rho)}\\
&\quad = E_0[e^{\langle\rho,X_{\ell-k}\rangle - (\ell-k)\log\phi_a(\rho)} f\circ S^{-k}].
\end{align*}
\end{itemize}
Finally, the $S$-invariance of $\mu^\xi$ follows from (ii). Indeed, $f\circ S\in b\cA_{k+1,\ell+1}^{1,\ell+1}\subset b\cA_{k,\ell+1}^{0,\ell+1}$ and
\begin{align*}
\int f\circ Sd\mu^\xi &= E_0[e^{\langle\rho, X_{\ell-k+1}\rangle - (\ell-k + 1)\log\phi_a(\rho)}f\circ S^{-k}]\\
&= E_0[e^{\langle\rho,X_{\ell-k}\rangle - (\ell-k)\log\phi_a(\rho)} f\circ S^{-k}] = \int fd\mu^\xi.
\end{align*}

\item [(b)] For every $\ell\in\bN$, 
\begin{align*}
&E^{\mu^\xi}[f(\slabv_{0,\ell-1})g(\slabv_\ell)] = E_0[e^{\langle\rho, X_{\ell+1}\rangle - (\ell + 1)\log\phi_a(\rho)}f(\slabv_{0,\ell-1})g(\slabv_\ell)]\\
&\quad = \sum_{x}\bE\left[E_0^\w[e^{\langle\rho, X_{\ell}\rangle - \ell\log\phi_a(\rho)}f(\slabv_{0,\ell-1}),\,X_{\ell} = x]\sum_{z\in\cR}\pi_{0,1}(0,z\,|\,T_{\ell,x}\w)e^{\langle\rho,z\rangle - \log\phi_a(\rho)}g(T_x^s\wlev_\ell,z)\right]\\
&\quad = \sum_{x}E_0[e^{\langle\rho, X_{\ell}\rangle - \ell\log\phi_a(\rho)}f(\slabv_{0,\ell-1}),\,X_{\ell} = x]\,\bE\left[\sum_{z\in\cR}\pi_{0,1}(0,z\,|\,T_{\ell,x}\w)e^{\langle\rho,z\rangle - \log\phi_a(\rho)}g(T_x^s\wlev_\ell,z)\right]\\
&\quad = \sum_{x}E_0[e^{\langle\rho, X_{\ell}\rangle - \ell\log\phi_a(\rho)}f(\slabv_{0,\ell-1}),\,X_{\ell} = x]\,\bE\left[\sum_{z\in\cR}\pi_{0,1}(0,z\,|\,\w)e^{\langle\rho,z\rangle - \log\phi_a(\rho)}g(\wlev_0,z)\right]\\
&\quad = E^{\mu^\xi}[f(\slabv_{0,\ell-1})]\,E^{\mu^\xi}[g(\slabv_0)],
\end{align*}
where $f$ (resp.\ $g$) is a test function on $\ell$ (resp.\ $1$) slab variable(s).

\item [(c)] $E^{\mu^\xi}[Z_1] = E_0[e^{\langle\rho,Z_1\rangle - \log\phi_a(\rho)}Z_1] = \nabla\log\phi_a(\rho) = \xi$ (see \eqref{entemizig} in Appendix \ref{app:subdifferential} for the last equality.)

\item [(d)] For every $\ell\in\bN$,
\begin{align*}
H_{0,\ell}(\mu^\xi\,|\,P_0) &= \ell H_{0,1}(\mu^\xi\,|\,P_0) = \ell E_0[e^{\langle\rho,Z_1\rangle - \log\phi_a(\rho)}(\langle\rho,Z_1\rangle - \log\phi_a(\rho))]\\
&= \ell(\langle\rho,\xi\rangle - \log\phi_a(\rho)) = \ell I_{1,a}(\xi),
\end{align*}
where the first and third equalities use (b) and (c), respectively. (See \eqref{entemizig} in Appendix \ref{app:subdifferential} for the last equality.)

\item [(e)] This is immediate from (a) and (d).\qedhere

\end{itemize}
\end{proof}

The main ingredient in the proof of Theorem \ref{thm:uniquemin} is the following result.

\begin{proposition}\label{averagedconditioning}
For every $\xi\in\mathrm{ri}(\cD)$, $\e>0$, $\ell\in\bN$ and $f\in b\cA_{-\ell,\ell}^{0,\ell}$,
$$\limsup_{\delta\to0}\limsup_{n\rightarrow\infty}\frac{1}{n}\log P_0\left(\ |\int\!\! f\,dL_n^\infty-\int\!\! fd\mu^\xi|>\e\ \left|\ |\frac{X_n}{n}-\xi|<\delta\right.\right)<0.$$
\end{proposition}

This result is essentially \cite[Theorem 1]{Yil09a} which is stated there under the assumptions in \eqref{specass} which are more stringent than our current assumptions. For the sake of completeness and convenience, we provide below  a  streamlined adaptation of the proof to our setting. 

\begin{proof}[Proof of Proposition \ref{averagedconditioning}]
Let $g = f - \int fd\mu^\xi$. Then, $\int f\,dL_n^\infty - \int fd\mu^\xi = \int g\,dL_n^\infty =: \langle g,L_n^\infty\rangle$. By a standard change-of-measure argument and the level-1 averaged LDP, we see that for any $s>0$,
\begin{align}
&\limsup_{n\rightarrow\infty}\frac{1}{n}\log P_0(\langle g,L_n^\infty\rangle>\e\,|\, |\frac{X_n}{n}-\xi|<\delta)\nonumber\\
\leq&\limsup_{n\rightarrow\infty}\frac{1}{n}\log P_0(\langle g,L_n^\infty\rangle>\e\,, |\frac{X_n}{n}-\xi|<\delta)-\liminf_{n\rightarrow\infty}\frac{1}{n}\log P_0(|\frac{X_n}{n}-\xi|<\delta)\nonumber\\
\leq&\limsup_{n\rightarrow\infty}\frac{1}{n}\log E_0[e^{\langle\rho,X_n\rangle},\langle g,L_n^\infty\rangle>\e\,, |\frac{X_n}{n}-\xi|<\delta] - \langle\rho,\xi\rangle + I_{1,a}(\xi) + |\rho|\delta\nonumber\\
\leq&\limsup_{n\rightarrow\infty}\frac{1}{n}\log E_0[e^{\langle\rho,X_n\rangle-n\log\phi_a(\rho)},\langle g,L_n^\infty\rangle>\e] + |\rho|\delta\nonumber\\
\leq&\limsup_{n\rightarrow\infty}\frac{1}{n}\log E_0[e^{\langle\rho,X_n\rangle-n\log\phi_a(\rho) + ns\langle g,L_n^\infty\rangle}] - s\e + |\rho|\delta,\label{cebi}
\end{align}
where the last line follows from the exponential Chebyshev inequality. Let $G_j = g\circ S^j$ for $0\le j\le n-1$ and note that
\begin{align}
E_0[e^{\langle\rho,X_n\rangle-n\log\phi_a(\rho) + ns\langle g,L_n^\infty\rangle}] = &E_0[e^{\langle\rho,X_n\rangle-n\log\phi_a(\rho) + s\sum_{j=0}^{n-1}G_j}]\nonumber\\
\leq&\prod_{i=0}^{2\ell-1}E_0[e^{\langle\rho,X_n\rangle-n\log\phi_a(\rho) + 2\ell s(G_i+G_{2\ell+i}+G_{4\ell+i}+\cdots)}]^{1/{2\ell}}\label{carpim}
\end{align}
holds by H\"{o}lder's inequality under $e^{\langle\rho,X_n\rangle-n\log\phi_a(\rho)}dP_0$.

For $0\le i\le 2\ell-1$, let $c=c(i)$ be the largest integer such that $2c\ell+i\le n-1$. Since $g\in b\cA_{-\ell,\ell}^{0,\ell}$, its shifted versions $G_i, G_{2\ell+i},\ldots,G_{2(c-1)\ell+i}, G_{2c\ell+i}$ are independent under $e^{\langle\rho,X_n\rangle-n\log\phi_a(\rho)}dP_0$. For $n\geq4\ell$,
\begin{align*}
&E_0[e^{\langle\rho,X_n\rangle-n\log\phi_a(\rho) + 2\ell s(G_i+G_{2\ell+i}+\cdots+G_{2(c-1)\ell+i}+G_{2c\ell+i})}]\\
&\quad = E_0[e^{\langle\rho,X_{\ell+i}\rangle-(\ell+i)\log\phi_a(\rho) + 2\ell sG_i}]\left(E_0[e^{\langle\rho,X_{2\ell}\rangle-2\ell\log\phi_a(\rho) + 2\ell sG_\ell}]\right)^{c-1}\\ &\qquad\times E_0[e^{\langle\rho,X_{n-(2c-1)\ell-i}\rangle-(n-(2c-1)\ell-i)\log\phi_a(\rho) + 2\ell sG_\ell}].
\end{align*}
The boundedness of $f$ (and, hence, of $g$) allows us to control the first and last expectations. Therefore,
$$\lim_{n\rightarrow\infty}\frac{1}{n}\log E_0[e^{\langle\rho,X_n\rangle-n\log\phi_a(\rho) + 2\ell s(G_i+G_{2\ell+i}+\cdots+G_{2(c-1)\ell+i}+G_{2c\ell+i})}] = \frac{1}{2\ell}\log E_0[e^{\langle\rho,X_{2\ell}\rangle-2\ell\log\phi_a(\rho) + 2\ell sG_\ell}].$$
Recalling (\ref{carpim}), we deduce the following inequality:
$$\limsup_{n\rightarrow\infty}\frac{1}{n}\log E_0[e^{\langle\rho,X_n\rangle-n\log\phi_a(\rho) + ns\langle g,L_n^\infty\rangle}] \leq \frac{1}{2\ell}\log E_0[e^{\langle\rho,X_{2\ell}\rangle-2\ell\log\phi_a(\rho) + 2\ell sG_\ell}] =: \zeta(s).$$
Note that $\zeta(0)=0$ and
$$\zeta'(0) = E_0[e^{\langle\rho,X_{2\ell}\rangle-2\ell\log\phi_a(\rho)}G_\ell] = \int gd\mu^\xi = \int fd\mu^\xi - \int fd\mu^\xi = 0.$$
Therefore, $\zeta(s)=o(s)$ as $s\to0$, and it follows from (\ref{cebi}) that
$$\limsup_{\delta\to0}\limsup_{n\rightarrow\infty}\frac{1}{n}\log P_0(\langle g,L_n^\infty\rangle>\e\,|\, |\frac{X_n}{n}-\xi|<\delta) < 0.$$
Combining this inequality with the analogous one for $-f$ (and, hence, $-g$), we obtain the desired result.
\end{proof}

\begin{proof}[Proof of Theorem \ref{thm:uniquemin}]   We checked  in Proposition \ref{prop:elem2}(e)  that $\mu^\xi$ is a minimizer of (\ref{a_contraction}).  It remains to rule out other  minimizers.   
For every $\nu\in\cM_1(\Ospace_\bN)$ such that $\nu\neq\mu^\xi$ and $E^\nu[Z_1] = \xi$, there exist $\e>0$, $\ell\in\bN$ and $f\in b\cA_{-\ell,\ell}^{0,\ell}$ such that
$$\nu \in \{\mu\in\cM_1(\Ospace_\bN):\,|\int fd\mu - \int fd\mu^\xi| > \e, |E^\mu[Z_1] - \xi| < \delta\}$$
for every $\delta>0$, which is an open set. Therefore,
\begin{align*}
-I_{3,a}(\nu) &\le \liminf_{n\rightarrow\infty}\frac{1}{n}\log P_0\Bigl(|\int f\,dL_n^\infty-\int fd\mu^\xi|>\e\,,|\frac{X_n}{n}-\xi|<\delta\Bigr)\\
&\le \limsup_{n\rightarrow\infty}\frac{1}{n}\log P_0\Bigl(|\int f\,dL_n^\infty-\int fd\mu^\xi|>\e\,\Big\vert\,|\frac{X_n}{n}-\xi|<\delta\Bigr) + \limsup_{n\rightarrow\infty}\frac{1}{n}\log P_0(|\frac{X_n}{n}-\xi|\le\delta)\\
&\le \limsup_{n\rightarrow\infty}\frac{1}{n}\log P_0\Bigl(|\int f\,dL_n^\infty-\int fd\mu^\xi|>\e\,\Big\vert\,|\frac{X_n}{n}-\xi|<\delta\Bigr) - \inf\{I_{1,a}(\xi'):\,|\xi' - \xi| \le\delta\}
\end{align*}
by the level-3 and level-1 averaged LDPs. Limit superior as $\delta\to0$ gives 
$-I_{3,a}(\nu) < -I_{1,a}(\xi)$ by Proposition \ref{averagedconditioning}. 
Thus  $\nu$ cannot be a minimizer.  
\end{proof}

\begin{proof}[Proof of Proposition \ref{markova}]   Fix $\xi\in\mathrm{ri}(\cD)$ and $\rho\in\partial I_{1,a}(\xi)$.   Let 
\[  
\alpha_n(\w,z)=\frac{ E_0^{\w}[e^{\langle\rho,X_{n}\rangle },Z_{1} = z] }{ E_0^{\w}[e^{\langle\rho,X_{n}\rangle }]} .  \]
Let $m,n\ge 1$ and $j\ge 0$. The calculation below shows that 
\be\label{a9}   \alpha_n(T_{j,x_j}\w,z)  = \bar\mu^\xi( Z_{j+1}=z\,\vert\,\cA_{-m,j+n}^{-m,j})(\w,z_{-m+1,j}). \ee 
Take a test function   $f\in b\cA_{-m,j+n}^{-m,j}$.   Then, by $S$-invariance of $\bar\mu^\xi$, the definition \eqref{sutesti} of $\mu^\xi$,  and two uses of the Markov property of the quenched walk, 
\begin{align*}
&\int f(\w, Z_{-m+1,j})\,\alpha_n(T_{j,X_j}\w,z) \,d\bar\mu^\xi
= 
\int f(T_{m,X_m} \w, Z_{1,m+j})\,\alpha_n(T_{m+j,X_{m+j}}\w,z) \,d\mu^\xi \\[3pt]
&=E_0\bigl[ f(T_{m,X_m} \w, Z_{1,m+j})\,\alpha_n(T_{m+j,X_{m+j}}\w,z) \, 
e^{\langle\rho, X_{m+j+n}\rangle - (m+j+n)\log\phi_a(\rho)} \bigr] \\[3pt]
&=E_0\Bigl[ f(T_{m,X_m} \w, Z_{1,m+j})\, 
e^{\langle\rho, X_{m+j}\rangle - (m+j)\log\phi_a(\rho)}\,
\alpha_n(T_{m+j,X_{m+j}}\w,z) \,E^{T_{m+j,X_m+j} \w}_0\bigl[ e^{\langle\rho, X_{n}\rangle - n\log\phi_a(\rho)}\bigr]   \Bigr] \\[3pt]
&=E_0\Bigl[ f(T_{m,X_m} \w, Z_{1,m+j})\, 
e^{\langle\rho, X_{m+j}\rangle - (m+j)\log\phi_a(\rho)}\,
E^{T_{m+j,X_{m+j}} \w}_0\bigl[ e^{\langle\rho, X_{n}\rangle - n\log\phi_a(\rho)},Z_{1} = z\bigr]   \Bigr] \\[3pt]
&=E_0\bigl[ f(T_{m,X_m} \w, Z_{1,m+j})\,\one\{Z_{m+j+1}=z\}  \, 
e^{\langle\rho, X_{m+j+n}\rangle - (m+j+n)\log\phi_a(\rho)} \bigr] \\[3pt]
&=\int f(T_{m,X_m} \w, Z_{1,m+j})\,\one\{Z_{m+j+1}=z\}  \,d\mu^\xi 
=\int f(\w, Z_{-m+1,j})\,\one\{Z_{j+1}=z\}  \,d\bar\mu^\xi. 
\end{align*}  
This verifies \eqref{a9}.  
Let $m\to\infty$ in \eqref{a9}.  Martingale convergence yields  
\be\label{a10} \nn\begin{aligned}
  \alpha_n(T_{j,x_j}\w,z)  &= \bar\mu^\xi( Z_{j+1}=z\,\vert\,\cA_{-\infty,j+n}^{-\infty,j})(\w,z_{-\infty,j}). 
  \end{aligned}  \ee 
For the case $j=0$, by the $\kS_{0,n}$-measurability of $\alpha_n(\cdot, z)$, 
\be\label{a12}  \nn\begin{aligned}
  \alpha_n(\w,z)  &= \bar\mu^\xi( Z_{1}=z\,\vert\,\cA_{-\infty,j+n}^{-\infty,0})(\w,z_{-\infty,0})\\[3pt]
&= \mu^\xi( Z_{1}=z\,\vert\,\kS_{0,n})(\w).  \end{aligned}  \ee 
In the last expression above  $\bar\mu^\xi$ can be replaced with $\mu^\xi$ since the statement does not involve the backward path.  
 Combining the last two displays gives, for $j\ge 0$ and $n\ge 1$,  
  \be\label{a14} \begin{aligned}
  \bar\mu^\xi( Z_{j+1}=z\,\vert\,\cA_{-\infty,j+n}^{-\infty,j})(\w,z_{-\infty,j})  
 = \mu^\xi( Z_{1}=z\,\vert\,\kS_{0,n})(T_{j,x_j}\w).  
  \end{aligned}  \ee  
 As $n\to\infty$ martingale convergence yields  \eqref{yenker-1}.  The remainder  of Proposition \ref{markova} follows from this.  
\end{proof}

For every $\xi\in\mathrm{ri}(\cD)$,
\begin{equation}\label{def:un}
\left.\frac{d\mu_\Omega^\xi}{d\bP}\right|_{\kS_{0,n}}\!\!\!\!\!\!\!\!(\w) = E_0^\w[e^{\langle\rho,X_n\rangle - n\log\phi_a(\rho)}] =: u_n(\rho,\w)
\end{equation}
is a positive martingale on $(\Omega,\kS_{0,\infty},\bP)$. Throughout the paper, we will sometimes suppress $\rho$ and simply write $u_n$ or $u_n(\w)$ whenever it does not lead to any confusion.

\begin{proof}[Proof of Theorem \ref{thm:doobchar}]

\underline{$(i)\implies(ii)$:} Summing both sides of \eqref{doob} over $z\in\cR$, we see that $u\in L^1(\Omega,\kS_{0,\infty},\bP)$ satisfies
$$u(\w) = \sum_{z\in\cR}\pi_{0,1}(0,z\,|\,\w)\frac{e^{\langle \rho,z\rangle}}{\phi_a(\rho)}u(T_{1,z}\w) = \sum_{z\in\cR}E_0^\w[e^{\langle\rho,Z_1\rangle - \log\phi_a(\rho)},Z_1 = z]u(T_{1,z}\w).$$
Iterating this identity $n\ge1$ times, we deduce that
\be\label{ciftaraf}
u(\w) = \sum_{x}u_n(\rho,\w,x)u(T_{n,x}\w),
\ee
where
\be\label{lazimolc}
u_n(\rho,\w,x) = E_0^\w[e^{\langle\rho,X_n\rangle - n\log\phi_a(\rho)},X_n = x].
\ee
Taking the conditional expectation of both sides of \eqref{ciftaraf}, we get
$$\bE[u\,|\,\kS_{0,n}](\w) = \sum_{x}u_n(\rho,\w,x)\bE[u\circ T_{n,x}] = \bE[u]u_n(\rho,\w)$$
by temporal independence and spatial translation invariance. Since $\bP(u>0)=1$, we have $\bE[u]>0$. Therefore, as $n\to\infty$,
$$u_n = \frac{\bE[u\,|\,\kS_{0,n}]}{\bE[u]} \to \frac{u}{\bE[u]}$$
$\bP$-a.s.\ and in $L^1(\Omega,\kS_{0,\infty},\bP)$ (see \cite[Theorem 5.5.6]{Dur10}). We conclude that $\mu_\Omega^\xi\ll\bP$ on $\kS_{0,\infty}$, and $$\left.\frac{d\mu_\Omega^\xi}{d\bP}\right|_{\kS_{0,\infty}}\!\!\! = \frac{u}{\bE[u]}.$$

\vspace{2mm}

\noindent\underline{$(ii)\implies(i)$:} Let $u = \left.\frac{d\mu_\Omega^\xi}{d\bP}\right|_{\kS_{0,\infty}}\!\!\!\!\!\!$. Note that $\bE[u] = 1$ and hence $\bP(u=0)<1$. We will first show that \eqref{newell} implies $\bP(u=0) = 0$.

By martingale convergence, $u_n\to u$ $\bP$-a.s. 
It follows immediately from the Markov property and the definition in \eqref{lazimolc} that
$$u_{m+n}(\rho,\w) = \sum_{x}u_n(\rho,\w,x)u_m(\rho,T_{n,x}\w)$$ for every $m,n\in\bN$. Sending $m\to\infty$, we deduce \eqref{ciftaraf}. In particular,
$$u(\w) \ge u_n(\rho,\w,nz)u(T_{n,nz}\w)$$
for every $z\in\cR$. If \eqref{newell} holds, then $\exists\,z'\in\cR$ such that $u_n(\rho,\w,nz')>0$ for every $n\in\bN$. Therefore,
\be\label{nullinter}
\{\w: u(\w) = 0\}\subset\bigcap_{n=1}^\infty\{\w: u(T_{n,nz'}\w) = 0\}.
\ee
By our temporal independence and spatial translation invariance assumptions, $(T_{n,nz'}\w)_{n\ge1}$ is an $\Omega$-valued stationary and ergodic process under $\bP$. Since $\bP(u=0)<1$, we apply the ergodic theorem and deduce that, for $\bP$-a.e.\ $\w$, there exists an $n\in\bN$ such that $u(T_{n,nz'}\w) > 0$, i.e., the RHS of \eqref{nullinter} is a $\bP$-null set. Consequently, $\bP(u=0) = 0$.

Finally, we derive \eqref{doob}: For $\bP$-a.e.\ $\w$ and every $z\in\cR$, 
\begin{align}
\pi_{0,1}^\xi(0,z\,|\,\w) &= \mu^\xi( Z_1=z\,\vert\, \kS_{0,\infty})(\w) = \lim_{n\to\infty}\mu^\xi( Z_1=z\,\vert\, \kS_{0,n})(\w)\label{measdikkat}\\
& = \lim_{n\to\infty}\frac{E_0^\w[e^{\langle\rho,X_n\rangle - n\log\phi_a(\rho)},Z_1 = z]}{E_0^\w[e^{\langle\rho,X_n\rangle - n\log\phi_a(\rho)}]}\nonumber\\
& = \lim_{n\to\infty}\pi_{0,1}(0,z\,|\,\w)\frac{e^{\langle \rho,z\rangle}}{\phi_a(\rho)}\frac{E_0^{T_{1,z}\w}[e^{\langle\rho,X_{n-1}\rangle - (n-1)\log\phi_a(\rho)}]}{E_0^\w[e^{\langle\rho,X_n\rangle - n\log\phi_a(\rho)}]}\nonumber\\
& = \lim_{n\to\infty}\pi_{0,1}(0,z\,|\,\w)\frac{e^{\langle \rho,z\rangle}}{\phi_a(\rho)}\frac{u_{n-1}(T_{1,z}\w)}{u_n(\w)}\label{cohilg}\\
& = \pi_{0,1}(0,z\,|\,\w)\frac{e^{\langle \rho,z\rangle}}{\phi_a(\rho)}\frac{u(T_{1,z}\w)}{u(\w)}.\nonumber
\end{align}
Note that the second equality in \eqref{measdikkat} follows from martingale convergence under $\mu_\Omega^\xi$ which is mutually absolutely continuous with $\bP$ since $\bP(u>0) = 1$.
\end{proof}

%
%

\section{Modified variational formulas for the quenched rate functions}\label{sec:for_mod}

Fix a sequence $(f_j)_{j\in\bN}$ of test functions $f_j\in b\cA_{-j,\infty}^{0,j}$ that separate $\cM_1(\Ospace_\bN)$ and satisfy $\|f_j\|_\infty = 1$. For every $\mu\in\cM_1(\Ospace_\bN)$ and $\ell\in\bN$, the set
$$G_{\mu,\ell} = \left\{\nu\in\cM_1(\Ospace_\bN):\,|\langle f_j,\mu\rangle - \langle f_j,\nu\rangle| < \ell^{-1}\ \text{for $1\le j\le\ell$}\right\}$$
is a weakly open neighborhood of $\mu$. Note that $\cap_{\ell\in\bN}\,G_{\mu,\ell} = \cap_{\ell\in\bN}\,\overline G_{\mu,\ell} = \{\mu\}$. The following result gives the lower bound in Theorem \ref{3qLDPmod}.
\begin{theorem}\label{altsiniroz}
Assume that
\be\label{azlazimol}
\text{$\bE[|\log\w_{0,0}(z)|] < \infty$ for every $z\in\cR$.}
\ee
Then, for every $\ell\in\bN$ and every $S$-invariant $\mu\in\cM_1(\Ospace_\bN)$ such that $\mu_\Omega\ll\bP$ on $\kS_{0,\infty}$, 
\be\label{uyuiyi}
\liminf_{n\to\infty}\frac1{n}\log P_0^\w(L_n^\infty\in G_{\mu,\ell}) \ge  - H_q(\mu).
\ee
\end{theorem}

\begin{proof}
The proof uses a strategy involving a change-of-measure, Jensen's inequality, and the ergodic theorem, which is standard for obtaining LDP lower bounds for Markov chains, and has been successfully carried out in the context of (undirected) RWRE (see \cite{Ros06, Yil09b, RasSep11} for the level-1,2,3 quenched LDPs). In fact, keeping future applications in mind, the level-3 quenched LDP lower bound was derived in \cite[Section 4]{RasSep11} in full detail and without using the assumption that the walk is undirected. In particular, the lower bound of the LDP in Theorem \ref{3qLDP} is covered by \cite[Section 4]{RasSep11}, which readily implies that \eqref{uyuiyi} holds for every $S$-invariant $\mu\in\cM_1(\Ospace_\bN)$ such that $\mu_\Omega\ll\bP$ (on $\kS$). Therefore, to prove Theorem \ref{altsiniroz}, we need to replace $\kS$ with $\kS_{0,\infty}$. Since the walk is directed in time, this modification requires only two minor changes in the proofs in \cite[Section 4]{RasSep11}. Below we go over the whole argument for the sake of completeness, point out the two differences, and provide references for further details.

\vspace{2mm}

\textit{Step 1.} For every $\ell\in\bN$, denote the marginal of $P_0$ on $\Ospace_\ell = \Omega\times\cR^\ell$ by $P_0^{(\ell)}$. Let $\bar\eta_i = (T_{i,X_i}\w,Z_{i+1,i+\ell})$, $i\ge0$. Then, under $P_0^\w(\,\cdot\,|\,Z_{1,\ell} = z_{1,\ell})$, $(\bar\eta_i)_{i\ge0}$ is a Markov chain with state space $\Ospace_\ell$ and transition kernel
$$\pi^{(\ell)}(S_z^+\eta\,|\,\eta) = \pi_{0,1}(0,z\,|\,T_{\ell,x_\ell}\w).$$
Here and throughout, $\eta = (\w,z_{1,\ell})\in\Ospace_\ell$ 
and $S_z^+:\Ospace_\ell\to\Ospace_\ell:(\w,z_{1,\ell})\mapsto(T_{1,z_1}\w,z_{2,\ell},z)$ for $z\in\cR$.

For every $S$-invariant $\mu\in\cM_1(\Ospace_\bN)$, let $\mu_\Omega$ and $\mu_{\Ospace_\ell}$ denote the marginals of $\mu$ on $\Omega$ and $\Ospace_\ell$, respectively, and define
$$\pi_\mu^{(\ell)}(S_z^+\eta\,|\,\eta) = \mu(Z_{\ell+1} = z\,|\,\bar\eta_0 = \eta)$$
which can be viewed as the transition kernel of a Markov chain with state space $\Ospace_\ell$. Since $\mu$ is $S$-invariant, $\mu_{\Ospace_\ell}$ is an invariant measure for $\pi_\mu^{(\ell)}$. Moreover, if $\mu_\Omega\ll\bP$ on $\kS_{0,\infty}$ and
\be\label{simdiakli}
\text{$\pi_\mu^{(\ell)}(S_z^+\eta\,|\,\eta) > 0$ for $\mu_{\Ospace_\ell}$-a.e.\ $\eta$ and every $z\in\cR$,}
\ee
then $\mu_{\Ospace_\ell}$ is an ergodic invariant measure for $\pi_\mu^{(\ell)}$ and mutually absolutely continuous with $P_0^{(\ell)}$ on $\cA_{0,\infty}^{0,\ell}$. This follows from a minor modification of \cite[Lemma 4.1]{RasSep11} (see Remark \ref{ikigunebiter} below).

\vspace{2mm}

\textit{Step 2.} Let $P_\eta$ (resp.\ $P_\eta^\mu$) stand for the law of the Markov chain $(\bar\eta_i)_{i\ge0}$ with initial state $\bar\eta_0 = \eta$ and transition kernel $\pi^{(\ell)}$ (resp.\ $\pi_\mu^{(\ell)}$). Observe that
\begin{align}
P_0^\w(L_n^\infty\in G_{\mu,\ell}) &= \sum_{z_{1,\ell}\in\cR^\ell}P_0^\w(Z_{1,\ell} = z_{1,\ell})P_0^\w(L_n^\infty\in G_{\mu,\ell}\,|\,Z_{1,\ell} = z_{1,\ell})\nonumber\\
&= \sum_{z_{1,\ell}\in\cR^\ell}P_0^\w(Z_{1,\ell} = z_{1,\ell})P_\eta(L_n^\ell\in G_{\mu,\ell}^{(\ell)})\nonumber\\
&\ge \sum_{z_{1,\ell}\in\cR^\ell}P_0^\w(Z_{1,\ell} = z_{1,\ell})P_\eta(\tilde L_n^\ell\in \tilde G_{\mu,\ell}^{(\ell)})\label{farkibu}
\end{align}
for $n\ge 4\ell^2$, where
\begin{align*}
L_n^\ell &= \frac1{n}\sum_{i=0}^{n-1}\delta_{T_{i,X_i}\w,Z_{i+1,i+\ell}} =  \frac1{n}\sum_{i=0}^{n-1}\delta_{\bar\eta_i}\in\cM_1(\Ospace_\ell),\qquad \tilde L_n^\ell = \frac1{n-\ell}\sum_{i=\ell}^{n-1}\delta_{\bar\eta_i}\in\cM_1(\Ospace_\ell),\\
G_{\mu,\ell}^{(\ell)} &= \left\{\nu\in\cM_1(\Ospace_\ell):\,|\langle f_j,\mu_{\Ospace_\ell}\rangle - \langle f_j,\nu\rangle| < \ell^{-1}\ \text{for $1\le j\le\ell$}\right\}\quad\text{and}\\
\tilde G_{\mu,\ell}^{(\ell)} &= \left\{\nu\in\cM_1(\Ospace_\ell):\,|\langle f_j,\mu_{\Ospace_\ell}\rangle - \langle f_j,\nu\rangle| < (2\ell)^{-1}\ \text{for $1\le j\le\ell$}\right\}.
\end{align*}
The inequality in \eqref{farkibu} is needed in our case because the function $(\w,z_{1,\ell}) = \eta\mapsto P_\eta(\tilde L_n^\ell\in \tilde G_{\mu,\ell}^{(\ell)})$ is $\cA_{0,\infty}^{0,\ell}$-measurable whereas $\eta\mapsto P_\eta(L_n^\ell\in G_{\mu,\ell}^{(\ell)})$ is not.

\vspace{2mm}

\textit{Step 3.} For $P_0^{(\ell)}$-a.e.\ $\eta$, we change the measure from $P_\eta$ to $P_\eta^\mu$, apply Jensen's inequality (with the logarithm function), send $n\to\infty$, use the ergodicity of $\mu_{\Ospace_\ell}$ for $\pi_\mu^{(\ell)}$, and thereby deduce that
\begin{align*}
\liminf_{n\to\infty}\frac1{n}\log P_\eta(\tilde L_n^\ell\in\tilde G_{\mu,\ell}^{(\ell)}) &\ge  \lim_{n\to\infty}\frac1{n}\log P_\eta^\mu(\tilde L_n^\ell\in\tilde G_{\mu,\ell}^{(\ell)}) - H(\mu_{\Ospace_\ell}\times\pi_\mu^{(\ell)}\,|\,\mu_{\Ospace_\ell}\times\pi^{(\ell)})\\
&= - H_{\cA_{-\infty,\infty}^{-\ell, 1}}(\bar\mu_-\times\pi^{\bar\mu}\,\vert\,\bar\mu_-\times\pi) \ge - H_{\cA_{-\infty,\infty}^{-\infty, 1}}(\bar\mu_-\times\pi^{\bar\mu}\,\vert\,\bar\mu_-\times\pi) = - H_q(\mu).
\end{align*}
For further details regarding this step, see \cite[Lemma 4.2]{RasSep11}. The desired bound \eqref{uyuiyi} now follows from \eqref{farkibu}. 

\vspace{2mm}

\textit{Step 4.} If $\mu_\Omega\ll\bP$ on $\kS_{0,\infty}$ but \eqref{simdiakli} fails to hold, we introduce a $\hat\mu\in\cM_1(\Ospace_\bN)$ of the form $\hat\mu(d\w,d\zvec) = \bP(d\w)\otimes p^{\otimes\bN}(d\zvec)$ for some (deterministic) $p\in\cP$ such that $p(z)>0$ for every $z\in\cR$. Note that 
\be\label{muhat7}  H_q(\hat\mu) = \bE\left[\sum_{z\in\cR}p(z)\log\left(\frac{p(z)}{\w_{0,0}(z)}\right)\right]<\infty\ee
 by \eqref{azlazimol}. (In fact, this is the only point in the proof where \eqref{azlazimol} is fully used.) We replace $\mu$ with $\mu_\epsilon  = (1-\epsilon)\mu + \epsilon\hat\mu$ which is an element of the open set $G_{\mu,\ell}$ for sufficiently small $\epsilon>0$. 
Since $\hat\mu$ is $S$-invariant, its marginal $\hat\mu_{\Ospace_\ell}$ on $\Ospace_\ell$ is an invariant measure for the transition kernel
$$\pi_{\hat\mu}^{(\ell)}(S_z^+\eta\,|\,\eta) = \hat\mu(Z_{\ell+1} = z\,|\,\bar\eta_0 = \eta) = p(z) > 0.$$
The $\Ospace_\ell$-marginal $\mu_{\Ospace_\ell}^\epsilon$ of $\mu_\epsilon$ is an invariant measure for a transition kernel $\pi_{\mu_\epsilon}^{(\ell)}$ (suitably defined as a combination of $\pi_{\mu}^{(\ell)}$ and $\pi_{\hat\mu}^{(\ell)}$) which satisfies the analog of \eqref{simdiakli}. For further details regarding this step, see \cite[Proof of the lower bound in Theorem 3.1, page 224]{RasSep11}. Therefore, $\mu_{\Ospace_\ell}^\epsilon$ and $P_0^{(\ell)}$ are mutually absolutely continuous on $\cA_{0,\infty}^{0,\ell}$, and for $P_0^{(\ell)}$-a.e.\ $\eta$,
$$\lim_{n\to\infty}\frac1{n}\log P_\eta(\tilde L_n^\ell\in\tilde G_{\mu,\ell}^{(\ell)}) \ge - H_q(\mu_\epsilon) \ge - (1-\epsilon) H_q(\mu) - \epsilon H_q(\hat\mu).$$
The last inequality follows from the convexity of the relative entropy $H_q$. Finally, we send $\epsilon$ to $0$, and recall \eqref{farkibu} to deduce \eqref{uyuiyi} as in Step 3. 
\end{proof}

\begin{remark}\label{ikigunebiter}
In Step 1 of the proof of Theorem \ref{altsiniroz}, we cited \cite[Lemma 4.1]{RasSep11} which assumes that $\mu_\Omega\ll\bP$ (on $\kS$), which is equivalent to $\mu_{\Ospace_\ell}\ll P_0^{(\ell)}$ on $\cA_{-\infty,\infty}^{0,\ell}$ since $\bP(\w_{0,0}(z)>0) = 1$ for every $z\in\cR$ by \eqref{azlazimol}. In that paper, the mutual absolute continuity of $\mu_{\Ospace_\ell}$ and $P_0^{(\ell)}$
is established by showing that $f = \frac{d\mu_{\Ospace_\ell}}{dP_0^{(\ell)}}$ satisfies
$$\one_{\{f(\w,z_{1,\ell}) > 0\}} \le \one_{\{f(T_{\ell+1,x_\ell +z}\w,\tilde z_{1,\ell}) > 0\}}$$
for $\bP$-a.e.\ $\w$, every $z_{1,\ell},\tilde z_{1,\ell}\in\cR^\ell$ and $z\in\cR$, and then using the ergodicity of $\bP$ under $(T_{1,z})_{z\in\cR}$ to argue that $P_0^{(\ell)}(f>0) = 1$. Using this, the ergodicity of $\mu_{\Ospace_\ell}$ for $\pi_\mu^{(\ell)}$ follows from a similar argument.

If $\mu_\Omega\ll\bP$ on $\kS_{0,\infty}$, then we replace $f = \frac{d\mu_{\Ospace_\ell}}{dP_0^{(\ell)}}$ with $g = \left.\frac{d\mu_{\Ospace_\ell}}{dP_0^{(\ell)}}\right|_{\cA_{0,\infty}^{0,\ell}}\!\!\!$ in the proof of \cite[Lemma 4.1]{RasSep11}. This modification causes no complications since the negative environment levels $\wvec_{-\infty,-1}$ do not play any role. For example, the function $(\w,z_{1,\ell})\mapsto g(T_{\ell+1,x_\ell +z}\w,\tilde z_{1,\ell})$ is measurable w.r.t.\ $\cA_{\ell+1,\infty}^{0,\ell}\subset\cA_{0,\infty}^{0,\ell}$ for every $\tilde z_{1,\ell}\in\cR^\ell$ and $z\in\cR$.
\end{remark}

We are now ready to verify the modified variational formula for the level-3 quenched rate function.

\begin{proof}[Proof of Theorem \ref{3qLDPmod}]

It follows immediately from the definitions in \eqref{tanimuymaz} and \eqref{tanimmod} that $H_{q,\bP}^{S,+}(\mu) \le H_{q,\bP}^{S}(\mu)$ for every $\mu\in\cM_1(\Ospace_\bN)$. Therefore,
\be\label{birbdd}
(H_{q,\bP}^{S,+})^{**}(\mu) \le (H_{q,\bP}^{S})^{**}(\mu) = I_{3,q}(\mu).
\ee
On the other hand, Theorem \ref{altsiniroz} and the upper bound in the level-3 quenched LDP (Theorem \ref{3qLDP}) give
$$- H_{q,\bP}^{S,+}(\mu) \le \liminf_{n\to\infty}\frac1{n}\log P_0^\w(L_n^\infty\in G_{\mu,\ell}) \le \limsup_{n\to\infty}\frac1{n}\log P_0^\w(L_n^\infty\in G_{\mu,\ell}) \le -\inf_{\nu\in \overline G_{\mu,\ell}}I_{3,q}(\nu)$$
for every $\ell\in\bN$. Sending $\ell\to\infty$, we get $$I_{3,q}(\mu) \le H_{q,\bP}^{S,+}(\mu)$$
since $I_{3,q}$ is lower semicontinuous and $\cap_{\ell\in\bN}\overline G_{\mu,\ell} = \{\mu\}$, and then deduce that
\be\label{ikibdd}
I_{3,q}(\mu)  = (I_{3,q})^{**}(\mu) \le (H_{q,\bP}^{S,+})^{**}(\mu).
\ee
Finally, we put \eqref{birbdd} and \eqref{ikibdd} together to obtain the desired equality \eqref{level3qmod}.
\end{proof}

\begin{proof}[Proof of Corollary \ref{1qLDPmod}]
For every $\xi\in\cD$, the variational formula
$$I_{1,q}(\xi) = \inf\{(H_{q,\bP}^{S,+})^{**}(\mu):\,\mu\in\cM_1(\Ospace_\bN), E^\mu[Z_1] = \xi\}$$
follows immediately from Theorem \ref{3qLDPmod} by the contraction principle. Define
$$\tilde I_{1,q}(\xi) = \inf\{H_{q,\bP}^{S,+}(\mu):\,\mu\in\cM_1(\Ospace_\bN), E^\mu[Z_1] = \xi\}$$
which is equal to the RHS of \eqref{qvarmod2}.  $\tilde I_{1,q}(\xi)<\infty$ because we  can choose $p$ to have mean $\xi$  in the measure $\hat\mu$ in \eqref{muhat7}.   Since $H_{q,\bP}^{S,+}$ is convex (which readily follows from the convexity of $H_q$), $\tilde I_{1,q}$ is convex on $\cD$ and hence continuous on $\mathrm{ri}(\cD)$. For every $\xi\in\mathrm{ri}(\cD)$,
\begin{align*}
I_{1,q}(\xi) &= \lim_{\delta\to0}\inf\{I_{1,q}(\xi'):\,\xi'\in\cD, |\xi' - \xi| < \delta\}\\
&= \lim_{\delta\to0}\inf\{(H_{q,\bP}^{S,+})^{**}(\mu):\,\mu\in\cM_1(\Ospace_\bN), |E^\mu[Z_1] - \xi| < \delta\}\\
&= \lim_{\delta\to0}\inf\{H_{q,\bP}^{S,+}(\mu):\,\mu\in\cM_1(\Ospace_\bN), |E^\mu[Z_1] - \xi| < \delta\}\\
&= \lim_{\delta\to0}\inf\{\tilde I_{1,q}(\xi'):\,\xi'\in\cD, |\xi' - \xi| < \delta\} = \tilde I_{1,q}(\xi)
\end{align*}
by the fact that $(H_{q,\bP}^{S,+})^{**}$ is the lower semicontinuous regularization of $H_{q,\bP}^{S,+}$ (see \cite[Theorem 4.17]{RasSep15}) and $\{\mu\in\cM_1(\Ospace_\bN):\, |E^\mu[Z_1] - \xi| < \delta\}$ is an open set.
\end{proof}

%
%

\section{Decomposing the level-$3$ averaged rate function}\label{sec:connect}

\begin{proof}[Proof of Theorem \ref{entropy_connection}]
Observe that, for $1\le k\le n$,
$$P_0( Z_k=z\,\vert\,  \cA_{0,n}^{0,k-1}\,)(\wlev_{0,n-1}, z_{1,k-1})= \pi_{k-1,k}(x_{k-1},x_{k-1}+z\,|\,\w) = \pi_{0,1}(0,z\,|\,T_{k-1, x_{k-1}}\w).$$
If $\mu\in\cM_1(\Ospace_\bN)$ is $S$-invariant, then
\begin{align*}
H_{0,n}(\mu\,\vert\,P_0) &= H_{\cA_{0,n}^{0,n}}(\mu\,\vert\,P_0)\\
&= H_{\cA_{0,n}^{0,n-1}}(\mu\,\vert\,P_0) + \int  H\bigl( \mu(Z_n=\cdot \,\vert\, \cA_{0,n}^{0,n-1}\,)\,\big\vert \,\pi_{0,1}(0,\cdot\,|\,T_{n-1, x_{n-1}}\w)\bigr) \, \mu(d\wlev_{0,n-1},\,dz_{1,n-1})\\
&= H_{\cA_{0,n}^{0,n-1}}(\mu\,\vert\,P_0) + \int  H\bigl( \bar\mu(Z_1=\cdot \,\vert\, \cA_{-n+1,1}^{-n+1,0}\,)\,\big\vert \,\pi_{0,1}(0,\cdot\,|\,\w)\bigr) \, \bar\mu(d\wlev_{-n+1,0},\,dz_{-n+2,0})\\
&= H_{\cA_{0,n}^{0,n-1}}(\mu\,\vert\,P_0) + H_{\cA_{-n+1,1}^{-n+1,1}}(\bar\mu_-\times\pi^{\bar\mu}\,\vert\,\bar\mu_-\times\pi)
\end{align*}
by the chain rule for relative entropy.  
We can apply the chain rule repeatedly and thereby successively remove all the $z$-coordinates from the first relative entropy on the RHS. The general step is, for $1\le k\le n-1$, 
\begin{align*}
H_{0,n}(\mu\,\vert\,P_0) &= H_{\cA_{0,n}^{0,k}}(\mu\,\vert\,P_0) + \sum_{j=k+1}^{n} H_{\cA_{-j+1,n-j+1}^{-j+1,1}}(\bar\mu_-\times\pi^{\bar\mu}\,\vert\,\bar\mu_-\times\pi)\\
&=  H_{\cA_{0,n}^{0,k-1}}(\mu\,\vert\,P_0) + \int  H\bigl( \mu(Z_k=\cdot \,\vert\, \cA_{0,n}^{0,k-1}\,)\,\big\vert \,\pi_{0,1}(0,\cdot\,|\,T_{k-1, x_{k-1}}\w)\bigr) \, \mu(d\wlev_{0,n-1},\,dz_{1,k-1})\\
& \qquad  \qquad  \qquad  \qquad\! + \sum_{j=k+1}^{n} H_{\cA_{-j+1,n-j+1}^{-j+1,1}}(\bar\mu_-\times\pi^{\bar\mu}\,\vert\,\bar\mu_-\times\pi)\\
&=  H_{\cA_{0,n}^{0,k-1}}(\mu\,\vert\,P_0) + \int  H\bigl( \bar\mu(Z_1=\cdot \,\vert\, \cA_{-k+1,n-k+1}^{-k+1,0}\,)\,\big\vert \,\pi_{0,1}(0,\cdot\,|\,\w) \bigr) \, \bar\mu(d\wlev_{-k+1,n-k},\,dz_{-k+2,0})\\
& \qquad  \qquad  \qquad  \qquad\! + \sum_{j=k+1}^{n} H_{\cA_{-j+1,n-j+1}^{-j+1,1}}(\bar\mu_-\times\pi^{\bar\mu}\,\vert\,\bar\mu_-\times\pi)\\
&=  H_{\cA_{0,n}^{0,k-1}}(\mu\,\vert\,P_0) + \sum_{j=k}^{n} H_{\cA_{-j+1,n-j+1}^{-j+1,1}}(\bar\mu_-\times\pi^{\bar\mu}\,\vert\,\bar\mu_-\times\pi). 
\end{align*}
When all $z$-coordinates have been removed, we end up with this identity: 
\be\label{ent-4}
H_{0,n}(\mu\,\vert\,P_0) =  H_{\kS_{0,n}}(\mu_\Omega\,\vert\,\P) + \sum_{j=1}^{n} H_{\cA_{-j+1,n-j+1}^{-j+1,1}}(\bar\mu_-\times\pi^{\bar\mu}\,\vert\,\bar\mu_-\times\pi).
\ee

\begin{lemma}\label{yanle}
If $\mu\in\cM_1(\Ospace_\bN)$ is $S$-invariant, then
\be\label{ent-lim3}
\lim_{n\to\infty}\frac1n \sum_{j=1}^{n} H_{\cA_{-j+1,n-j+1}^{-j+1,1}}(\bar\mu_-\times\pi^{\bar\mu}\,\vert\,\bar\mu_-\times\pi) = H_{\cA_{-\infty,\infty}^{-\infty, 1}}(\bar\mu_-\times\pi^{\bar\mu}\,\vert\,\bar\mu_-\times\pi) = H_q(\mu). 
\ee
\end{lemma}

\begin{proof}
The relative entropy on the RHS of \eqref{ent-lim3} is an upper bound on each term in the sum on the LHS. On the other hand,  if simultaneously  $j\nearrow\infty$ and $n-j\nearrow\infty$, then 
$$H_{\cA_{-j+1,n-j+1}^{-j+1,1}}(\bar\mu_-\times\pi^{\bar\mu}\,\vert\,\bar\mu_-\times\pi) \to H_{\cA_{-\infty,\infty}^{-\infty, 1}}(\bar\mu_-\times\pi^{\bar\mu}\,\vert\,\bar\mu_-\times\pi),$$
which implies the desired result.
\end{proof} 

Continuing with the proof of Theorem \ref{entropy_connection}, we have seen in Section \ref{sec:3aLDP} that the specific relative entropy $$h(\mu\,\vert\,P_0) = \lim_{n\to\infty}  \frac1n H_{0,n}(\mu\,\vert\,P_0)$$ exists. In combination with \eqref{ent-4} and Lemma \ref{yanle}, this implies that the limit
$$h_{\kS_{0,\infty}}(\mu_\Omega\,\vert\,\P) = \lim_{n\to\infty}  \frac1n H_{\kS_{0,n}}(\mu_\Omega\,\vert\,\P)$$
exists, and satisfies
$$h(\mu\,\vert\,P_0) = h_{\kS_{0,\infty}}(\mu_\Omega\,\vert\,\P) +H_q(\mu).\qedhere$$
\end{proof}

\begin{proof}[Proof of Corollary \ref{cor:zincir}]
For every $S$-invariant $\mu\in\cM_1(\Ospace_\bN)$, 
$$I_{3,a}(\mu) = h(\mu\,|\,P_0) = h_{\kS_{0,\infty}}(\mu_\Omega\,|\,\mathbb{P}) + H_q(\mu) \ge H_q(\mu)$$
by Theorems \ref{3aLDP} and \ref{entropy_connection}. Moreover, if \eqref{polyell} holds, then from  Theorem  \ref{3qLDPmod} and a basic property of the double convex conjugate   (see  \cite[Proposition 4.10]{RasSep15}), 
$$I_{3,q}(\mu) = (H_{q,\bP}^{S,+})^{**}(\mu) \le H_{q,\bP}^{S,+}(\mu).$$

It remains to show that $I_{3,a}(\mu) \le I_{3,q}(\mu)$. Define
\be\label{3log}
\Lambda_{3,a}(f) = \lim_{n\to\infty}\frac1{n}\log E_0[e^{\langle f,L_n^\infty\rangle}]\quad\text{and}\quad\Lambda_{3,q}(f) = \lim_{n\to\infty}\frac1{n}\log E_0^\w[e^{\langle f,L_n^\infty\rangle}]
\ee
for every 
continuous $f\in b\cA_{-\infty,\infty}^{0,\infty}$ and $\bP$-a.e.\ $\w$.   
By Varadhan's lemma (see, e.g., \cite[Section 3.2]{RasSep15}) these limits exist and are convex conjugates of the rate functions:  $$\Lambda_{3,a}(f) = (I_{3,a})^*(f)\quad\text{and}\quad\Lambda_{3,q}(f) = (I_{3,q})^*(f).$$    Then $I_{3,a}(\mu) \le I_{3,q}(\mu)$ follows from 
$$\Lambda_{3,q}(f) = \bE\left[\lim_{n\to\infty}\frac1{n}\log E_0^\w[e^{\langle f,L_n^\infty\rangle}]\right] = \lim_{n\to\infty}\bE\left[\frac1{n}\log E_0^\w[e^{\langle f,L_n^\infty\rangle}]\right] \le \lim_{n\to\infty}\frac1{n}\log E_0[e^{\langle f,L_n^\infty\rangle}] = \Lambda_{3,a}(f).$$
\end{proof}

\section{Equality of the averaged and quenched rate functions}\label{sec:charac}

Throughout this section, we assume \eqref{polyell} which ensures that the quenched LDPs hold.   Again by Varadhan's lemma the limit   
\be\label{bitsedegitsek}
\Lambda_{1,q}(\rho) = \lim_{n\to\infty}\frac1{n}\log E_0^\omega[e^{\langle\rho,X_n\rangle}] 
\ee
exists  for every $\rho\in\bR^d$ and $\bP$-a.e.\ $\w$, and satisfies
\be\label{lvas}
\Lambda_{1,q}(\rho) = (I_{1,q})^*(\rho) = \sup_{\xi\in\cD}\{\langle\rho,\xi\rangle - I_{1,q}(\xi)\}. 
\ee
  We have seen in Proposition \ref{prop:elem}(b) (and its proof in Appendix \ref{app:elementary}) that $I_{1,a}(\xi) \le I_{1,q}(\xi)$ and $\Lambda_{1,q}(\rho) \le \log\phi_a(\rho)$ for every $\xi\in\cD$ and $\rho\in\bR^d$.

\begin{lemma}\label{lem:denk}
For every $\xi\in\mathrm{ri}(\cD)$ and $\rho \in\partial I_{1,a}(\xi)$,
$$I_{1,a}(\xi) < I_{1,q}(\xi)\quad\text{if and only if}\quad\Lambda_{1,q}(\rho) < \log\phi_a(\rho).$$
\end{lemma}

\begin{proof}
For every $\xi\in\mathrm{ri}(\cD)$ and $\rho \in\partial I_{1,a}(\xi)$, if $\Lambda_{1,q}(\rho) < \log\phi_a(\rho)$, then
$$I_{1,a}(\xi) = \langle\rho,\xi\rangle - \log\phi_a(\rho) <  \langle\rho,\xi\rangle - \Lambda_{1,q}(\rho) \le \sup_{\rho'\in\bR^d}\{\langle\rho',\xi\rangle - \Lambda_{1,q}(\rho')\} = I_{1,q}(\xi).$$
Here, the first equality is shown in \eqref{entemizig} from Appendix \ref{app:subdifferential}, and the last equality follows from the convexity of $I_{1,q}$ (see Proposition \ref{prop:elem}(a)).

Conversely, if $I_{1,a}(\xi) < I_{1,q}(\xi)$, then the continuity of $I_{1,q}$ on $\cD$ (see Proposition \ref{prop:elem}(a)) implies that
\begin{align}
\Lambda_{1,q}(\rho) &= \sup_{\xi'\in\cD}\{\langle\rho,\xi'\rangle - I_{1,q}(\xi')\} = \langle\rho,\xi''\rangle - I_{1,q}(\xi'')\nonumber\\
&\le \langle\rho,\xi''\rangle - I_{1,a}(\xi'') \le \langle\rho,\xi\rangle - I_{1,a}(\xi) = \log\phi_a(\rho)\label{isimversene}
\end{align}
for some $\xi''\in\cD$. If $\xi'' = \xi$, then the first inequality in \eqref{isimversene} is strict; if $\xi'' \neq \xi$, then the second inequality in \eqref{isimversene} is strict by \eqref{entemizig}. 
\end{proof}

Recall $u_n(\rho,\w,x)$ from definition \eqref{lazimolc}. When $\rho$ is understood we can drop it from the notation.  The next theorem 
is adapted from \cite[Theorem 3.3]{ComVar06} which is concerned with upper bounds for the free energy of directed polymers in random environments.
 

\begin{theorem}\label{thm:cv}
For every $\rho\in\bR^d$,
$$\Lambda_{1,q}(\rho) - \log\phi_a(\rho) \le \inf\left\{\frac1{tm}\log \bE\left[\sum_x u_m(\rho,\cdot,x)^t\right]:\, t\in(0,1),\,m\in\bN\right\}.$$
\end{theorem}

\begin{proof}
 It follows from the definition of $u_n(\w) = u_n(\rho,\w)$ in \eqref{def:un} and the Markov property of the quenched walk that
$$u_n(\w) = \sum_{x\in\bZ^d}u_n(\w,x)\qquad\text{and}\qquad u_{n_1+n_2}(\w) = \sum_{x_1,x_2}u_{n_1}(\w,x_1)u_{n_2}(T_{n_1,x_1}\w,x_2-x_1).$$
For every $\rho\in\bR^d$, $t\in(0,1)$, and $m,n\ge1$,
\begin{align*}
&\bE\left[\frac1{n}\log u_{nm}\right] = \bE\left[\frac1{tn}\log (u_{nm})^t\right]\\
&\quad = \bE\left[\frac1{tn}\log\left(\sum_{x_1,\ldots,x_n}u_m(\cdot,x_1)u_{m}(T_{m,x_1}\cdot,x_2-x_1)\cdots u_{m}(T_{(n-1)m,x_{n-1}}\cdot,x_n-x_{n-1}) \right)^t\right]\\
&\quad \le \bE\left[\frac1{tn}\log\left(\sum_{x_1,\ldots,x_n}u_m(\cdot,x_1)^tu_{m}(T_{m,x_1}\cdot,x_2-x_1)^t\cdots u_{m}(T_{(n-1)m,x_{n-1}}\cdot,x_n-x_{n-1})^t\right)\right]\\
&\quad \le \frac1{tn}\log\bE\left[\sum_{x_1,\ldots,x_n}u_m(\cdot,x_1)^tu_{m}(T_{m,x_1}\cdot,x_2-x_1)^t\cdots u_{m}(T_{(n-1)m,x_n-1}\cdot,x_n-x_{n-1})^t\right]\\
&\quad = \frac1{tn}\log\left(\bE\left[\sum_{x}u_m(\cdot,x)^t\right]\right)^n = \frac1{t}\log\bE\left[\sum_{x}u_m(\cdot,x)^t\right]
\end{align*}
by the temporal independence and spatial translation invariance assumptions. Sending $n\to\infty$ and using the bounded convergence theorem, we get
$$m(\Lambda_{1,q}(\rho) - \log\phi_a(\rho)) = \lim_{n\to\infty}\frac1{n}\log u_{nm} =  \lim_{n\to\infty}\bE\left[\frac1{n}\log u_{nm}\right] \le \frac1{t}\log\bE\left[\sum_{x}u_m(\cdot,x)^t\right].\qedhere$$
\end{proof}

\begin{corollary}\label{cor:strict}
For every $\xi\in\mathrm{ri}(\cD)$ and $\rho \in\partial I_{1,a}(\xi)$, if there exist $t\in(0,1)$ and $m\in\bN$ such that $$f(t) := \bE\left[\sum_{x} u_m(\rho,\cdot,x)^t\right]<1,$$
then $I_{1,a}(\xi) < I_{1,q}(\xi)$.
\end{corollary}

\begin{proof}
This follows immediately from Lemma \ref{lem:denk} and Theorem \ref{thm:cv}.
\end{proof}

We need two additional lemmas before giving the proof of Theorem \ref{thm:charac}. In fact, the second one is part of Theorem \ref{thm:miniq}, but we state and prove it separately here to avoid circular reasoning (we will later use Theorem \ref{thm:charac} in the proof of Theorem \ref{thm:miniq}).

\begin{lemma}\label{lem:ent}
Let $\mu$ and $\lambda$ be probability measures with finite relative entropy given by 
\[H(\mu\,\vert\,\lambda)=\sup_g\{ E^\mu[g] - \log E^\lambda[e^g]\}\]
with supremum over bounded measurable functions $g$.  Then for any event $A$, 
\be
\lambda(A) \ge \exp\bigl\{ -\mu(A)^{-1} \bigl( H(\mu\,\vert\,\lambda)+\log 2 \bigr)\bigr\}. \label{entr-ineq}
\ee
\end{lemma}

\begin{proof}
Assume $\mu(A)>0$ for otherwise the inequality is trivially true. Then also $\lambda(A)>0$ because finite entropy implies $\mu\ll\lambda$. Take $g=(-\log\lambda(A))\cdot\one_A$ in the variational formula.  
\end{proof}

\begin{lemma}\label{ondenver}
Assume \eqref{polyell}. If $I_{1,a}(\xi) = I_{1,q}(\xi)$, then
$$I_{1,q}(\xi) = (H_{q,\bP}^{S,+})^{**}(\mu^\xi),$$
and $\mu^\xi$ is the unique minimizer of the quenched contraction \eqref{qvarmod1}.
\end{lemma}

\begin{proof}
If $I_{1,a}(\xi) = I_{1,q}(\xi)$, then for every $\nu\in\cM_1(\Ospace_\bN)$ such that $\nu\neq\mu^\xi$ and $E^\nu[Z_1] = \xi$,
$$I_{1,q}(\xi) = I_{1,a}(\xi) < I_{3,a}(\nu) \le  I_{3,q}(\nu) = (H_{q,\bP}^{S,+})^{**}(\nu)$$
by Theorem \ref{thm:uniquemin} and Corollary \ref{cor:zincir}. Hence, $\nu$ is not a minimizer of the quenched contraction \eqref{qvarmod1}. However, the compactness of $\{\mu\in\cM_1(\Ospace_\bN):\,E^\mu[Z_1] = \xi\}$ and the lower semicontinuity of  $(H_{q,\bP}^{S,+})^{**}$ guarantee that there is a minimizer. This implies the desired result.
\end{proof}

\begin{proof}[Proof of Theorem \ref{thm:charac}]
\underline{$\neg(iv)\implies\neg(i):$} Observe that
\be
\begin{aligned}  H_{\kS_{0,n}}(\mu_\Omega\,\vert\,\P) = \bE[u_n\log u_n] &= \sum_x \bE[u_n(\cdot,x)\log u_n(\cdot,x)] - \bE\biggl[u_n \sum_x\frac{u_n(\cdot,x)}{u_n} \log \frac{u_n(\cdot,x)}{u_n}\biggr]\\
&\le\sum_x \bE[u_n(\cdot,x)\log u_n(\cdot,x)]  -  \bE[u_n]\log\left(\frac1{(cn)^d}\right)\\
&=\sum_x \bE[u_n(\cdot,x)\log u_n(\cdot,x)]  +  d\log(cn). 
\end{aligned}\label{str-10}
\ee
Here, we used the following facts: the entropy $-\sum p_i\log p_i$ of a discrete probability distribution with a finite support is dominated by that of the uniform distribution (with the same support); and $(cn)^d$ is a crude upper bound for the number of distinct endpoints of paths of length $n$ started at the origin (with steps in $\cR$).

If $h_{\kS_{0,\infty}}(\mu_\Omega^\xi\,|\,\mathbb{P}) > 0$, then $H_{\kS_{0,n}}(\mu_\Omega^\xi\,\vert\,\P)$ grows linearly in $n$. From \eqref{str-10} we take the very weak consequence that there exists an $m\ge1$ such that $\sum_x \bE\bigl[u_m(\cdot,x)\log u_m(\cdot,x)\bigr] >0$. The desired result $I_{1,a}(\xi) < I_{1,q}(\xi)$ follows from Corollary \ref{cor:strict} which is applicable since $f(1)=1$ and  \[  f'(1)= \sum_x \bE\bigl[u_m(\cdot,x)\log u_m(\cdot,x)\bigr]>0.\]

\noindent\underline{$(i)\iff(ii):$} If (i) is true, then so is (iv) by the previous part. Therefore,
$$I_{1,q}(\xi) = I_{1,a}(\xi) = I_{3,a}(\mu^\xi) = h_{\kS_{0,\infty}}(\mu_\Omega^\xi\,|\,\mathbb{P}) +H_q(\mu^\xi) = H_q(\mu^\xi)$$
by Proposition \ref{prop:elem2}(e) and Corollary \ref{cor:zincir}, and hence (ii) is true.
Conversely, if (ii) is true, then
$$I_{1,q}(\xi) = H_q(\mu^\xi) \le I_{3,a}(\mu^\xi) = I_{1,a}(\xi) \le I_{1,q}(\xi)$$
by  Corollary \ref{cor:zincir}, Proposition \ref{prop:elem2}(e) and Proposition \ref{prop:elem}(b), and hence (i) is true.

\vspace{5mm}

\noindent\underline{$(i)\iff(iii):$} If (i) is true, then so is (iv) by the first part. Therefore,
$$(H_{q,\bP}^{S,+})^{**}(\mu^\xi) = I_{1,q}(\xi)  = I_{1,a}(\xi) = I_{3,a}(\mu^\xi) = H_q(\mu^\xi) \le (H_{q,\bP}^{S,+})^{**}(\mu^\xi)$$
by Lemma \ref{ondenver}, Proposition \ref{prop:elem2}(e) and Corollary \ref{cor:zincir}, and hence (iii) is true. Conversely, if (iii) is true, then
$$I_{3,q}(\mu^\xi) = I_{3,a}(\mu^\xi) = I_{1,a}(\xi) \le  I_{1,q}(\xi) \le I_{3,q}(\mu^\xi)$$
by Corollary \ref{cor:zincir}, Proposition \ref{prop:elem2}(c,e), Proposition \ref{prop:elem}(b) and \eqref{nasilda}, 
and hence (i) is true.

\vspace{5mm}

\noindent\underline{$\neg(i)\implies\neg(iv):$} Assume \eqref{expell}. Theorem \ref{thm:LiuWat} in Appendix \ref{app:concentration} gives the concentration inequality
$$\bP\bigl( \,\abs{\log u_n -  \bE[\log u_n]}\ge n\e\bigr) \le 2\exp(-cn)$$
with a constant $c=c(\e)>0$. If $I_{1,a}(\xi) < I_{1,q}(\xi)$, then
$$\lim_{n\to\infty}\frac1{n}\bE[\log u_n] = \Lambda_{1,q}(\rho) - \log\phi_a(\rho) < 0$$
by Lemma \ref{lem:denk}, where $\rho \in\partial I_{1,a}(\xi)$. Therefore, there is a $\delta>0$ such that for large enough $n$,
\be
\bP(u_n\ge \tfrac12) = \bP(\log u_n\ge -\log 2) \le \bP\bigl(\log u_n  \ge -n\delta\bigr) \le 2\exp(-cn).\label{str-7}
\ee
On the other hand, 
$$\mu_\Omega^\xi(u_n\ge \tfrac12)  = 1 - \mu_\Omega^\xi(u_n < \tfrac12) = 1- \bE\bigl[u_n\one_{\{u_n< \tfrac12\}}\,\bigr] \ge \tfrac12.$$
Applying Lemma \ref{lem:ent} with $\mu = \mu_\Omega^\xi$, $\lambda = \bP$ and $A = \{u_n\ge\tfrac12\}$ on $\kS_{0,n}$, we see that \eqref{entr-ineq} becomes
\be
\bP(u_n\ge \tfrac12) \ge \exp\bigl\{ -2\bigl(  \bE[u_n\log u_n]  +\log 2 \bigr)\bigr\}.  \label{str-8}
\ee
Combining \eqref{str-7}--\eqref{str-8} shows that $\bE[u_n\log u_n]$ grows linearly in $n$, contradicting $h_{\kS_{0,\infty}}(\mu_\Omega^\xi\,|\,\mathbb{P}) = 0$.
\end{proof}

\begin{proof}[Proof of Corollary \ref{cor:ugurmugur}]
If $\mu\in\cM_1(\Ospace_\bN)$ is $S$-invariant and $\mu_\Omega\ll\bP$ on $\kS_{0,\infty}$, then
$H_{q,\bP}^{S,+}(\mu) = H_q(\mu)$ by definition \eqref{tanimmod}. Therefore,
$I_{3,a}(\mu) = I_{3,q}(\mu) = H_q(\mu)$ by Corollary \ref{cor:zincir}. In fact, the second equality follows directly from the lower semicontinuity of $H_q$:
$$H_{q,\bP}^{S,+}(\mu) = H_q(\mu) \le (H_{q,\bP}^{S,+})^{**}(\mu) \le H_{q,\bP}^{S,+}(\mu).$$

Under uniform ellipticity \eqref{unifell}, if $\mu$ is $S$-invariant and $\mu_\Omega\ll\bP$ on $\kS_{0,\infty}$, then
$$H_q(\mu) \le \sum_{z\in\cR}E^{\mu_\Omega}[ | \log\w_{0,0}(z) | ] \le |\cR||\log c| < \infty.$$
Therefore, $h(\mu\,\vert\,P_0) = I_{3,a}(\mu) = H_q(\mu)$ can be canceled from \eqref{a=kS+q} to give $h_{\kS_{0,\infty}}(\mu_\Omega\,|\,\mathbb{P})=0$.
\end{proof}

\section{Minimizers of the quenched contractions}\label{sec:miniq}

\begin{proof}[Proof of Theorem \ref{thm:miniq}]
Fix an arbitrary $\xi\in\mathrm{ri}(\cD)$.
If $I_{1,a}(\xi) = I_{1,q}(\xi)$, then \eqref{kipcak} follows immediately from Theorem \ref{thm:charac}, and we have already shown in Lemma \ref{ondenver} that $\mu^\xi$ is the unique minimizer of the quenched contraction \eqref{qvarmod1}. This concludes the proof of part (a).

If $\mu_\Omega^\xi\ll\bP$ on $\kS_{0,\infty}$, then recall from the proof of Theorem \ref{thm:doobchar} that $u_n(\w) = E_0^\w[e^{\langle\rho,X_n\rangle - n\log\phi_a(\rho)}]$ converges  to $u(\w) = \left.\frac{d\mu_\Omega^\xi}{d\bP}\right|_{\kS_{0,\infty}}\!\!\!\!\!\!\!\!(\w)$ for $\bP$-a.e.\ $\w$, and $\bP(u>0) = 1$. Therefore,
$$\Lambda_{1,q}(\rho) - \log\phi_a(\rho) = \lim_{n\to\infty}\frac1{n}\log u_n(\w) = \lim_{n\to\infty}\frac1{n}\log u(\w) = 0.$$
We deduce from Lemma \ref{lem:denk} that $I_{1,a}(\xi) = I_{1,q}(\xi)$, and part (a) is applicable. Since $\mu^\xi$ is $S$-invariant and $E^{\mu^\xi}[Z_1] = \xi$ (see Proposition \ref{prop:elem2}(a,c)), it is a minimizer of the quenched contraction \eqref{qvarmod2}.

It remains to show that \eqref{qvarmod2} has no minimizers other than $\mu^\xi$. To this end, consider any $S$-invariant $\nu\in\cM_1(\Ospace_\bN)$ such that $\nu\neq\mu^\xi$, $E^\nu[Z_1] = \xi$, and $\nu_\Omega\ll\bP$ on $\kS_{0,\infty}$. Observe that
$$I_{1,q}(\xi) < (H_{q,\bP}^{S,+})^{**}(\nu) = H_q(\nu)$$
by part (a) and Corollary \ref{cor:ugurmugur}. This concludes the proof of part (b).
\end{proof}

\section{Spatially constant environments}\label{sec:constant}

\begin{proof}[Proof of Proposition \ref{prop:scneq}]
The quenched walk $X_n$ is now a sum of independent steps $Z_i\sim\bar q_{i-1}$, and so,  by the   strong LLN,  for 
$\bP$-a.e.\ $\w$,
$$\Lambda_{1,q}(\rho) = \lim_{n\to\infty}\frac1{n}\log E_0^\w[e^{\langle\rho,X_n\rangle}] = \lim_{n\to\infty}\frac1{n}\sum_{i=0}^{n-1}\log W(\rho,T_{i,0}\w) = \bE[\log W(\rho,\w)].$$
Therefore, the first equality in \eqref{metgel} follows from \eqref{lvas} and the convexity of $I_{1,q}$ (see Proposition \ref{prop:elem}(a)), and the rest from Jensen's inequality and \eqref{level1a}.

Assume \eqref{yoksasic}.  Let $\xi\in\mathrm{ri}(\cD)\setminus\{\llnxi\}$ and  $\rho\in\partial I_{1,a}(\xi)$.  Then   $\rho\notin\partial I_{1,a}(\llnxi)$ by \eqref{entemizig} from Appendix \ref{app:subdifferential}. Consequently by Proposition \ref{abimdeca}  the inequality in  \eqref{yoksasic} holds and gives  
\be\label{jenverdi}
\Lambda_{1,q}(\rho) = \bE[\log W(\rho,\w)] < \log\bE[W(\rho,\w)] = \log\phi_a(\rho) 
\ee
  by Jensen's inequality.  This  implies $I_{1,a}(\xi) < I_{1,q}(\xi)$ by Lemma \ref{lem:denk}.
\end{proof}

\begin{proof}[Proof of Proposition \ref{prop:scmin}]
Recall from Proposition \ref{prop:elem2}(b) that the slab variables $(\slabv_i)_{i\ge0}$ (defined in \eqref{slabvaroglu}) are i.i.d.\ under $\mu^\xi$ for every $\xi\in\mathrm{ri}(\cD)$. Since the environments are spatially constant, the slab variables are simply $(\wvec_i,Z_{i+1})_{i\ge0}$.

By definition  \eqref{def:un}, for every $\rho\in\partial I_{1,a}(\xi)$,   
\be\label{gavarmis}
\frac{d\mu_\Omega^\xi}{d\bP}\bigg\vert_{\kS_{0,n}}\!\!\!\!\!\!\!\!(\w) = u_n(\rho,\w) = E_0^\w[e^{\langle\rho,X_n\rangle - n\log\phi_a(\rho)}] = \prod_{i=0}^{n-1}\frac{W(\rho,T_{i,0}\w)}{\phi_a(\rho)} = \prod_{i=0}^{n-1}u_1(\rho,T_{i,0}\w).
\ee
 Therefore,  by the limit in \eqref{cohilg}, 
\be\label{pi-xi-99}\begin{aligned}\pi_{0,1}^\xi(0,z\,|\,\w) &= \lim_{n\to\infty}\pi_{0,1}(0,z\,|\,\w)\frac{e^{\langle \rho,z\rangle}}{\phi_a(\rho)}\frac{u_{n-1}(\rho,T_{1,0}\w)}{u_n(\rho,\w)} = \pi_{0,1}(0,z\,|\,\w)\frac{e^{\langle \rho,z\rangle}}{\phi_a(\rho)}\frac1{u_1(\rho,\w)}\\
& = \pi_{0,1}(0,z\,|\,\omega)\frac{e^{\langle\rho,z\rangle}}{W(\rho,\w)}.
\end{aligned}\ee 

If $\xi\neq\llnxi$ and \eqref{yoksasic} holds, then for $\bP$-a.e.\ $\w$,
$$\lim_{n\to\infty}\frac1{n}\log u_n(\rho,\w) = \lim_{n\to\infty}\frac1{n}\sum_{i=0}^{n-1}\log u_1(\rho,T_{i,0}\w) = \bE[\log u_1(\rho,\w)] < \log\bE[u_1(\rho,\w)] = 0$$
by the strong LLN and Jensen's inequality. In particular, $u_n(\rho,\w)\to0$ as $n\to\infty$, and $\mu_\Omega^\xi\not\ll\bP$ on $\kS_{0,\infty}$. 

Under the same assumptions, if $\pi_{0,1}^\xi$  satisfied  \eqref{doob} for some $u\in L^1(\Omega,\kS_{0,\infty},\bP)$ such that $\bE[u] = 1$ and $\bP(u>0) = 1$, then comparison with \eqref{pi-xi-99} gives 
$$u(\w) = u_1(\rho,\w)u(T_{1,0}\w).$$   Iterating this identity, we get
$$u(\w) = \prod_{i=0}^{n-1}u_1(\rho,T_{i,0}\w)u(T_{n,0}\w) = u_n(\rho,\w)u(T_{n,0}\w).$$
Therefore, for $\bP$-a.e.\ $\w$, $\bE[u\,|\,\kS_{0,n}](\w) = u_n(\rho,\w) \to u(\w) > 0$ as $n\to\infty$, which is a contradiction. 
\end{proof}

\begin{proof}[Proof of Proposition \ref{prop:scent}]
Equality $h_{\kS_{0,\infty}}(\mu_\Omega^\xi\,|\,\bP) = \bE[u_1(\rho,\w)\log u_1(\rho,\w)]$  comes  from \eqref{gavarmis}.  



Substitute the second-last formula of  \eqref{pi-xi-99}    into  \eqref{isimdustu} and use  the independence of $(\wvec_i, Z_{i+1})_{i\ge0}$ under $\mu^\xi$: 
\begin{align*}
H_q(\mu^\xi) &= E^{\mu^\xi}\left[\log\left(\frac{\pi_{0,1}^\xi(0,Z_1\,|\,\omega)}{\pi_{0,1}(0,Z_1\,|\,\omega)}\right)\right] = E^{\mu^\xi}\left[\log\left(\frac{e^{\langle\rho,Z_1\rangle - \log\phi_a(\rho)}}{u_1(\rho,\w)}\right)\right]\\
&= E^{\mu^\xi}[\langle\rho,Z_1\rangle] - \log\phi_a(\rho) - E^{\mu^\xi}[\log u_1(\rho,\w)] = \langle\rho,\xi\rangle - \log\phi_a(\rho) - \bE[u_1(\rho,\w)\log u_1(\rho,\w)]. 
\end{align*}
The last equality used Proposition \ref{prop:elem2}(c).

If $\xi\neq\llnxi$ and \eqref{yoksasic} holds, then $\bP(u_1(\rho,\w) = 1) < 1$ for every $\rho\in\partial I_{1,a}(\xi)$ while  $\bE[u_1(\rho,\w)]=1$.   Strict convexity of $u\mapsto u\log u$ gives 
$h_{\kS_{0,\infty}}(\mu_\Omega^\xi\,|\,\bP) = \bE[u_1(\rho,\w)\log u_1(\rho,\w)]>0$.  
\end{proof}

\begin{proof}[Proof of Proposition \ref{prop:scq}]
Under the quenched measure $P^\w_0$,  the $\Omega$-marginal of $L^\infty_n$ is now a deterministic measure $ n^{-1}\sum_{i=0}^{n-1}\delta_{T_{i,X_i}\w} =  n^{-1}\sum_{i=0}^{n-1}\delta_{T_{i,0}\w}$ that converges weakly to $\bP$, for $\bP$-a.e.~$\w$.    Hence the rate  $I_{3,q}(\mu)$ must be infinite if $\mu_\Omega\ne\bP$.  

By   Proposition \ref{prop:scmin},  if $\xi\in\mathrm{ri}(\cD)\setminus\{\llnxi\}$ and \eqref{yoksasic} holds, then $\mu_\Omega^\xi\not\ll\bP$ on $\kS_{0,\infty}$. Therefore, $H_{q,\bP}^{S,+}(\mu^\xi) = \infty$ by definition (see \eqref{tanimmod}), and  
$I_{3,q}(\mu^\xi) = (H_{q,\bP}^{S,+})^{**}(\mu^\xi) = \infty$
by \eqref{level3qmod} and the paragraph above.
\end{proof}

\begin{proof}[Proof of Proposition \ref{prop:scminq}]
Fix an arbitrary $\xi\in\mathrm{ri}(\cD)$.
\begin{itemize}
\item [(a)] We prove in Theorem \ref{epencer}(b) in Appendix \ref{app:subdifferential} that $\langle \rho,z\rangle - \log W(\rho,\w) = \langle \rho',z\rangle - \log W(\rho',\w)$ for every $\rho,\rho'\in\partial I_{1,q}(\xi)$, $z\in\cR$ and $\bP$-a.e.\ $\w$. Therefore, the RHS of \eqref{sutestiqqq} is well-defined, and so is $\nu^\xi$ by consistency. Taking the POV of the particle, $\nu^\xi$ induces a Markov chain on $\Omega$ with transition kernel $\bar\pi^{\bar\nu^\xi}(\w'|\,\w) = \one_{\{T_{1,0}\w\}}(\w')$, for which the $\Omega$-marginal $\nu_\Omega^\xi = \bP$ of $\nu^\xi$ is an invariant measure. Therefore, $\nu^\xi$ is $S$-invariant.

\item[(b)] Recall \eqref{entemizih} and \eqref{tekmisk} from Appendix \ref{app:subdifferential} and observe that
$$E^{\nu^\xi}[Z_1] = \bE\left[\sum_{z\in\cR}\bar q_0(z)\frac{e^{\langle\rho,z\rangle}z}{W(\rho,\w)}\right] = \bE\left[\frac{E_0^\w[e^{\langle\rho,Z_1\rangle}Z_1]}{E_0^\w[e^{\langle\rho,Z_1\rangle}]}\right] = \bE[\nabla\log W(\rho,\w)] = \nabla\Lambda_{1,q}(\rho) = \xi.$$

\item [(c)] $(H_{q,\bP}^{S,+})^{**}(\nu^\xi) = H_q(\nu^\xi)$ by Corollary \ref{cor:ugurmugur} since $\nu^\xi$ is $S$-invariant and $\nu_\Omega^\xi = \bP$.  
Similar to the proof of Proposition \ref{prop:scent}, 
\begin{align*}
H_q(\nu^\xi) &=  E^{\nu^\xi}\left[\log\left(\frac{\pi_{0,1}^{\bar\nu^\xi}(0,Z_1\,|\,\omega)}{\pi_{0,1}(0,Z_1\,|\,\omega)}\right)\right] = E^{\nu^\xi}\left[\log\left(\frac{e^{\langle\rho,Z_1\rangle}}{W(\rho,\w)}\right)\right]\\
&= E^{\nu^\xi}[\langle\rho,Z_1\rangle] - E^{\nu_\Omega^\xi}[\log W(\rho,\w)] = \langle\rho,\xi\rangle - \bE[\log W(\rho,\w)] = I_{1,q}(\xi),
\end{align*}
where $\rho\in\partial I_{1,q}(\xi)$, and the fourth equality uses part (b). See \eqref{entemizih} 
in Appendix \ref{app:subdifferential} for the last equality.

\item [(d)] We know from part (c) that $\nu^\xi$ is a minimizer of \eqref{qvarmod1} and \eqref{qvarmod2}. 
Take any $S$-invariant $\mu\in\cM_1(\Ospace_\bN)$ such that $E^\mu[Z_1] = \xi$.

(i) If $\mu$ is a minimizer of \eqref{qvarmod1}, then $\mu_\Omega = \bP$ by Proposition \ref{prop:scq}, and hence $(H_{q,\bP}^{S,+})^{**}(\mu) = H_q(\mu)$ by Corollary \ref{cor:ugurmugur}.
\begin{align*}
I_{1,q}(\xi) = H_q(\mu) &= H_{\cA_{-\infty,\infty}^{-\infty, 1}}(\bar\mu_-\times\pi^{\bar\mu}\,\vert\,\bar\mu_-\times\pi)\\
&= \int \bar\mu_-(d\w, \, dz_{-\infty, 0})\sum_{z\in\cR}\pi_{0,1}^{\bar\mu}(0,z\,|\,\omega,z_{-\infty,0})\log\left(\frac{\pi_{0,1}^{\bar\mu}(0,z\,|\,\omega,z_{-\infty,0})}{\pi_{0,1}(0,z\,|\,\omega)}\right)\\
&= \int \bar\mu_-(d\w, \, dz_{-\infty, 0})\sum_{z\in\cR}\pi_{0,1}^{\bar\mu}(0,z\,|\,\omega,z_{-\infty,0})\log\left(\frac{\pi_{0,1}^{\bar\mu}(0,z\,|\,\omega,z_{-\infty,0})}{\pi_{0,1}^{\bar\nu^\xi}(0,z\,|\,\omega)}\right)\\
&\quad + E^\mu[\langle\rho,Z_1\rangle] - E^{\mu_\Omega}[\log W(\rho,\w)]\\
&=  H_{\cA_{-\infty,\infty}^{-\infty, 1}}(\bar\mu_-\times\pi^{\bar\mu}\,\vert\,\bar\mu_-\times\pi^{\bar\nu^\xi}) + \langle\rho,\xi\rangle - \bE[\log W(\rho,\w)]\\
&=  H_{\cA_{-\infty,\infty}^{-\infty, 1}}(\bar\mu_-\times\pi^{\bar\mu}\,\vert\,\bar\mu_-\times\pi^{\bar\nu^\xi}) + I_{1,q}(\xi).
\end{align*}
Therefore, $\pi_{0,1}^{\bar\mu}(0,z\,|\,\omega,z_{-\infty,0}) = \pi_{0,1}^{\bar\nu^\xi}(0,z\,|\,\omega)$ for $\bar\mu_-$-a.e.\ $(\w,z_{-\infty, 0})$ and $z\in\cR$. Since $\mu_\Omega = \nu_\Omega^\xi = \bP$, we conclude that $\mu = \nu^\xi$.

(ii) If $\mu$ is a minimizer of \eqref{qvarmod2}, then $\mu_\Omega \ll \bP$ on $\kS_{0,\infty}$, therefore $(H_{q,\bP}^{S,+})^{**}(\mu) = H_q(\mu)$ by Corollary \ref{cor:ugurmugur}, which implies that $\mu$ is a minimizer of \eqref{qvarmod1}, and the previous part is applicable.
\end{itemize}
\end{proof}

%
%



%
%

\section*{Appendices}

\appendices

\section{Sufficient condition for the level-3 quenched LDP}\label{app:sufficient}


The following definition is adapted from \cite[Section 2]{RasSepYil13} to our specific space-time setting and notation. 
Let $c = \max\{|z|_1:z\in\cR\}$. Here and below, $|\cdot|_1$ denotes the $\ell_1$-norm.
\begin{definition}
A function $g:\Omega\to\bR$ is said to be in class $\cL$ if $g\in L^1(\Omega,\kS,\bP)$ and
\be\label{tektek}
\limsup_{\e\to0}\limsup_{n\to\infty}\max\left\{\frac1{n}\sum_{0\le j\le \e n}|g\circ T_{i+j,x+jz}|:\,(i,x)\in\bZ\times\bZ^d, 0\le i\le n, |x|_1\le ci\right\} = 0\quad\bP\text{-a.s.}
\ee
for every $z\in\cR$. 
\end{definition}

The level-3 quenched LDP we have established in \cite[Sections 3\&4]{RasSepYil13} covers RWDRE subject to the following conditions: (i) $\bP$ is stationary and ergodic under the family of shifts $(T_{1,z})_{z\in\cR}$; and (ii) the function $$\omega\mapsto\log\pi_{0,1}(0,z\,|\,\omega) = \log\omega_{0,0}(z)$$
is in class $\cL$ for every $z\in\cR$. The first condition is satisfied thanks to the temporal independence of the environment. (In fact, $\bP$ is stationary and ergodic under $T_{1,z}$ for each $z\in\cR$.) Therefore, to prove Theorem \ref{3qLDP} (under the ellipticity assumption \eqref{polyell}), it suffices to show the following result.

\begin{proposition}\label{prop:suffcon3q}
If a Borel measurable function $g_0:\cP^{\bZ^d}\to\bR$ satisfies
\be\label{bandana}
\int |g_0(\wlev_0)|^pd\bP_s(\wlev_0) < \infty
\ee 
for some $p>d+1$, then $g:\Omega\to\bR$ defined by
$\w\mapsto g(\w) := g_0(\wlev_0)$ is in class $\cL$.
\end{proposition}

\begin{proof}
Since constant functions are in class $\cL$, we can assume without loss of generality that $\bE[g] = 0$. It suffices to show a modified version of \eqref{tektek}, namely,
\be\label{ciftli}
\limsup_{n\to\infty}\max_{(i,x)\in A_n^\e}\frac1{n}\sum_{0\le j\le 2\e n}|g\circ T_{i+j,x+jz}| = 0\qquad\bP\text{-a.s.}
\ee
for every $\e>0$, where $A_n^\e$ is a thinned out subset of $\{(i,x)\in\bZ\times\bZ^d:\,0\le i\le n, |x|_1\le ci\}$ of size $|A_n^\e| \le C_1n^d\e^{-1}$ with some constant $C_1 = C_1(c,d)$. 

For each $(i,x)\in A_n^\e$, the summands in $\sum_{0\le j\le 2\e n}|g\circ T_{i+j,x+jz}|$ are i.i.d. Therefore, for every $\delta>0$,
$$\bP\left(\sum_{0\le j\le 2\e n}|g\circ T_{i+j,x+jz}|\ge n\delta\right) \le C_2(n\delta)^{-p}\e n$$
by \eqref{bandana} and the Fuk-Nagaev inequality (see \cite[Corollary 1.8]{Nag79}), where $C_2 = C_2(p)$ is some constant and $n$ is sufficiently large (depending on $p,\delta,\e$). Hence,
$$\bP\left(\max_{(i,x)\in A_n^\e}\sum_{0\le j\le 2\e n}|g\circ T_{i+j,x+jz}|\ge n\delta\right) \le C_1n^d\e^{-1}C_2(n\delta)^{-p}\e n = C_1C_2\delta^{-p}n^{d +1-p}$$
by a union bound.

Consider the subsequence $n_m = m^\gamma$ with some $\gamma > (p - d - 1)^{-1}$. Then, $\sum_{m=1}^\infty C_1C_2\delta^{-p}(n_m)^{d +1-p} < \infty$ and
$$\limsup_{m\to\infty}\max_{(i,x)\in A_{n_m}^\e}\frac1{n_m}\sum_{0\le j\le 2\e n_m}|g\circ T_{i+j,x+jz}| \le \delta\qquad\bP\text{-a.s.}$$
by the Borel-Cantelli lemma. This bound generalizes to the full sequence, too, since $\lim_{m\to\infty}\frac{n_{m+1}}{n_m} = 1$. Finally, sending $\delta\to 0$ implies \eqref{ciftli}. 
\end{proof}

\section{Elementary facts regarding the level-1 rate functions}\label{app:elementary}

\begin{proof}[Proof of Proposition \ref{prop:elem}]
(a) $I_{1,a} = (\log\phi_a)^*$ and $I_{3,q} = (H_{q,\bP}^{S})^{**}$ are convex conjugates and hence convex. $I_{1,q}$ is defined in \eqref{nasilda} via contraction, and therefore it is convex, too. Since the rate functions $I_{1,a}$ and $I_{1,q}$ are lower semicontinuous on their domain $\cD$, they are in fact continuous on $\cD$, see \cite[Theorem 10.2]{Roc70}.

(b) Recall from \eqref{bitsedegitsek} that 
$$\Lambda_{1,q}(\rho) := \lim_{n\to\infty}\frac1{n}\log E_0^\omega[e^{\langle\rho,X_n\rangle}]$$
for every $\rho\in\bR^d$. Varadhan's lemma gives $\Lambda_{1,q}(\rho) = (I_{1,q})^*(\rho)$. Observe that
$$\Lambda_{1,q}(\rho) = \bE\left[\lim_{n\to\infty}\frac1{n}\log E_0^\w[e^{\langle\rho,X_n\rangle}]\right] \le \lim_{n\to\infty}\frac1{n}\log E_0[e^{\langle\rho,X_n\rangle}] = \log\phi_a(\rho)$$
by the bounded convergence theorem and Jensen's inequality. Therefore, for every $\xi\in\cD$,
$$I_{1,a}(\xi) = (\log\phi_a)^*(\xi) = \sup_{\rho\in\bR^d}\{\langle\rho,\xi\rangle - \log\phi_a(\rho)\} \le  \sup_{\rho\in\bR^d}\{\langle\rho,\xi\rangle - \Lambda_{1,q}(\rho)\} = (\Lambda_{1,q})^*(\xi) = I_{1,q}(\xi)$$
by \eqref{level1a} and the convexity of $I_{1,q}$.

For every $z\in\cR$, the level-1 quenched LDP upper bound gives
\begin{align*}
-I_{1,q}(z) &\ge \limsup_{n\to\infty}\frac1{n}\log P_0^\w(X_n = nz) \ge \limsup_{n\to\infty}\frac1{n}\log\prod_{i=0}^{n-1}\pi_{i,i+1}(iz,(i+1)z\,|\,\w)\\
&= \lim_{n\to\infty}\frac1{n}\sum_{i=0}^{n-1}\log\pi_{0,1}(0,z\,|\,T_{i,iz}\w) = \bE[\log\w_{0,0}(z)].
\end{align*}
The desired bound follows from \eqref{polyell} and the convexity of $I_{1,q}$.

(c) $I_{1,a}(\llnxi) = I_{1,q}(\llnxi) = 0$ by the LLN. $\log\phi_a$ is analytic on $\bR^d$ and hence $I_{1,a} = (\log\phi_a)^*$ is strictly convex on $\mathrm{ri}(\cD)$. 
Therefore, $0 < I_{1,a}(\xi) \le I_{1,q}(\xi)$ for every $\xi\neq\llnxi$ by part (b), which proves the desired implications.

(d) For every $z\in\cR$, the level-1 averaged LDP upper bound gives
$$-I_{1,a}(z) \ge \limsup_{n\to\infty}\frac1{n}\log P_0(X_n = nz) \ge \lim_{n\to\infty}\frac1{n}\log(\hat q(z))^n = \log\hat q(z) = \log\bE[\w_{0,0}(z)].$$
If $z$ is an extremal point of $\cD$, then for every $\e>0$,
$$-I_{1,q}(z) \le \liminf_{n\to\infty}\frac1{n}\log P_0^\w\left(|\frac{X_n}{n} - z| < \e\right) \le (1 - c\e)\bE[\log\w_{0,0}(z)] + O(\e).$$
Here, the first inequality is an instance of the level-1 quenched LDP lower bound. The second inequality follows from three observations: (i) the event $\left\{|\frac{X_n}{n} - z| < \e\right\}$ consists of $e^{nO(\e)}$ paths, (ii) each path contains at least $(1 - c\e)n$ many $z$-steps for some constant $c = c(\cR)$, and (iii) the probabilities of these z-steps are i.i.d.\ by assumption. Sending $\e\to0$, we deduce that
$$-I_{1,q}(z) \le \bE[\log\w_{0,0}(z)] < \log\bE[\w_{0,0}(z)] \le -I_{1,a}(z)$$
by Jensen's inequality (unless $\w_{0,0}(z)$ is deterministic).
\end{proof}

\section{Subdifferentials of the level-1 rate functions}\label{app:subdifferential}

The convex hull and the affine hull of the finite set $\cR\subset\bZ^d$ are defined as
\begin{align*}
\cD = \mathrm{conv}(\cR) &= \left\{\sum_{z\in\cR}\lambda(z)z:\ \lambda(z)\in[0,1]\ \text{for every $z\in\cR$,}\ \sum_{z\in\cR}\lambda(z) = 1\right\}\ \text{and}\\
M = \mathrm{aff}(\cR) &= \left\{\sum_{z\in\cR}\lambda(z)z:\ \lambda(z)\in\bR\ \text{for every $z\in\cR$,}\ \sum_{z\in\cR}\lambda(z) = 1\right\},
\end{align*}
respectively. The relative interior $\mathrm{ri}(\cD)$    is the  interior of $\cD$  in the relative topology of $M$.

Recall from Appendix \ref{app:elementary} that the functions $I_{1,a}$ and $\log\phi_a$ (resp.\  $I_{1,q}$ and $\Lambda_{1,q}$) are convex conjugates of each other. The subdifferential $\partial I_{1,a}(\xi)$ of $I_{1,a}$ at $\xi\in\cD$ is defined as
\be\label{dd-def} \partial I_{1,a}(\xi) = \{\rho\in\bR^d:\, I_{1,a}(\xi') \ge I_{1,a}(\xi) + \langle\rho,\xi' - \xi\rangle\ \text{for every $\xi'\in\cD$}\}.\ee
$\partial I_{1,q}(\xi)$, $\partial\Lambda_{1,q}(\rho)$ and $\partial\log\phi_a(\rho)$ are defined similarly. Note that $\log\phi_a$ is a smooth function, therefore
$$\partial\log\phi_a(\rho) = \{\nabla\log\phi_a(\rho)\}$$
at every $\rho\in\bR^d$ (see \cite[Theorem 25.1]{Roc70}).

\begin{theorem}\label{epenkur}
If $\xi\in\mathrm{ri}(\cD)$, then $\partial I_{1,a}(\xi)$ and $\partial I_{1,q}(\xi)$ are nonempty and convex. For every $\rho\in\bR^d$,
\be\label{entemizig}
\rho\in\partial I_{1,a}(\xi)\quad\iff\quad I_{1,a}(\xi) + \log\phi_a(\rho) = \langle\rho,\xi\rangle\quad\iff\quad\xi = \nabla\log \phi_a(\rho)
\ee
and
\be\label{entemizih}
\rho\in\partial I_{1,q}(\xi)\quad\iff\quad I_{1,q}(\xi) + \Lambda_{1,q}(\rho) = \langle\rho,\xi\rangle\quad\iff\quad\xi \in\partial\Lambda_{1,q}(\rho).
\ee
\end{theorem}

\begin{proof}
These statements are special instances of \cite[Theorems 23.4 and 23.5]{Roc70}. (Convexity is clear from the definition of subdifferentials.)
\end{proof}

There is a unique linear subspace $L$ of $\bR^d$, given by $L := M - M = \{\xi - \xi':\,\xi,\xi'\in M\}$, that is parallel to $M$, i.e., $M = \xi + L$ for every $\xi\in M$ (see \cite[Theorem 1.2]{Roc70}). Set $\mathrm{dim}(\cD) = \mathrm{dim}(L)$, where $\mathrm{dim}$ denotes dimension. Let $L^\perp$ be the orthogonal complement of $L$ in $\bR^d$.
\begin{theorem}\label{epenlaz}
For every $\xi\in\mathrm{ri}(\cD)$:
\begin{itemize}
\item [(a)] $\partial I_{1,a}(\xi)$ is an affine set that is parallel to $L^\perp$, i.e., $\partial I_{1,a}(\xi) = \rho + L^\perp$ for every $\rho\in\partial I_{1,a}(\xi)$.
\item [(b)] $\langle \rho,z\rangle - \log\phi_a(\rho) = \langle \rho',z\rangle - \log\phi_a(\rho')$ for every $\rho,\rho'\in\partial I_{1,a}(\xi)$ and $z\in\cR$.
\item [(c)]  $\mathrm{dim}(\cD) + \text{dim}(\partial I_{1,a}(\xi)) = d$.
\item [(d)]$I_{1,a}$ is differentiable at $\xi$ if and only if $\mathrm{dim}(\cD) = d$.
\end{itemize}
\end{theorem}

\begin{proof} $ $ 
\begin{itemize}
\item [(a)] 
That $\rho + L^\perp \subset\partial I_{1,a}(\xi)$ for any $\rho\in\partial I_{1,a}(\xi)$ follows immediately from   definition \eqref{dd-def}  because $\langle\rho',\xi' - \xi\rangle=0$ for all $\xi,\xi'\in\cD$ and $\rho'\in L^\perp$.  


Conversely, suppose  $\rho'\notin L^\perp$. Then $\langle\rho',z\rangle$ is not constant over $z\in\cR$, and
$$\langle\rho',\mathbf{J}(\nabla\log \phi_a)(\rho)\rho'\rangle =  E_0[e^{\langle\rho,Z_1\rangle - \log\phi_a(\rho)}\langle\rho',Z_1\rangle^2] - E_0[e^{\langle\rho,Z_1\rangle - \log\phi_a(\rho)}\langle\rho',Z_1\rangle]^2 > 0$$ 
by Jensen's inequality. Here, $\mathbf{J}$ denotes the Jacobian and $\mathbf{J}(\nabla\log \phi_a)(\rho)$ is the Hessian matrix of $\log \phi_a$ at $\rho$. Therefore, $\nabla\log\phi_a(\rho + \epsilon\rho')\ne\xi$ for sufficiently small $\epsilon>0$, and hence $\rho + \epsilon\rho'\not\in\partial I_{1,a}(\xi)$ by \eqref{entemizig}. Since $\partial I_{1,a}(\xi)$ is convex, we deduce that $\rho + \rho'\not\in\partial I_{1,a}(\xi)$. 

\item [(b)] If $\rho,\rho'\in\partial I_{1,a}(\xi)$, then $\rho - \rho'\in L^\perp$ by part (a), 
and so  $\langle\rho-\rho',z'\rangle$ is constant over $z'\in\cR$. Consequently, for any particular   $z\in\cR$, 
\be\label{yerdesti} \begin{aligned}
\langle\rho ,z\rangle - \log\phi_a(\rho) &= \langle\rho- \rho'+\rho',z\rangle - \log\sum_{z'\in\cR}\hat q(z')e^{\langle\rho- \rho' + \rho',z'\rangle}\\
&= \langle\rho',z\rangle - \log\sum_{z'\in\cR}\hat q(z')e^{\langle\rho',z'\rangle} = \langle\rho',z\rangle - \log\phi_a(\rho'). 
\end{aligned}\ee


\item [(c)] $\mathrm{dim}(\cD) + \text{dim}(\partial I_{1,a}(\xi)) = \text{dim}(L) + \mathrm{dim}(L^\perp) = d$ by part (a).

\item [(d)] This follows from part (c) and \cite[Theorem 25.1]{Roc70}.\qedhere

\end{itemize}
\end{proof}

  The next proposition states  some properties of    $\partial I_{1,a}(\llnxi)$ where $\llnxi = \sum_{z\in\cR}\hat q(z)z$ is the LLN velocity.  It is used in conjunction with assumption  \eqref{yoksasic}   for   results on spatially constant environments. 

\begin{proposition}\label{abimdeca}
For every $\rho\in\bR^d$, the following are equivalent:
\begin{itemize}
\item [(i)] $\phi_a(\rho) = e^{\langle\rho,\llnxi\rangle}$;
\item [(ii)] $\rho\in\partial I_{1,a}(\llnxi)$;
\item [(iii)] $\langle\rho,z\rangle = \langle\rho,\llnxi\rangle\ \text{for every}\ z\in\cR$.
\end{itemize}
\end{proposition}

\begin{proof}
For every $\rho\in\bR^d$,
$$\log\phi_a(\rho) = \log\sum_{z\in\cR}\hat q(z)e^{\langle\rho,z\rangle} \ge \sum_{z\in\cR}\hat q(z)\langle\rho,z\rangle = \langle\rho,\llnxi\rangle$$
by Jensen's inequality, and equality holds if and only if $\langle\rho,z\rangle$ is constant over $z\in\cR$. This proves the equivalence of (i) and (iii).

Observe that $\nabla\log\phi_a(0) = \sum_{z\in\cR}\hat q(z)z = \llnxi$. Therefore, $0\in\partial I_{1,a}(\llnxi)$ by \eqref{entemizig}, and $\partial I_{1,a}(\llnxi) = L^\perp$ by Theorem \ref{epenlaz}.
The equivalence of (ii) and (iii) now follows since $\{z-\llnxi:\,z\in\cR\}$ spans $L$.
\end{proof}

When the environment is spatially constant, recall from \eqref{jenverdi} that $\Lambda_{1,q}(\rho) = \bE[\log W(\rho,\w)]$. In particular, it is a smooth function and
\be\label{tekmisk}
\partial\Lambda_{1,q}(\rho) = \{\nabla\Lambda_{1,q}(\rho)\}
\ee at every $\rho\in\bR^d$. 
In this case, the following quenched version of Theorem \ref{epenlaz} holds, with the same proof. 

\begin{theorem}\label{epencer}
Assume \eqref{polyell} and \eqref{def:constenv}. Then, for every $\xi\in\mathrm{ri}(\cD)$:
\begin{itemize}
\item [(a)] $\partial I_{1,q}(\xi)$ is an affine set that is parallel to $L^\perp$, i.e., $\partial I_{1,q}(\xi) = \rho + L^\perp$ for every $\rho\in\partial I_{1,q}(\xi)$.
\item [(b)] $\langle \rho,z\rangle - \log W(\rho,\w) = \langle \rho',z\rangle - \log W(\rho',\w)$ for every $\rho,\rho'\in\partial I_{1,q}(\xi)$, $z\in\cR$ and $\bP$-a.e.\ $\w$.
\item [(c)]  $\mathrm{dim}(\cD) + \text{dim}(\partial I_{1,q}(\xi)) = d$.
\item [(d)]$I_{1,q}$ is differentiable at $\xi$ if and only if $\mathrm{dim}(\cD) = d$.
\end{itemize}
\end{theorem}

\section{A concentration inequality}\label{app:concentration}

Consider a random walk on $\bZ^d$ starting at the origin whose steps are independent and uniformly distributed on $\cR$. Denote the corresponding path measure (resp.\ expectation) by $\hat P_0$ (resp.\ $\hat E_0$). For any $\rho\in\bR^d$, define a function $\eta:\Omega\times\cR\to\bR$ by $$\eta(\w,z) = \langle\rho,z\rangle + \log(|\cR|\pi_{0,1}(0,z\,|\,\w)).$$
With this notation, $$u_n(\w) = E_0^\w[e^{\langle\rho,X_n\rangle - n\log\phi_a(\rho)}] = \hat E_0[e^{\sum_{i=0}^{n-1}(\eta(T_{i,X_i}\w,Z_{i+1}) - \log\phi_a(\rho))}].$$
This representation enables one to study RWDRE via techniques developed in the context of directed polymers (see \cite{ComShiYos04} for a survey). For instance, the following result is an adaptation of a concentration inequality by Liu and Watbled for the quenched free energy of directed polymers (see \cite[Section 6]{LiuWat09}).
\begin{theorem}\label{thm:LiuWat}
Assume \eqref{expell}. Then, for every $\rho\in\bR^d$ and $\e>0$, $\exists\,c = c(\rho,\e)>0$ such that
$$\bP\bigl( \,\abs{\log u_n -  \bE[\log u_n]}\ge n\e\bigr) \le 2\exp(-cn).$$
\end{theorem}

\begin{proof}
We can write $\log u_n - \bE[\log u_n]$ as a sum of $(\kS_{0,i+1})_{0\le i\le n-1}$ martingale differences:
$$\log u_n - \bE[\log u_n] = \sum_{i=0}^{n-1}V_{n,i},\qquad\text{with}\qquad V_{n,i} = \bE_{i+1}[\log u_n] - \bE_{i}[\log u_n],$$
where $\bE_i[\,\cdots]$ is shorthand for $\bE[\,\cdots\,\vert\,\kS_{0,i}]$. 

\begin{lemma}\label{lem:expbd}
For every $0\le i\le n-1$ and $t\in\bR$,
$$\bE_i[\exp(tV_{n,i})]\le K(t) := \begin{cases} \left(\sup_{z\in\cR}\bE[e^{-|t|\eta(0,0,z)}]\right)\left(\sup_{z\in\cR}\bE[e^{\eta(0,0,z)}]^{|t|}\right)\qquad\text{if $|t|<1$};\\ \left(\sup_{z\in\cR}\bE[e^{-|t|\eta(0,0,z)}]\right)\left(\sup_{z\in\cR}\bE[e^{|t|\eta(0,0,z)}]\right)\qquad\text{if $|t|\ge1$}.\end{cases}$$
\end{lemma}

\begin{proof}
Set $$e_{n,i} = \exp\left(\sum_{0\le j\le n-1,\,j\ne i}(\eta(T_{j,X_j}\w,Z_{j+1}) - \log\phi_a(\rho))\right),\qquad u_{n,i} = \hat E_0[e_{n,i}].$$
Since $\bE_{i+1}[\log u_{n,i}] =\bE_{i}[\log u_{n,i}]$, we have
\be
V_{n,i} = \bE_{i+1}\left[\log\frac{u_n}{u_{n,i}}\right] - \bE_{i}\left[\log\frac{u_n}{u_{n,i}}\right].\label{lazimm}
\ee
For every $x\in\bZ^d$ and $z\in\cR$, define
$$\bar\eta(i,x,z) = \exp(\eta(T_{i,x}\w,z) - \log\phi_a(\rho)),\qquad \alpha(i,x,z) = \frac{\hat E_0[e_{n,i}\one_{\{X_i = x, Z_{i+1} = z\}}]}{u_{n,i}}.$$
Then,
$$\sum_{x\in\bZ^d}\sum_{z\in\cR}\alpha(i,x,z) = 1\qquad\text{and}\qquad\frac{u_n}{u_{n,i}} = \sum_{x\in\bZ^d}\sum_{z\in\cR}\alpha(i,x,z)\bar\eta(i,x,z).$$
By \eqref{lazimm}, Jensen's inequality and the fact that $\kS_{0,i}\subset\kS_{0,i+1}$, we get
\begin{align}
\bE_{i}[\exp(tV_{n,i})] &= \exp\left(-t\bE_{i}\left[\log\frac{u_n}{u_{n,i}}\right]\right)\bE_{i}\left[\exp\left(t\bE_{i+1}\left[\log\frac{u_n}{u_{n,i}}\right]\right)\right]\nonumber\\
&\le \bE_{i}\left[\left(\frac{u_n}{u_{n,i}}\right)^{-t}\right]\bE_{i}\left[\left(\frac{u_n}{u_{n,i}}\right)^t\right].\label{fully}
\end{align}
If $t<0$ or $t\ge1$, then the function $u\to u^t$ is convex; therefore Jensen's inequality gives
$$\left(\frac{u_n}{u_{n,i}}\right)^t = \left(\sum_{x\in\bZ^d}\sum_{z\in\cR}\alpha(i,x,z)\bar\eta(i,x,z)\right)^t \le \sum_{x\in\bZ^d}\sum_{z\in\cR}\alpha(i,x,z)(\bar\eta(i,x,z))^t.$$
For every $x\in\bZ^d$ and $z\in\cR$, the random variables $\bar\eta(i,x,z)$ and $\alpha(i,x,z)$ are measurable w.r.t.\ $\sigma\{\w_i\}$ and $\sigma\{\w_j:\,0\le j\le n-1, j\ne i\}$, respectively. Since these two $\sigma$-algebras are independent and the latter one contains $\kS_{0,i}$, we get
$$\bE_{i}[\alpha(i,x,z)(\bar\eta(i,x,z))^t] = \bE_{i}[\alpha(i,x,z)]\bE[(\bar\eta(i,x,z))^t] = \bE_{i}[\alpha(i,x,z)]\frac{\bE[e^{t\eta(0,0,z)}]}{(\phi_a(\rho))^t}.$$
Hence, for $t<0$ or $t\ge1$,
$$\bE_{i}\left[\left(\frac{u_n}{u_{n,i}}\right)^t\right] \le \sup_{z\in\cR}\frac{\bE[e^{t\eta(0,0,z)}]}{(\phi_a(\rho))^t}.$$
If $t\in(0,1)$, then the function $u\to u^t$ is concave; therefore Jensen's inequality gives
$$\bE_{i}\left[\left(\frac{u_n}{u_{n,i}}\right)^t\right] \le \left(\bE_{i}\left[\frac{u_n}{u_{n,i}}\right]\right)^t \le \sup_{z\in\cR}\frac{\left(\bE[e^{\eta(0,0,z)}]\right)^t}{(\phi_a(\rho))^t}.$$
The desired result follows from plugging these bounds in \eqref{fully}.
\end{proof}

Continuing with the proof of Theorem \ref{thm:LiuWat}, recall the ellipticity assumption \eqref{expell}. Lemma \ref{lem:expbd} implies
$$\bE_i[\exp(\delta|V_{n,i}|)]\le \bE_i[\exp(\delta V_{n,i})] + \bE_i[\exp(- \delta V_{n,i})] \le 2K(\delta).$$ Since $\bE[e^{-\delta\eta(\cdot,z)}] = \left(|\cR|e^{\langle\rho,z\rangle}\right)^{-\delta}\bE[\w_{0,0}(z)^{-\delta}]$, we deduce that $K(\delta)<\infty$.
A suitable generalization of the Azuma-Hoeffding inequality (see \cite[Theorem 2.1]{LiuWat09}) gives
$$\bE\left[e^{\delta t(\log u_n - \bE[\log u_n])}\right] \le \exp\left(\frac{2nK(\delta)t^2}{1-t}\right)$$
for every $t\in(0,1)$.
Therefore, 
$$\bP\bigl( \,\abs{\log u_n -  \bE[\log u_n]}\ge n\e\bigr) \le \exp\left(-n\e \delta t + \frac{2nK(\delta)t^2}{1-t}\right)$$
by the exponential Chebyshev inequality. The desired result is obtained by optimizing over $t\in(0,1)$.
\end{proof}

\bibliographystyle{abbrv}
\bibliography{LD_entropy_RWDRE_references}

\begin{thebibliography}{10}

\bibitem{AveHolRed10}
L.~Avena, F.~den Hollander, and F.~Redig.
\newblock Large deviation principle for one-dimensional random walk in dynamic
  random environment: attractive spin-flips and simple symmetric exclusion.
\newblock {\em Markov Process. Related Fields}, 16(1):139--168, 2010.

\bibitem{baxt-jain-sepp}
J.~R. Baxter, N.~C. Jain, and T.~O. Sepp{\"a}l{\"a}inen.
\newblock Large deviations for nonstationary arrays and sequences.
\newblock {\em Illinois J. Math.}, 37(2):302--328, 1993.

\bibitem{BerDreRam14}
N.~Berger, A.~Drewitz, and A.~F. Ram{\'{\i}}rez.
\newblock Effective polynomial ballisticity conditions for random walk in
  random environment.
\newblock {\em Comm. Pure Appl. Math.}, 67(12):1947--1973, 2014.

\bibitem{BolMinPel09}
C.~Boldrighini, R.~A. Minlos, and A.~Pellegrinotti.
\newblock Discrete-time random motion in a continuous random medium.
\newblock {\em Stochastic Process. Appl.}, 119(10):3285--3299, 2009.

\bibitem{BolSzn02}
E.~Bolthausen and A.-S. Sznitman.
\newblock On the static and dynamic points of view for certain random walks in
  random environment.
\newblock {\em Methods Appl. Anal.}, 9(3):345--375, 2002.
\newblock Special issue dedicated to Daniel W. Stroock and Srinivasa S. R.
  Varadhan on the occasion of their 60th birthday.

\bibitem{CDRRS13}
D.~Campos, A.~Drewitz, A.~F. Ram{\'{\i}}rez, F.~Rassoul-Agha, and
  T.~Sepp{\"a}l{\"a}inen.
\newblock Level 1 quenched large deviation principle for random walk in dynamic
  random environment.
\newblock {\em Bull. Inst. Math. Acad. Sin. (N.S.)}, 8(1):1--29, 2013.

\bibitem{come-89}
F.~Comets.
\newblock Large deviation estimates for a conditional probability distribution.
  {A}pplications to random interaction {G}ibbs measures.
\newblock {\em Probab. Theory Related Fields}, 80(3):407--432, 1989.

\bibitem{ComGanZei00}
F.~Comets, N.~Gantert, and O.~Zeitouni.
\newblock Quenched, annealed and functional large deviations for
  one-dimensional random walk in random environment.
\newblock {\em Probab. Theory Related Fields}, 118(1):65--114, 2000.

\bibitem{ComShiYos04}
F.~Comets, T.~Shiga, and N.~Yoshida.
\newblock Probabilistic analysis of directed polymers in a random environment:
  a review.
\newblock In {\em Stochastic analysis on large scale interacting systems},
  volume~39 of {\em Adv. Stud. Pure Math.}, pages 115--142. Math. Soc. Japan,
  Tokyo, 2004.

\bibitem{ComVar06}
F.~Comets and V.~Vargas.
\newblock Majorizing multiplicative cascades for directed polymers in random
  media.
\newblock {\em ALEA Lat. Am. J. Probab. Math. Stat.}, 2:267--277, 2006.

\bibitem{ComYos06}
F.~Comets and N.~Yoshida.
\newblock Directed polymers in random environment are diffusive at weak
  disorder.
\newblock {\em Ann. Probab.}, 34(5):1746--1770, 2006.

\bibitem{DemZei10}
A.~Dembo and O.~Zeitouni.
\newblock {\em Large deviations techniques and applications}, volume~38 of {\em
  Stochastic Modelling and Applied Probability}.
\newblock Springer-Verlag, Berlin, 2010.
\newblock Corrected reprint of the second (1998) edition.

\bibitem{Hol00}
F.~den Hollander.
\newblock {\em Large deviations}, volume~14 of {\em Fields Institute
  Monographs}.
\newblock American Mathematical Society, Providence, RI, 2000.

\bibitem{Hol07}
F.~den Hollander.
\newblock {\em Random polymers}, volume 1974 of {\em Lecture Notes in
  Mathematics}.
\newblock Springer-Verlag, Berlin, 2009.
\newblock Lectures from the 37th Probability Summer School held in Saint-Flour,
  2007.

\bibitem{DeuStr89}
J.-D. Deuschel and D.~W. Stroock.
\newblock {\em Large deviations}, volume 137 of {\em Pure and Applied
  Mathematics}.
\newblock Academic Press, Inc., Boston, MA, 1989.

\bibitem{DonVar75}
M.~D. Donsker and S.~R.~S. Varadhan.
\newblock Asymptotic evaluation of certain {M}arkov process expectations for
  large time. {I}. {II}.
\newblock {\em Comm. Pure Appl. Math.}, 28:1--47; ibid. 28 (1975), 279--301,
  1975.

\bibitem{DonVar76}
M.~D. Donsker and S.~R.~S. Varadhan.
\newblock Asymptotic evaluation of certain {M}arkov process expectations for
  large time. {III}.
\newblock {\em Comm. Pure Appl. Math.}, 29(4):389--461, 1976.

\bibitem{DonVar83}
M.~D. Donsker and S.~R.~S. Varadhan.
\newblock Asymptotic evaluation of certain {M}arkov process expectations for
  large time. {IV}.
\newblock {\em Comm. Pure Appl. Math.}, 36(2):183--212, 1983.

\bibitem{Dur10}
R.~Durrett.
\newblock {\em Probability: theory and examples}.
\newblock Cambridge Series in Statistical and Probabilistic Mathematics.
  Cambridge University Press, Cambridge, fourth edition, 2010.

\bibitem{Ell85}
R.~S. Ellis.
\newblock {\em Entropy, large deviations, and statistical mechanics}, volume
  271 of {\em Grundlehren der Mathematischen Wissenschaften [Fundamental
  Principles of Mathematical Sciences]}.
\newblock Springer-Verlag, New York, 1985.

\bibitem{emra-janj-15}
E.~Emrah and C.~Janjigian.
\newblock Large deviations for some corner growth models with inhomogeneity.
\newblock Available at {\tt arXiv:1509.02234}.

\bibitem{geor-rass-sepp-lppgeo}
N.~Georgiou, F.~Rassoul-Agha, and T.~Sepp{\"a}l{\"a}inen.
\newblock Geodesics and the competition interface for the corner growth model.
\newblock To appear in {\it Probab.~Theory Rel.~Fields}, available at {\tt
  arXiv:1510.00860}.

\bibitem{geor-rass-sepp-lppbuse}
N.~Georgiou, F.~Rassoul-Agha, and T.~Sepp{\"a}l{\"a}inen.
\newblock Stationary cocycles and {B}usemann functions for the corner growth
  model.
\newblock To appear in {\it Probab.~Theory Rel.~Fields}, available at {\tt
  arXiv:1510.00859}.

\bibitem{geor-rass-sepp-var}
N.~Georgiou, F.~Rassoul-Agha, and T.~Sepp{\"a}l{\"a}inen.
\newblock Variational formulas and cocycle solutions for directed polymer and
  percolation models.
\newblock To appear in {\it Comm.\ Math.\ Phys.}, available at {\tt
  arXiv:1311.3016}.

\bibitem{GeoRasSepYil15}
N.~Georgiou, F.~Rassoul-Agha, T.~Sepp{\"a}l{\"a}inen, and A.~Yilmaz.
\newblock Ratios of partition functions for the log-gamma polymer.
\newblock {\em Ann. Probab.}, 43(5):2282--2331, 2015.

\bibitem{JosRas11}
M.~Joseph and F.~Rassoul-Agha.
\newblock Almost sure invariance principle for continuous-space random walk in
  dynamic random environment.
\newblock {\em ALEA Lat. Am. J. Probab. Math. Stat.}, 8:43--57, 2011.

\bibitem{KipVar86}
C.~Kipnis and S.~R.~S. Varadhan.
\newblock Central limit theorem for additive functionals of reversible {M}arkov
  processes and applications to simple exclusions.
\newblock {\em Comm. Math. Phys.}, 104(1):1--19, 1986.

\bibitem{LiuWat09}
Q.~Liu and F.~Watbled.
\newblock Exponential inequalities for martingales and asymptotic properties of
  the free energy of directed polymers in a random environment.
\newblock {\em Stochastic Process. Appl.}, 119(10):3101--3132, 2009.

\bibitem{Nag79}
S.~V. Nagaev.
\newblock Large deviations of sums of independent random variables.
\newblock {\em Ann. Probab.}, 7(5):745--789, 1979.

\bibitem{New95}
C.~M. Newman.
\newblock A surface view of first-passage percolation.
\newblock In {\em Proceedings of the {I}nternational {C}ongress of
  {M}athematicians, {V}ol.\ 1, 2 ({Z}\"urich, 1994)}, pages 1017--1023, Basel,
  1995. Birkh\"auser.

\bibitem{Ras03}
F.~Rassoul-Agha.
\newblock The point of view of the particle on the law of large numbers for
  random walks in a mixing random environment.
\newblock {\em Ann. Probab.}, 31(3):1441--1463, 2003.

\bibitem{Ras04}
F.~Rassoul-Agha.
\newblock Large deviations for random walks in a mixing random environment and
  other (non-{M}arkov) random walks.
\newblock {\em Comm. Pure Appl. Math.}, 57(9):1178--1196, 2004.

\bibitem{RasSep05}
F.~Rassoul-Agha and T.~Sepp{\"a}l{\"a}inen.
\newblock An almost sure invariance principle for random walks in a space-time
  random environment.
\newblock {\em Probab. Theory Related Fields}, 133(3):299--314, 2005.

\bibitem{RasSep09}
F.~Rassoul-Agha and T.~Sepp{\"a}l{\"a}inen.
\newblock Almost sure functional central limit theorem for ballistic random
  walk in random environment.
\newblock {\em Ann. Inst. Henri Poincar\'e Probab. Stat.}, 45(2):373--420,
  2009.

\bibitem{RasSep11}
F.~Rassoul-Agha and T.~Sepp{\"a}l{\"a}inen.
\newblock Process-level quenched large deviations for random walk in random
  environment.
\newblock {\em Ann. Inst. Henri Poincar\'e Probab. Stat.}, 47(1):214--242,
  2011.

\bibitem{rass-sepp-p2p}
F.~Rassoul-Agha and T.~Sepp{\"a}l{\"a}inen.
\newblock Quenched point-to-point free energy for random walks in random
  potentials.
\newblock {\em Probab. Theory Related Fields}, 158(3-4):711--750, 2014.

\bibitem{RasSep15}
F.~Rassoul-Agha and T.~Sepp{\"a}l{\"a}inen.
\newblock {\em A course on large deviations with an introduction to {G}ibbs
  measures}, volume 162 of {\em Graduate Studies in Mathematics}.
\newblock American Mathematical Society, Providence, RI, 2015.

\bibitem{RasSepYil_preprint}
F.~Rassoul-Agha, T.~Sepp{\"a}l{\"a}inen, and A.~Yilmaz.
\newblock Variational formulas and disorder regimes of random walks in random
  potentials.
\newblock To appear in {\it {B}ernoulli}, available at {\tt arXiv:1410.4474}.

\bibitem{RasSepYil13}
F.~Rassoul-Agha, T.~Sepp{\"a}l{\"a}inen, and A.~Yilmaz.
\newblock Quenched free energy and large deviations for random walks in random
  potentials.
\newblock {\em Comm. Pure Appl. Math.}, 66(2):202--244, 2013.

\bibitem{Roc70}
R.~T. Rockafellar.
\newblock {\em Convex analysis}.
\newblock Princeton Mathematical Series, No. 28. Princeton University Press,
  Princeton, N.J., 1970.

\bibitem{Ros06}
J.~M. Rosenbluth.
\newblock {\em Quenched large deviation for multidimensional random walk in
  random environment: {A} variational formula}.
\newblock ProQuest LLC, Ann Arbor, MI, 2006.
\newblock Thesis (Ph.D.)--New York University.

\bibitem{sepp-ptrf-93I}
T.~Sepp{\"a}l{\"a}inen.
\newblock Large deviations for lattice systems. {I}. {P}arametrized independent
  fields.
\newblock {\em Probab. Theory Related Fields}, 96(2):241--260, 1993.

\bibitem{Var03}
S.~R.~S. Varadhan.
\newblock Large deviations for random walks in a random environment.
\newblock {\em Comm. Pure Appl. Math.}, 56(8):1222--1245, 2003.
\newblock Dedicated to the memory of J{\"u}rgen K. Moser.

\bibitem{Yil09a}
A.~Yilmaz.
\newblock Large deviations for random walk in a space-time product environment.
\newblock {\em Ann. Probab.}, 37(1):189--205, 2009.

\bibitem{Yil09b}
A.~Yilmaz.
\newblock Quenched large deviations for random walk in a random environment.
\newblock {\em Comm. Pure Appl. Math.}, 62(8):1033--1075, 2009.

\bibitem{Yil11a}
A.~Yilmaz.
\newblock Harmonic functions, {$h$}-transform and large deviations for random
  walks in random environments in dimensions four and higher.
\newblock {\em Ann. Probab.}, 39(2):471--506, 2011.

\bibitem{YilZei10}
A.~Yilmaz and O.~Zeitouni.
\newblock Differing averaged and quenched large deviations for random walks in
  random environments in dimensions two and three.
\newblock {\em Comm. Math. Phys.}, 300(1):243--271, 2010.

\bibitem{Zei04}
O.~Zeitouni.
\newblock Random walks in random environment.
\newblock In {\em Lectures on probability theory and statistics}, volume 1837
  of {\em Lecture Notes in Math.}, pages 189--312. Springer, Berlin, 2004.

\end{thebibliography}

\end{document}